\documentclass{amsart}	

\usepackage{amsmath}
\usepackage{amssymb}
\usepackage{amsthm}
\usepackage{mathrsfs}
\usepackage[all]{xy} 
\usepackage{bm}
\usepackage{comment}
\usepackage{color}
\usepackage{manfnt}
\usepackage{indentfirst}
\usepackage{caption}
\usepackage{tikz-cd}
\usetikzlibrary{cd}

\usepackage{breqn}
\usepackage{autobreak}
\usepackage{tikz}
    \usetikzlibrary{arrows.meta} 

\usepackage{mathtools}
\mathtoolsset{showonlyrefs=true}

\newtheorem{Theorem}{Theorem}[section]
\newtheorem{Lemma}[Theorem]{Lemma}
\newtheorem{Proposition}[Theorem]{Proposition}
\newtheorem{Corollary}[Theorem]{Corollary}

\newtheorem*{ntheorem}{Theorem}

\theoremstyle{definition}
\newtheorem{definition}[Theorem]{Definition}
\newtheorem{Remark}[Theorem]{Remark}
\newtheorem{example}[Theorem]{Example}

\numberwithin{equation}{section}

\newcommand{\dep}{\operatorname{dep}}

\allowdisplaybreaks[4]

\DeclareMathOperator{\trdeg}{\mathrm{tr\text{.}deg}}

\title[On Taylor coefficients of Anderson-Thakur series]{On algebraic independence of Taylor coefficients of certain Anderson-Thakur series}
\author{Daichi Matsuzuki}
\email{m19044h@math.nagoya-u.ac.jp}
\address{Graduate School of Mathematics, Nagoya University, 
Furo-cho, Chikusa-ku, Nagoya, 464-8602, Japan}
\date{\today}

\begin{document}

\begin{abstract}
In this paper, we study algebraic independence problem for the Taylor coefficients of the Anderson-Thakur series which arise as deformation series of positive characteristic multiple zeta values (abbreviated as MZV's). These Taylor coefficients are simply specializations at $t=\theta$ of hyperderivatives of the Anderson-Thakur series. We consider the prolongations of $t$-motives associated with MZV's, and then determine the dimension of the $t$-motivic Galois groups in question under certain hypotheses. By virtue of Papanikolas' theory, it enables us to obtain the desired algebraic independence result.

\end{abstract}

\maketitle
\tableofcontents

\section{Introduction}

\subsection{Motivation}
In this paper, we mainly investigate an algebraic independence problem for Taylor coefficients of deformation series that arise from multiple zeta values (abbreviated as MZV's) over function fields in positive characteristic. The motivation of our study is as follows. We first let $A$ be the polynomial ring in the variable $\theta$ over the finite field $\mathbb{F}_q$ of $q$ elements in characteristic $p$, an analogue of the ring of integers $\mathbb{Z}$. Inspired by the real MZV's defined as follows:
\begin{equation}
\zeta(s_{1},\ldots,s_{d}):=\sum_{m_1>\cdots>m_d>0} \frac{1}{m_1^{s_1}\cdots m_d^{s_d}} \in \mathbb{R}, (\text{$s_1 \geq 2$ and $s_2,\,\dots,\,s_d \geq 1$}),
\end{equation}
Thakur~\cite{ThakurBook} introduced positive characteristic analogues of MZV's. For a given tuple $\mathbf{s}=(s_1,\,\dots,\,s_d)$ of positive integers, called an \textit{index}, he put
\begin{equation}
    \zeta_A(s_1,\,\dots,\,s_d):=\sum_{\substack{a_1,\,\dots,\,a_d \in A\text{: monic}\\\deg a_1>\cdots >
    \deg a_d \geq0}} \frac{1}{a_1^{s_1}\cdots a_d^{s_d}} \in \mathbb{F}_q((1/\theta)).
\end{equation}
The number $\operatorname{dep}(\mathbf{s}):=d$ and the sum $\operatorname{wt}(\mathbf{s}):=s_1+\cdots+s_d$ are respectively called \textit{depth} and \textit{weight} of the representation $\zeta_A(\mathbf{s})$.
In what follows, Thakur's MZV's are simply referred to as MZV's if no confusion can arise.

Real MZV's have garnered interests of many researchers since they appear in various areas of mathematics such as arithmetic geometry (\cite{Brown2012}, \cite{Deligne2005}, and \cite{Terasoma2002}), low-dimensional topology \cite{Le1995}, and mathematical physics \cite{Broadhurst1995}. 
There are few known transcendence results of MZV's. For example, $\zeta(2n)$ is transcendental because of transcendence of $\pi$ and Euler's formula for each $n \geq 1$
, $\zeta(3)$ is proven to be irrational by Apery, and Zudilin~\cite{Zudilin2001} proved that at least one of $\zeta(5),\,\zeta(7),\,\zeta(9),\,\zeta(11)$ is irrational. 
Considering transcendence and algebraic independence of general MZV's seems to be a difficult problem.

It is well known that Thakur's multiple zeta values have similar properties as those of ones in characteristic $0$. For example,  MZV's form an $\mathbb{F}_q(\theta)$-algebra by Thakur's $q$-shuffle product formula (\cite{Thakur2010}) analogously to the real case, in which real MZV's form a $\mathbb{Q}$-algebra by shuffle or harmonic product.
Carlitz~\cite{Carlitz1935} proved that the value $\zeta_A(n)$ can be written as a product of $\tilde{\pi}^{n}$ and an explicit element of $\mathbb{F}_q(\theta)$ for positive multiple $n$ of $q-1$. Here, $\tilde{\pi}$ is the Carlitz period (see Example \ref{Omega} for the definition), which are proven to be transcendental over $\mathbb{F}_q(\theta)$ by Wade~\cite{Wade1941}.
Based on the transcendence theory developed by Yu~\cite{Yu1997}, the so-called Yu's sub-$t$-module theorem, Anderson-Brownawell-Papanikolas~\cite{Anderson2004}, the so-called ABP-criterion, and Papanikloas~\cite{Papanikolas2008}, there are good developments regarding transcendence, linear and algebraic independence problems for Thakur's MZV's over the years.

For single zeta values (MZV's of depth one which are also called Carlitz zeta values), their transcendence was known by Yu~\cite{Yu1991}, and all algebraic relations was completely determined by Chang and  Yu~\cite{Chang2007} using Papanikolas' theory~\cite{Papanikolas2008}. For higher depth MZV's, Chang~\cite{Chang2014} used ABP-criterion~\cite{Anderson2004} to prove that there do not exist non-trivial $\mathbb{F}_q(\theta)$-linear relations among MZVs with different weights, whence all MZV's are transcendental over $\mathbb{F}_q(\theta)$.

Recently, two significant conjectures on linear independence of MZV's in positive characteristic: Todd's dimension conjecture and Thakur's basis conjecture, were solved independently in \cite{Chang2023} and \cite{Im2022}. The former conjecture predicts the dimension of the $\mathbb{F}_q(\theta)$-linear subspace of $\mathbb{F}_q((1/\theta))$ spanned by MZV's of a fixed weight and is a positive characteristic analogue of Zagier's dimension conjecture for real multiple zeta values. The latter gives an $\mathbb{F}_q(\theta)$-basis of the subspace generated by multiple zeta values of fixed weight and is an analogue of Hoffman's basis conjecture. This means that we now have a description of all $\mathbb{F}_q(\theta)$-linear relations among MZV's.
There are also Mishiba's works \cite{Mishiba2015a} and \cite{Mishiba2015} on algebraic independence of certain families of MZV's.

This paper focuses on the Taylor coefficients of Anderon-Thakur series, which are deformation series of Thakur's MZV's. In their fundamental work~\cite{Anderson1990}, for each index $\bf{s}$, Anderson and Thakur constructed the Anderson-Thakur series $\zeta_{A}^{\mathrm{AT}}({\bf{s}})$, which is a power series in the variable $t$ with coefficients algebraic over $\mathbb{F}_q((1/\theta))$ showed that the constant term of the Taylor expansion of $\zeta_{A}^{\mathrm{AT}}$ at $t=\theta$ gives $\zeta_{A}({\bf{s}})$ (up to an explicit scalar multiple in $\mathbb{F}_q(\theta)$). 
We mention that their Taylor coefficients are also important values as Chang, Green, and Mishiba showed that these Taylor coefficients of Anderson-Thakur series in question have logarithmic illustrations, see \cite{Chang2021a} for details.
The motivation of our study in this paper is to investigate whether the (higher) Taylor coefficients of $\zeta_{A}^{\mathrm{AT}}({\bf{s}})$ are algebraically independent over $\mathbb{F}_q(\theta)$ or not, and our main result, stated as Theorem \ref{MainB}, answers this question under some hypotheses in terms of  $q$ and $p$, the characteristic of the prime field.

\subsection{The statement of the main result}

In order to state our main result, we consider the Taylor expansion
\begin{equation}
    \Omega=\sum \alpha_n (t-\theta)^n
\end{equation}
of $\Omega$ (see Example \ref{Omega} for the definition). 
For each index $\mathbf{s}=(s_1,\,\dots,\,s_r)$, we also consider the associated Anderson-Thakur series $\zeta_A^{\mathrm{AT}}(\mathbf{s})$ given in Definition \ref{DefATseries} and its Taylor expansion
\begin{equation}
     \zeta_A^{\mathrm{AT}}(\mathbf{s})=\sum \alpha_{\mathbf{s},\,n}(t-\theta)^n,
\end{equation}
further, we define the following set:
\begin{equation}
\operatorname{Sub}^\prime(\mathbf{s}):= \{ (s_{i_1},\,s_{i_2},\,\dots,\,s_{i_d}) \mid 1\leq d \leq r,\, 1\leq i_1<\cdots<i_d\leq r \}. \label{subprime}
\end{equation}
The main results of this paper is the following:
\begin{ntheorem}(Theorem \ref{MainB})
   Fix any integer $n \geq 0$. We consider any index $\mathbf{s}=(s_1,\,\dots,\,s_r)\in \mathbb{Z}_{\geq 1}^{r}$ satisfying that $s_1,\,\dots,\,s_r$ are distinct, and $p \nmid s_i,\,(q-1) \nmid s_i$ for $1\leq i\leq r$. Then the field extension over $\overline{\mathbb{F}_q(\theta)}$ generated by
    \begin{equation}
        \{ \alpha_{n^\prime} ,\, \alpha_{\mathbf{s}^\prime,\,n^\prime} \mid 0\leq n^\prime  \leq n,\,\mathbf{s}^\prime \in \operatorname{Sub}^\prime(\mathbf{s}) \}
    \end{equation}
    has transcendental degree $(n+1)\left(\# \operatorname{Sub}^\prime(\mathbf{s}) +1\right)=(n+1)(2^r)$ over $\overline{\mathbb{F}_q(\theta)}$. Equivalently, the set above is algebraically independent over $\overline{\mathbb{F}_q(\theta)}$.
\end{ntheorem}

Note that the spirit of the theorem above asserts that under the given hypothesis there, the union of the first $n+1$ Taylor coefficients of the expansion of $\zeta_{A}^{AT}(\mathbf{s}')$ for $\mathbf{s}'\in \operatorname{Sub}^\prime(\mathbf{s})$ is an algebraically independent set over $\overline{\mathbb{F}_q(\theta)}$.

\begin{Remark}
    In a private discussion with Mishiba, the author was informed that he obtained algebraic independence of the set
    \begin{equation}
        \{\alpha_0,\,\dots,\,\alpha_n\} \cup \{\alpha_{\mathbf{s}^\prime,\,0} \mid \mathbf{s}^\prime\in \operatorname{Sub}^\prime(\mathbf{s}) \}\cup \{\alpha_{(s_j),\,1} \mid 1\leq j \leq r\}
    \end{equation}
    under the same assumption as Theorem \ref{MainB}. Our main theorem is a generalization of his result.
\end{Remark}

\subsection{Outiline of the paper}
We outline the structure of this paper as follows.
Section \ref{sectionPreliminaries} is devoted to recalling Papanikolas' theory, which is the main tools that we use for the proof of Theorem \ref{MainB}. The notions of pre-$t$-motives with their Betti cohomology realizations, and rigid analytically triviality are needed when reviewing Papanikolas' theory, and we recall them in Subsection \ref{subsectionpret}. The notion of periods of rigid analytically trivial pre-$t$-motives is also introduced with giving one concrete example.
In Subsection \ref{subsectionPapanikolas}, we then state the main result of~\cite{Papanikolas2008} in terms of $t$-motivic Galois groups.

In Section \ref{sectionperiodinterpretation}, we pursue period interpretations of the Taylor coefficients considered in Theorem \ref{MainB} in order to apply the machineries in Section \ref{sectionPreliminaries} to our problem.
We recall Carlitz multiple polylogarithms, $t$-motivic Carlitz multiple polylogarithms introduced by Chang~\cite{Chang2014}, and Anderson-Thakur series in Subsection \ref{subsectionATseries}, from which we can obtain period interpretations of multiple zeta values. 
Subsection \ref{subsectionprolongation} is for the review on the theory of prolongations of pre-$t$-motives, which enable us to obtain period interpretations of Taylor coefficients of Anderson-Thakur series.

In Section \ref{sectionfirststep}, we study algebraic independence of MZV's and Taylor coefficients of $\Omega$. Up to multiplication by elements of $\mathbb{F}_q(\theta)$, MZV's coincide with the $0$-th Taylor coefficient of Anderson-Thakur series and the consideration in this section would be the first step of the study of higher Taylor coefficients of the series.

We perform a proof of Theorem \ref{MainB} in Section \ref{sectionproof}. We concretely construct pre-$t$-motives which have Taylor coefficients of Anderson-Thakur series as periods and determine their $t$-motivic Galois groups. The explicit calculations of the  $t$-motivic Galois groups in question prove Theorem \ref{MainB} by virtue of Papanikolas' theory.

\section{Pre-$t$-motives and transcendence of their periods}\label{sectionPreliminaries}

In this section, we review the theory of pre-$t$-motives (\S \ref{subsectionpret}) and Papanikolas' theory on $t$-motivic Galois groups \cite{Papanikolas2008} (\S \ref{subsectionPapanikolas}). 
These are powerful tools in transcendence theory in arithmetic on function fields of positive characteristic and we will use them in the proof of Theorem \ref{MainB}.

Let us fix the notation. We define $A$ to be the polynomial ring $\mathbb{F}_q[\theta]$ and $K$ to be the field $\mathbb{F}_q(\theta)$ of fractions of $A$. Let $K_\infty:=\mathbb{F}_q((1/\theta)))$ be the completion of $K$ with respect to the $\infty$-adic absolute value given by
\begin{equation}
    |a/b|_\infty:=q^{\deg a -\deg b} \quad (a,\,b\in A,\,b \neq0),
\end{equation}
and the completion of a fixed algebraic closure $\overline{K}_\infty$ of $K_\infty$ is denoted by $\mathbb{C}_\infty$. For convenience, we still denote by $|\cdot|_{\infty}$ the extended $\infty$-adic absolute value on $\mathbb{C}_\infty$. 
We let $t$ be a new variable and consider the field $\mathbb{C}_\infty((t))$ of Laurent series. 
The symbol $\mathbb{T}$ denotes the \textit{Tate algebra} over $\mathbb{C}_\infty$ defined as follows:
\begin{equation}
    \mathbb{T}:=\left\{ \sum_{i=m}^\infty a_i t^i \, \middle|\, m \in \mathbb{Z}, \,a_i \in \mathbb{C}_\infty, \,|a_i|_{\infty}\to 0 \text{ for $i \rightarrow \infty$}\right\}.
\end{equation}
The field of fractions of $\mathbb{T}$ is denoted by $\mathbb{L}$. We say a power series 
\begin{equation}
    f=\sum_{i=0}^\infty a_i t^i \in \mathbb{C}_\infty[[t]]
\end{equation}
is entire if we have
\begin{equation}
    \lim_{i \geq \infty}\sqrt[i]{|a_i|_\infty}=0 \text{ and } [K_\infty(a_1,\,a_2,\,\dots):K_\infty]<\infty
\end{equation}  following \cite{Anderson2004}.
An entire power series converges for any $t\in \mathbb{C}_\infty$ and we write $\mathbb{E}$ for the ring of entire power series. 

\subsection{Pre-$t$-motives}\label{subsectionpret}
In this subsection, we recall the notion of pre-$t$-motives, which were introduced in \cite{Papanikolas2008}.
Following \cite{Anderson2004}, for given $n \in \mathbb{Z}$ and a Laurent series 
\begin{equation}
    f=\sum_{i=m}^\infty a_i t^i \in \mathbb{C}_\infty((t)),
\end{equation}
we define its \textit{$n$-fold twist} to be 
\begin{equation}
    f^{(n)}:=\sum_{i=m}^\infty a_i^{q^n} t^i.
\end{equation}
For any matrix $B=(b_{ij})$ with entries in $\mathbb{C}_\infty((t))$, we define $B^{(n)}:=(b_{ij}^{(n)})$.
We further define $\overline{K}(t)[\sigma,\,\sigma^{-1}]$ to be the non-commutative ring of Laurent polynomials over $\overline{K}(t)$ in the variable $\sigma$ subject to the relations $\sigma f=f^{(-1)}\sigma$ for $f \in \overline{K}(t)$. Note that the center of the ring $\overline{K}(t)[\sigma,\,\sigma^{-1}]$ contains $\mathbb{F}_q(t)$.

\begin{definition}[{\cite{Papanikolas2008}}]
A left $\overline{K}(t)[\sigma,\,\sigma^{-1}]$-module is called a \textit{pre-$t$-motive} if it is a finite dimensional vector space over $\overline{K}(t)$. Morphisms of pre-$t$-motives are defined to be left $\overline{K}(t)[\sigma,\,\sigma^{-1}]$-module homomorphisms between pre-$t$-motives.
\end{definition}

By the relation $\sigma f=f^{(-1)} \sigma=f \sigma$, which holds for each $f \in \mathbb{F}_q(t)$, Papanikolas~\cite{Papanikolas2008} deduced that the category $\mathcal{P}$ of pre-$t$-motives has a structure of $\mathbb{F}_q(t)$-linear category. 
He further proved in \cite[Theorem 3.2.13]{Papanikolas2008} that $\mathcal{P}$ is a rigid abelian tensor category over $\mathbb{F}_q(t)$ where the tensor product operation is given as follows. For two pre-$t$-motives $P$ and $P^\prime$, we define $P\otimes P^\prime:=P\otimes_{\overline{K}(t)}P^\prime$, on which $\sigma$ acts diagonally.

\begin{example}
 The \textit{trivial pre-$t$-motive}, which is denoted by $\mathbf{1}$, is the one-dimensional $\overline{K}(t)$-vector space $\overline{K}(t)$ with the $\sigma$-action given by $\sigma f:=f^{(-1)}$ for $f\in \overline{K}(t)$. 
\end{example}
\begin{example}
   The \textit{Carlitz motive} denoted by $C$ is $\overline{K}(t)$ with $\sigma$-action given by $\sigma f:= (t-\theta)f^{(-1)}$ for $f\in \overline{K}(t)$. For $n \geq 1$, the $n$-the tensor power of $C$ is denoted by $C^{\otimes n}:=C\otimes \cdots \otimes C$ ($n$ times). So underlying $\overline{K}(t)$-vector space of $C^{\otimes n}$ is $\overline{K}(t)$ and the $\sigma$-action is given by $\sigma f:= (t-\theta)^n f^{(-1)}$ for $f\in \overline{K}(t)$.
\end{example}
Throughout this paper, for convenience when we say that a pre-$t$-motive $P$ is defined by the matrix $\Phi\in \operatorname{Mat}_{r}(\overline{K}(t))$, it is understood that $P$ is of dimension $r$ over $\overline{K}(t)$, and with respect to a fixed  $\overline{K}(t)$-basis $\{ m_1,\ldots,m_r\}$ we have
$ \sigma {\bf{m}}=\Phi {\bf{m}}, $
where ${\bf{m}}=(m_1,\ldots,m_r)^{tr}$. Note that since $P$ is a left $\overline{K}(t)[\sigma,\sigma^{-1}]$-module, the matrix $\Phi$ defining $P$ above must be invertible.

For a given pre-$t$-motive $P$, we put
\begin{equation}
    P^{\text{B}}:=\{a \in \mathbb{L}\otimes_{\overline{K}(t)}P \mid \sigma(a)=a \}. 
\end{equation}
where $\sigma $-action on $\mathbb{L}\otimes_{\overline{K}(t)}P$ is given by $\sigma(f\otimes m):=f^{(-1)}\otimes \sigma m$ for $f \in \mathbb{L}$ and $m \in P$, following \cite{Papanikolas2008}. 
Note that $P^{\text{B}}$ is an $\mathbb{F}_q(t)$-vector space, and we call $P^{\text{B}}$ the \textit{Betti realization} of $P$.
If the natural map 
\begin{equation}
\mathbb{L}\otimes_{\mathbb{F}_q(t)}P^{\text{B}} \rightarrow \mathbb{L}\otimes_{\overline{K}(t)}P
\end{equation}
is an isomorphism of $\mathbb{L}$-vector spaces, then we say that $P$ is \textit{rigid analytically trivial}. 
It was shown by Papanikolas that we have the following criterion for rigid analytically triviality of pre-$t$-motives, see also \cite{Anderson1986}.

\begin{Proposition}[{\cite[Theorem 3.3.9]{Papanikolas2008}}]
    Suppose that $P$ is a pre-$t$-motive of dimension $r$ over $\overline{K}(t)$ defined by $\Phi \in \operatorname{GL}_r(\overline{K}(t))$. Then $P$ is rigid analytically trivial if and only if there exists $\Psi \in { \rm GL}_r (\mathbb{L})$ such that $\Psi^{(-1)}=\Phi \Psi$.
\end{Proposition}

For a rigid analytically trivial pre-$t$-motive $P$ defined by $\Phi$, the matrix $\Psi$ in the proposition above is called a \textit{rigid analytic trivialization} of $\Phi$. Note that $\Psi^{-1}\mathbf{m}$ forms a $\mathbb{F}_q(t)$-basis of $P^B$. We mention that rigid analytic trivialization of $\Phi$ is not unique. In fact, if $\Psi$ and $\Psi^\prime$ are two rigid analytic trivializations of a matrix $\Phi$, then $\Psi^{-1}\Psi^\prime \in { \rm GL} _r(\mathbb{F}_q(t))$ (\cite[\S 4.1]{Papanikolas2008}). Let us write $\Psi^{-1}=\Theta=(\Theta_{ij})$. If an entry $\Theta_{ij}$ converges at $t=\theta$, then the value $\Theta_{ij}|_{t=\theta}$ is called a \textit{period} of $P$ (cf.~\cite{Papanikolas2008}). Because of the following proposition, the entries of the matrices $\Psi$ we consider in the following context are entire.

\begin{Proposition}[{\cite[Proposition 3.1.3]{Anderson2004}}] 
 Given
        $ \Phi \in \operatorname{Mat}_{r \times r}(\overline{K}[t]),
    $
    suppose that there exists  
    $\psi \in \operatorname{Mat}_{r\times 1}(\mathbb{T})$
    so that \[ \psi^{(-1)}=\Phi \psi.\]
    If  $\det \Phi|_{t=\theta} \neq 0$, then all entries of $\psi$ are entire.
\end{Proposition}

\begin{example}\label{Omega}
    The pre-$t$-motive $\mathbf{1}$ is defined by the matrix $(1) \in  \operatorname{GL}_1(\overline{K}(t))$, which has rigid analytic trivialization $(1) \in  \operatorname{GL}_1(\mathbb{L})$. In order to study a period of the Carlitz motive $C$, we consider the following infinite product
    \begin{equation}
        \Omega(t):= (-\theta)^{\frac{-q}{q-1}}\prod_{i \geq 1} \left( 1-\frac{t}{\theta^{(i)}} \right)\in \mathbb{C}_\infty((t)),
    \end{equation}
    where  $(-\theta)^{\frac{1}{q-1}}$ is a fixed $(q-1)$th root of $-\theta$ following \cite{Anderson2004}. From the definition of $\Omega$, one can show that $\Omega^{(-1)}=(t-\theta)\Omega(t)$ and so $(\Omega) \in  \operatorname{GL}_1(\mathbb{L})$ is a rigid analytic trivialization of the matrix $(t-\theta)\in  \operatorname{GL}_1(\overline{K}(t))$, which defines the Carlitz motive $C$.   Since $\Omega$ is entire on $\mathbb{C}_{\infty}$ with simple zero at $t=\theta^{q^{i}}$ for each $i\in \mathbb{N}$, we see that $\Omega^{-1}$ converges at $t=\theta$. The value  
    \begin{equation}
        \widetilde{\pi}:=\Omega^{-1}|_{t=\theta}=\theta(-\theta)^{\frac{1}{q-1}}\prod_{i \geq 1}\left( 1-\frac{\theta}{\theta^{q^i}} \right)^{-1}
    \end{equation}
is a period of $C$ and known as  a fundamental period of the Carlitz module (\cite{Carlitz1935}). This value is viewed as a positive characteristic analogue of the complex period $2\pi \sqrt{-1}$ and is proven to be transcendental over $K$, such as the classical case, by Wade~(\cite{Wade1941}). 
\end{example}

Papanikolas proved that rigid analytically trivial pre-$t$-motives form a neutral Tannakian category over $\mathbb{F}_q(t)$ as the following theorem claims (for the definition of Tannakian category, we refer the readers to \cite{Deligne1982}). 
We will study periods of rigid analytically trivial pre-$t$-motives via Tannakian duality.

\begin{Theorem}[{\cite[Theorem 3.3.15]{Papanikolas2008}}]
The category $\mathcal{R}$ consisting of all rigid analytically trivial pre-$t$-motive forms a neutral Tannakian category over $\mathbb{F}_q(t)$ with the fiber functor $P\mapsto P^{\text{B}}$.
\end{Theorem}

For a rigid analytically trivial pre-$t$-motive $P$ we denote by $\langle P \rangle $ the Tannakian sub-category of $\mathcal{R}$ generated by $P$ in this paper. By Tannakian duality, there exists an algebraic group $\Gamma_P$ such that $\langle P \rangle $ is equivalent to the category $\operatorname{Rep}_{\mathbb{F}_q(t)}(\Gamma_P)$ of finite dimentional linear representations of $\Gamma_P$ over $\mathbb{F}_{q}(t)$. The algebraic group $\Gamma_P$ is called the \textit{$t$-motivic Galois group} of $P$.

\subsection{Papanikolas' theory on $t$-motivic Galois groups}\label{subsectionPapanikolas}
This subsection is devoted to recalling Papanikolas' theory on relationships between transcendence property of periods and $t$-motivic Galois groups.

Suppose that we have $\Phi\in \operatorname{GL}_r(\overline{K}(t))$ and $\Psi\in \operatorname{GL}_r(\mathbb{L})$ for which $\Psi^{(-1)}=\Phi \Psi$.
 We put $\Psi_1:=(\Psi_{ij}\otimes1) \in \operatorname{GL}_r(\mathbb{L}\otimes_{\overline{K}(t)}\mathbb{L})$, $\Psi_2:=(1\otimes\Psi_{ij})\in \operatorname{GL}_r(\mathbb{L}\otimes_{\overline{K}(t)}\mathbb{L})$, and \begin{equation}\widetilde{\Psi}:=\Psi_1^{-1}\Psi_2. \label{eqtilde}
 \end{equation}
 Let us consider the subscheme
\begin{equation}
    \Gamma_{\Psi}:=\operatorname{Spec}\mathbb{F}_q(t)[\widetilde{\Psi}_{ij},\,1/\det\widetilde{\Psi}] \label{GammaPsi}
\end{equation}
 of $\operatorname{GL}_{r/\mathbb{F}_q(t)}$, which is characterized to be the smallest closed subscheme which has $\widetilde{\Psi}$ as its $\mathbb{L}\otimes_{\overline{K}(t)} \mathbb{L}$-valued point. 
By the following theorem of Papanikolas, the variety $\Gamma_{\Psi}$ is isomorphic to the $t$-motivic Galois group of a pre-$t$-motive $P$ if $\Psi$ is a  rigid analytic trivialization of $\Phi$ defining $P$ and has connection with transcendence theory.
We mention that the proof of equation \eqref{Papanikolasthmtrdeg} highly depends on \textit{ABP-criterion} established in \cite{Anderson2004}.

\begin{Theorem}[{\cite[Proposition 3.1.3]{Anderson2004}, \cite[Theorem 5.2.2]{Papanikolas2008}}]
\label{ThmPapanikolas}
Let us take $\Phi \in \operatorname{Mat}_{ r}(\overline{K}[t])\cap\operatorname{GL}_r(\overline{K}(t))$ which has a rigid analytic trivialization $\Psi$ in $\operatorname{Mat}_{ r}(\mathbb{E})\cap\operatorname{GL}_r(\mathbb{L})$, and let $P$ be the pre-$t$-motive defined by $\Phi$. 
Then the subscheme $\Gamma_\Psi$ of $\operatorname{GL}_{r/\mathbb{F}_q(t)}$ defined in equation \eqref{GammaPsi} is a smooth algebraic subgroup over $\mathbb{F}_q(t)$ and we have an isomorphism
\begin{equation}
    \Gamma_{P} \simeq \Gamma_\Psi \label{EqThemPapanikolas}
\end{equation}
of algebraic groups over $\mathbb{F}_q(t)$. Moreover, if $\det \Phi = c(t-\theta)^m$ for some $c \in \overline{K}^\times$ and $m \geq 0$, we further have 
\begin{equation}
    \operatorname{tr.deg}_{\overline{K}}\overline{K}\left( \Psi_{ij} |_{t=\theta}\right)= \dim \Gamma_P, \label{Papanikolasthmtrdeg}
\end{equation}
 where $\overline{K}\left( \Psi_{ij} |_{t=\theta}\right)$ is the field generated by all entries of $\Psi$ evaluated at $t=\theta$ over $\overline{K}$ (refer also to \cite[Theorem 3.2]{Mishiba2017}). 
\end{Theorem}

Based on the theorem above, throughout this paper we always identify $\Gamma_P$ with $\Gamma_\Psi$ if $\Psi$ is a rigid analytic trivialization of a matrix representing a pre-$t$-motive $P$.
By abuse of language, we simply write $\Gamma_{\Psi}$ and $\Gamma_{P}$ for base changes $\Gamma_{\Psi}\times_{\operatorname{Spec}\mathbb{F}_q(t)} \operatorname{Spec}\overline{\mathbb{F}_q(t)}$ and $\Gamma_{P}\times_{\operatorname{Spec}\mathbb{F}_q(t)} \operatorname{Spec}\overline{\mathbb{F}_q(t)}$, respectively.
We sometimes use symbols $\Gamma_{\Psi}$ and $\Gamma_{P}$ for the group $\Gamma_{\Psi}(\overline{\mathbb{F}_q(t)})\simeq\Gamma_{P}(\overline{\mathbb{F}_q(t)})$ of $\overline{\mathbb{F}_q(t)}$-valued points.

Let $P$ and $P^\prime$ be rigid analytically trivial pre-$t$-motives defined respectively by the matrices $\Phi \in \operatorname{GL}_r(K(t))$ and $\Phi^\prime \in \operatorname{GL}_{r^\prime}(K(t))$ with $\Psi\in \operatorname{GL}_r(\mathbb{L})$ and $\Psi^\prime\in \operatorname{GL}_{r^\prime}(\mathbb{L})$ as their rigid analytic trivializations.
We note that the direct sum $P\oplus P^\prime$ is defined by the matrices $\Phi\oplus\Phi^\prime$
and that this matrix has a rigid analytic trivialization   $\Psi\oplus\Psi^\prime$. 
 Throughout this paper, for any square matrices $B_1$ and $B_2$ the symbol $B_1\oplus B_2$ denotes the canonical block diagonal matrix
 \begin{equation}
    \begin{pmatrix}
        B_1& O\\
        O & B_2
    \end{pmatrix}.
\end{equation}
By the definition \eqref{GammaPsi}, the algebraic group $\Gamma_{\Psi\oplus\Psi^\prime}$ is a closed subgroup of 
\begin{equation}
    \Gamma_{\Psi}\times\Gamma_{\Psi^\prime}=\left \{ B_1\oplus B_2 \ \middle| \ B_1\in \Gamma_{\Psi},\,B_2 \in   \Gamma_{\Psi^\prime} \right\}.
\end{equation}
On the other hand, Tannakian duality yields homomorphisms $\pi : \Gamma_{P\oplus P^\prime}\twoheadrightarrow \Gamma_{P}$ and $\pi^\prime : \Gamma_{P\oplus P^\prime}\twoheadrightarrow \Gamma_{P^\prime}$ of algebraic groups as the Tannakian categories $\langle P \rangle $ and $\langle P^\prime \rangle $ can be seen as subcategories of $\langle P \oplus P^\prime \rangle $. These homomorphisms are faithfully flat (\cite[Proposition 2.21]{Deligne1982}) and hence these induce surjective homomorphisms of groups of $\overline{\mathbb{F}_q(t)}$-valued points. We can describe these homomorphisms in terms of identifications $\Gamma_P=\Gamma_\Psi ,\,\Gamma_{P^\prime}=\Gamma_{\Psi^\prime} $, and $\Gamma_{P \oplus P^\prime}=\Gamma_{\Psi \oplus \Psi^\prime} $ as follows:

\begin{Lemma}\label{LemmaTannakianProj}
In the notations as above, the following diagram commutes:
\begin{center}
\begin{tikzpicture}[auto]
 \node (33) at (3, 3) {$\Gamma_{P\oplus P^\prime}$};
\node (00) at (0, 0) {$\Gamma_{P}$}; \node (30) at (3, 0) {$\Gamma_{P}\times\Gamma_{ P^\prime}$}; \node (60) at (6, 0) {$\Gamma_{P^\prime}$.};

\draw[->>] (33) to node {$\pi$}(00);
\draw[->>] (33) to node {$\pi^\prime$}(60);
\draw[->>] (30) to node {$\operatorname{pr}_1$}(00);
\draw[->>] (30) to node {$\operatorname{pr}_2$}(60);

\draw[{Hooks[right]}->] (33) to node {} (30);

\label{LemmaDiagram}
\end{tikzpicture}

\end{center}
(See \cite[Example 2.3]{Mishiba2015} for example.)
\end{Lemma}
This lemma is well-known by experts but we give a short proof in order to make the present paper self-contained.
We recall that for each $\mathbb{F}_q(t)$-algebra $R$, the $\Gamma_{P\oplus P^\prime}(R)$-action on the Betti realization $R\otimes_{\mathbb{F}_q(t)}(P\oplus P^\prime)^B$ which comes from the equivalence $\langle {P\oplus P^\prime} \rangle \simeq \operatorname{Rep}_{\mathbb{F}_q(t)}(\Gamma_{P\oplus P^\prime})$ is given by
\begin{equation}
    \Psi_{P\oplus P^\prime}^{-1}\mathbf{p}\rightarrow (\Psi_{P\oplus P^\prime}\gamma)^{-1}\mathbf{p},\quad \gamma \in  \Gamma_{P\oplus P^\prime}(R)
\end{equation}
where $\mathbf{p}$ is the $\overline{K}(t)$-basis of $P$ corresponding to $\Phi \oplus \Phi^\prime$ and the action on $\Gamma_{P}(R)$-action on $R\otimes_{\mathbb{F}_q(t)}P^B$ is given by the similar way (\cite[Theorem 4.5.3]{Papanikolas2008}). 


We consider the two $\Gamma_{P\oplus P^\prime}$-actions on $P^B$. The first one is the sub representation of the $\Gamma_{P\oplus P^\prime}(R)$-action $$\Gamma_{P\oplus P^\prime}(R) \curvearrowright R\otimes_{\mathbb{F}_q(t)}(P\oplus P^\prime)^B.$$  
The second one is induced by the action $\Gamma_{P} \curvearrowright R\otimes_{\mathbb{F}_q(t)}P^B$ via the surjection $\operatorname{pr}_1|_{\Gamma_{P\oplus P^\prime}}:\Gamma_{P\oplus P^\prime}\rightarrow \Gamma_{P}$. We can see that these two are the same one. Hence the left square of Lemma \ref{LemmaTannakianProj} is commutative. The commutativity of the right square is proved by similar arguments.

\section{Period interpretations}\label{sectionperiodinterpretation}

The core aim in this paper is considering the transcendence of the Taylor coefficients of $t$-motivic MZVs and $t$-motivic CMPLs at $t=\theta$, and they arise from the periods of some concrete pre-$t$-motives. 
In this section, we construct these pre-$t$-motives by considering the pre-$t$-motives which give us period interpretations of MZVs and special values of CMPLs introduced in \cite{Anderson2009} and \cite{Chang2014}, and applying the technique called prolongations~\cite{Maurischat2018} to them.

 \subsection{Carlitz multiple polylogarithms and Anderson-Thakur series}\label{subsectionATseries}
We recall the notion of $t$-motivic Carlitz multiple polylogarithms, Anderson-Thakur polynomials, and Anderson-Thakur series which play important roles in the period interpretations of MZV's.

For a polynomial
\begin{equation}
    u=\sum_{i=0}^m a_i t^i \in \overline{K}[t],
\end{equation}
we put $||u||_\infty:=\max_{i}(|a_i|_\infty)$. 
Let us take an index $\mathbf{s}=(s_1,\,\dots,\,s_d) \in \mathbb{Z}_{\geq 1}^d$ and a tuple $\mathbf{u}=(u_1,\,\dots,\,u_d) \in \overline{K}[t]^d$ of polynomials with $||u_i||_\infty<| \theta|_\infty^{\frac{s_i
q}{q-1}}$ ($1 \leq i \leq d$). Then we define \textit{$t$-motivic Carlitz multiple polylogarithm} (\cite{Chang2014}) as follows:
\begin{equation}
    \mathcal{L}_{\mathbf{u},\mathbf{s}}(t):=
    \sum_{i_1>\cdots>i_d\geq0}\frac{u_1^{(i_1)}\cdots u_d^{(i_d)}}{((t-\theta)^{(1)}\cdots(t-\theta)^{(i_1)})^{s_1}\cdots ((t-\theta)^{(1)}\cdots(t-\theta)^{(i_r)})^{s_r}}.
\end{equation}
Note that it satisfies the following equation:
\begin{equation}
    \mathcal{L}_{\mathbf{u},\mathbf{s}}^{(-1)}=\frac{u_d^{(-1)}\mathcal{L}_{(u_1,\,\dots,\,u_{d-1}),\,(s_1,\,\dots,\,s_{d-1})}}{(t-\theta)^{s_1+\cdots+s_{d-1}}}  \label{FrobEqPolylog} 
      +\frac{\mathcal{L}_{\mathbf{u},\mathbf{s}}}{(t-\theta)^{s_1+\cdots+s_{d}}}.
\end{equation}
In the case that $u_1,\,\dots,u_d \in \overline{K}$, the value at $t=\theta$ is equal to the special value
\begin{equation}
    \operatorname{Li}_{\mathbf{s}}(\mathbf{u})
    :=\sum_{i_1>\cdots>i_d\geq0}\frac{u_1^{(i_1)}\cdots u_d^{(i_d)}}{((\theta-\theta^{(1)})\cdots(\theta-\theta^{(i_1)}))^{s_1}\cdots ((\theta-\theta^{(1)})\cdots(\theta-\theta^{(i_r)}))^{s_r}}.
\end{equation}
of \textit{Carlitz multiple polylogarithm} (\cite{Chang2014}).

Anderson and Thakur introduced in \cite{Anderson1990} a sequence $H_0,\,H_1,\,\dots \in A[t]$ of polynimials, which are called \textit{Anderson-Thakur polynomials} by the following generating series:
\begin{equation}
  \left(1-\sum_{i\geq0}\frac{\prod_{j=1}^i (t^{q^i}-\theta^{q^j})}{\prod_{j=0}^{i-1}(t^{q^i}-t^{q^j})}x^{q^i} \right)^{-1}=\sum_{s\geq0} \frac{H_s(t)}{\Gamma_{s+1|_{\theta=t}}}x^{q^s}.
\end{equation}
Here $\Gamma_{s+1}$ is the Carlitz factorial defined as follows: for a non-negative integer $s$ with the $q$-adic digit expansion
\begin{equation}
    s=\sum_{i=0}^m s_{(i)}q^i, \quad (0 \leq s_{(i)} \leq q-1)
\end{equation}
we put 
\begin{equation}
    \Gamma_{s+1}:=\prod_{i=0}^m D_i^{s_{(i)}}\in A
\end{equation}
where $D_i$ is the product of all monic polinomial in $A$ of degree $i$, see \cite{ThakurBook} for details.
Anderson-Thakur polynomials enable us to interpret MZV's in terms of special values of Carlitz multiple polylogarithms.
Anderson and Thakur~(\cite[(3.7.3)]{Anderson1990}) obtained the inequality \begin{equation}||H_{s-1}(t)||_\infty \leq |\theta |_\infty^{\frac{sq}{q-1}},
\end{equation}
for $s\geq1$ and they showed that we have
\begin{equation}
   \mathcal{L}_{(H_{s_1-1},\,\dots,\,H_{s_r-1}),(s_1,\,\dots,\,s_d)}(t)|_{t=\theta}=\Gamma_{s_1}\cdots\Gamma_{s_r}\zeta_A(s_1,\,\dots,\,s_d), \label{ATpolynom}
\end{equation}
for any index $(s_1,\,\dots,\,s_d)\in \mathbb{Z}_{\geq 1}^d$ (\cite[Theorem 3.8.3]{Anderson1990} and \cite{Anderson2009}, see also \cite{Chang2014}).
\begin{definition}\label{DefATseries}
For an index $\mathbf{s}=(s_1,\,\dots,\,s_d) \in \mathbb{Z}_{\geq 1}^d$, the \textit{Anderson-Thakur series} $\zeta_A^{\mathrm{AT}}(s_1,\,\dots,\,s_r)$ is defined to be the series
\begin{equation}
    \mathcal{L}_{(H_{s_1-1},\,\dots,\,H_{s_r-1}),(s_1,\,\dots,\,s_d)}(t)\in \mathbb{C}_{\infty}((t)).
\end{equation}
By equation \eqref{ATpolynom}, these series can be viewed as \textit{$t$-motivic multiple zeta values}.
\end{definition}

Let us recall the period interpretations of multiple zeta values and special values of Carlitz multiple polylogarithms at algebraic points (\cite{Anderson2009} and \cite{Chang2014}). We take $\mathbf{s}=(s_1,\,\dots,\,s_r) \in \mathbb{Z}_{\geq 1}^d$ and $\mathbf{u}=(u_1,\,\dots,\,u_d) \in \overline{K}[t]^d$ so that
\begin{equation}
        ||u_i||_\infty<|\theta|_\infty^{\frac{s_i q}{q-1}},
    \end{equation}
and consider the pre-$t$-motive $M[\mathbf{u};\mathbf{s}]$ defined by the matrix
\begin{equation}
\Phi[\mathbf{u};\mathbf{s}] :=\begin{pmatrix}
   (t-\theta)^{s_{1}+\cdots+s_{d}} & 0& \cdots && \\
   (t-\theta)^{s_{1}+\cdots+s_{d}} u_{1}^{(-1)} & (t-\theta)^{s_{2}+\cdots+s_{d}}&0&\cdots&\\
   &\ddots&\ddots&\ddots& \\
   & &&(t-\theta)^{s_{d}}& 0\\
  & & &(t-\theta)^{s_{d}} u_{d}^{(-1)} &1
\end{pmatrix}.\label{PhiUS}
\end{equation}
As we have equation \eqref{FrobEqPolylog}, this defining matrix has rigid analytic trivialization 
\begin{align}
\Psi[\mathbf{u};\mathbf{s}]\label{PsiUS}
&:=\begin{pmatrix}
   \Omega^{s_{1}+\cdots+s_{d}} & 0& \cdots && \\
   \Omega^{s_{1}+\cdots+s_{d}}  L_{1,\,1} &  \Omega^{s_{2}+\cdots+s_{d}}&0&\cdots&\\
   \vdots&&\ddots&\ddots& \\
  \Omega^{s_{1}+\cdots+s_{d}}  L_{1,\,{d-1}}& &&\Omega^{s_{d}} & 0\\
  \Omega^{s_{1}+\cdots+s_{d}}  L_{1,\,d}&   \Omega^{s_{2}+\cdots+s_{d}} L_{2,\,d}&\cdots& \Omega^{s_{d}}  L_{d}&1
\end{pmatrix}
\end{align}
where each $L_{i,\,j}$ denotes the series $\mathcal{L}_{(u_i,\,u_{i+1},\,\dots,\,u_j),\,(s_i,\,s_{i+1},\,\dots,\,s_j)}$ for $1\leq i\leq j\leq d$ and $\Omega$ is the series introduced in Example \ref{Omega}. Considering the case where $u_i=H_{s_i-1}$ for $1 \leq i \leq d$, we can see that the MZV $\zeta_A(s_1,\,\dots,\,s_d)$ can be written in terms of periods of pre-$t$-motives.

\begin{Remark}
Let us recall the notation in Equation \eqref{GammaPsi}.
For later use, we mentioned that the $(i,\,j)$-component of the matrix $\widetilde{\Psi[\mathbf{u};\mathbf{s}]}$ is given by
    \begin{align}
       &(\Omega^{-1}\otimes \Omega)^{s_i+\cdots+s_d}\\ &\cdot\sum_{n=j}^{i}\sum_{m=0}^{i-n}(-1)^m \sum_{\substack{n=k_0<k_1<\cdots\\
    \cdots<k_{m-1}<k_m=i}}L_{k_1,\,k_0}L_{k_2,\,k_1}\cdots L_{k_m,\,k_m-1}\otimes \Omega^{s_1+\cdots+s_d}L_{n,\,j}
    \end{align}
for $1 \leq j \leq i \leq d+1$, and $(i,\,i)$-component is given by $(\Omega^{-1}\otimes \Omega)^{s_i+\cdots+s_d}$ for $1 \leq i \leq d+1$ (see \cite{Mishiba2015}). Here, we put $L_{i,\,i}=1$ for $1 \leq i \leq d+1$ by convention. Indeed, we have
\begin{align}
     &(\Omega^{-1}\cdot\Omega)^{s_i+\cdots+s_d}\\
     &\quad \cdot\sum_{n=j}^{i}\sum_{m=0}^{i-n}(-1)^m \sum_{\substack{n=k_0<k_1<\cdots\\
    \cdots<k_{m-1}<k_m=i}}L_{k_1,\,k_0}L_{k_2,\,k_1}\cdots L_{k_m,\,k_m-1}\Omega^{s_1+\cdots+s_d}L_{n,\,j}\\
    =&\sum_{n=j}^{i}\sum_{m=0}^{i-n}(-1)^m \sum_{\substack{n=k_0<k_1<\cdots\\
    \cdots<k_{m-1}<k_m=i}}L_{k_1,\,k_0}L_{k_2,\,k_1}\cdots L_{k_m,\,k_m-1}\Omega^{s_1+\cdots+s_d}L_{n,\,j}\\
    =&\sum_{m=0}^{i-j}(-1)^m \sum_{\substack{j=k_{-1} \leq k_0<k_1<\cdots\\
    \cdots<k_{m-1}<k_m=i}}L_{k_0,\,k_{-1}}L_{k_1,\,k_0}L_{k_2,\,k_1}\cdots L_{k_m,\,k_m-1}\Omega^{s_1+\cdots+s_d}=0
\end{align}
for $1 \leq j \leq i \leq d+1$.
\end{Remark}

\begin{example}\label{example depth 1}
  We consider the case $\mathbf{s}=(s)$ and $\mathbf{u}=(u)$. Let us assume 
\begin{equation}
    ||u||_\infty<|\theta|_\infty^{\frac{s q}{q-1}}.
\end{equation}
The pre-$t$-motive $C\oplus M[\mathbf{u};\mathbf{s}]$ is with a defining matrix and its rigid analytic trivialization
  \begin{equation}
      (t-\theta)\oplus \begin{pmatrix}
          (t-\theta)^{s}&0\\
          (t-\theta)^{s} u^{(-1)}&1
      \end{pmatrix}
      \text{ and }
      (\Omega)\oplus\begin{pmatrix}
          \Omega^{s}&0\\
          \Omega^{s}\mathcal{L}_{\mathbf{u},\,\mathbf{s}}&1
      \end{pmatrix}.
  \end{equation}
  Therefore, the $t$-motivic Galois group $\Gamma_{C\oplus M[\mathbf{u};\mathbf{s}]}\simeq \Gamma_{(\Omega)\oplus \Psi[\mathbf{u};\mathbf{s}]}$ is a closed subgroup of the algebraic group
  \begin{equation}
      \left \{ a\oplus\begin{pmatrix}
          a^{s}&0\\
          a^{s} x&1
          \end{pmatrix} \middle|\,a\neq 0 \right\}.
  \end{equation}
  Chang and Yu~\cite{Chang2007} show that $\tilde{\pi}$ and $\mathcal{L}_{\mathbf{u},\,\mathbf{s}}|_{t=\theta}$ are algebraically independent whenever $s$ is not divisible by $q-1$, $\mathcal{L}_{\mathbf{u},\,\mathbf{s}} \in \mathbb{T}$, and $\mathcal{L}_{\mathbf{u},\,\mathbf{s}}|_{t=\theta}\neq 0$. Therefore, we have $\dim \Gamma_{C\oplus M[\mathbf{u};\mathbf{s}]} =2$ and 
  \begin{equation}
     \Gamma_{C\oplus M[\mathbf{u};\mathbf{s}]} \simeq \Gamma_{(\Omega)\oplus \Psi[\mathbf{u};\mathbf{s}]} \left \{ a\oplus\begin{pmatrix}
          a^{s}&0\\
          a^{s} x&1
          \end{pmatrix} \middle|\,a\neq 0 \right\}.
  \end{equation}
\end{example}

 \subsection{Hyperderivatives and prolongations of pre-$t$-motives}\label{subsectionprolongation}
We study transcendence properties of Taylor coefficients of the Anderson-Thakur series and $t$-motivic CMPLs at $t=\theta$. These coefficients arise from the periods of some concrete pre-$t$-motives $\rho_n M[\mathbf{u},\,\mathbf{s}]$. 
In this section, we construct these pre-$t$-motives $\rho_n M[\mathbf{u},\,\mathbf{s}]$ by considering the pre-$t$-motives $M[\mathbf{u},\,\mathbf{s}]$'s in Section \ref{subsectionATseries}, which give us period interpretations of MZV's and special values of CMPLs and were introduced in \cite{Anderson2009} and \cite{Chang2014}, and applying the technique called prolongations, which is introduced by Maurischat~(\cite{Maurischat2018}), to those pre-$t$-motives $M[\mathbf{u},\,\mathbf{s}]$'s.

Maurischat introduced the technique called prolongations, which enables us to deal with Taylor coefficients at $t=\theta$ of entries of rigid analytic trivialization of a matrix defining a given pre-$t$-motive. 
We also mention that Namoijam and Papanikolas (\cite{Namoijam2022}) studied hyperderivatives of entries of rigid analytic trivialiations of matrices defining pre-$t$-motives.
In this section, we review Maurischat's theory on prolongations in the language of pre-$t$-motives and obtain period interpretations of Taylor coefficients at $t=\theta$ of $t$-motivic CMPLs and of Anderson-Thakur series as applications.

\begin{definition}

For each $n \geq0$, we define the $\mathbb{C}_\infty$-linear operator $\partial^{(n)}$ on $\mathbb{C}_\infty((t))$, which is called \textit{$n$-th hyperderivative}, by
\begin{align}
  &\partial^{(n)}:& &\mathbb{C}_\infty((t))& &\rightarrow& &\mathbb{C}_\infty ((t))  \label{defHD} \\
 && &\sum_{i=m}^\infty a_i t^i& &\mapsto&
 &\sum_{i=m}^\infty \binom{i}{n} a_i t^{i-n} 
\end{align}
where $a_i \in\mathbb{C}_\infty$, $\binom{i}{n}$ are the binomial coefficients modulo $p$, and we put $\binom{i}{n}=0$ for $i <n$.
\end{definition}
We have the following analogue of the Leibniz rule
\begin{equation}
   \partial^{(n)}(f_1\cdots f_r)=\sum_{j_1+j_2+\cdots+j_s=n}\partial^{(j_1)}f_1\partial^{(j_2)}f_2\cdots \partial^{(j_r)}f_r \label{Leibniz}
\end{equation}
for $n \geq 0$ and $f_1,\,\dots,\,f_r \in \mathbb{C}_\infty((t))$, see \cite{Maurischat2022}.
The following proposition relates the considerations on hyperderivatives of a power series in $\mathbb{T}$ to considering its Taylor coefficients at $t=\theta$ if the series converges at $t=\theta$:

\begin{Proposition}[{\cite[Corollary 2.7]{Uchino1998}}]\label{Uchino}
   Suppose $f \in \mathbb{T}$ converges at $t=\theta$ and write
   \begin{equation}
  f=\sum_{i=0}^\infty a_i (t-\theta)^i\in \mathbb{T},
  \end{equation}
  then its hyperderivative $\partial^{(i)}f$ also converges at $t=\theta$ for each $i$ and the equalityh $\partial^{(i)}f|_{t=\theta}=a_i$ holds.
\end{Proposition}

For any matrix $B=(b_{ij})$ with entries in $\mathbb{C}_\infty((t))$, we define $\partial^{(n)}B:=(\partial^{(n)}b_{ij})$. 
Using equation \eqref{Leibniz}, we can show that the map \begin{equation}
    \rho_{n}:\operatorname{Mat}_{d\times d}\left(\mathbb{C}_\infty((t))\right) \rightarrow \operatorname{Mat}_{d(n+1)\times d(n+1)}\left(\mathbb{C}_\infty((t))\right)
\end{equation}
    defined by
\begin{equation}
    \rho_n(X):=
    \begin{pmatrix}
   X & 0 &\cdots&\cdots&0\\
   \partial^{(1)}X & X&0&\cdots&0\\
   \vdots& \vdots&\ddots&&\vdots \\
   \partial^{(n-1)}X & \partial^{(n-2)}X&\cdots&X&0\\
   \partial^{(n)}X &\partial^{(n-1)}X &\cdots&\partial^{(1)}X &X
\end{pmatrix}
\end{equation}
is a homomorphism of $\mathbb{C}_\infty((t))$-algebras for each $n \geq 0$, see \cite{Maurischat2022}.

\begin{definition}
    For a pre-$t$-motive defined by a matrix $\Phi \in \operatorname{Mat}_r(\overline{K}[t])\cap \operatorname{GL}_r(\overline{K}(t))$, then its $n$-th prolongation $\rho_n P$ is the pre-$t$-motive defined by the matrix $\rho_n \Phi\in \operatorname{Mat}_{(n+1)r}(\overline{K}[t])\cap \operatorname{GL}_{(n+1)r}(\overline{K}(t))$.
\end{definition}
If we have an equation $\Psi^{(-1)}=\Phi \Psi$, then we have $(\rho_n\Psi)^{(-1)}=(\rho_n\Phi)\cdot( \rho_n\Psi)$ for each $n$ as $\rho_n$ is a homomorphism and commutes with $(-1)$-th fold twisting. Therefore, if $P$ is rigid analytically trivial then it follows that so is its $n$-th prolongation $\rho_n P$.

\begin{example} \label{Maurischatresult}
    For $n \geq1$, the $n$-th prolongation $\rho_nC$ of the Carlitz motive $C$ has the following representing matrix

\begin{equation}
\begin{pmatrix}
    t-\theta& 0&0 &\cdots & 0\\
    1& t-\theta & 0&\cdots & 0  \\
    0 &1 & t-\theta & \ddots&\vdots\\
    \vdots &\ddots & \ddots &\ddots&0\\
    0& 0 &\cdots & 1&t-\theta
\end{pmatrix}
\end{equation}
of size $n+1$, which has a rigid analytic trivialization
\begin{equation}
\begin{pmatrix}
   \Omega & 0 &\cdots&0\\
   \partial^{(1)} \Omega & \Omega&\cdots&0\\
   \vdots& \vdots&\ddots&\vdots \\
   \partial^{(n)} \Omega & \partial^{(n-1)} \Omega &\cdots& \Omega 
\end{pmatrix}.
\end{equation}
By Equation \eqref{GammaPsi} and Theorem \ref{ThmPapanikolas}, we have the closed immersion 
\begin{equation}
\Gamma_{\rho_n C} \subset \left \{  \begin{pmatrix}
   a_0 & 0 &\cdots&0\\
   a_1 & a_0&\cdots&0\\
   \vdots& \vdots&\ddots&\vdots \\
   a_{n}& a_{n-1}&\cdots&a_0
\end{pmatrix} \, \middle|\, a_0\in \mathbb{G}_m,\,a_1,\,\dots,\,a_n \in \mathbb{A}
\right\},
\label{MaurischatImmersion}
\end{equation}
where the latter is an algebraic group over $\overline{\mathbb{F}_q(t)}$ isomorphic to the smooth and connected variety $\mathbb{G}_m \times \mathbb{A}^n$ as varieties. Since Maurischat~(\cite[Theorem 7.2]{Maurischat2018}) showed that the values
$$
\Omega|_{t=\theta},\,\partial \Omega|_{t=\theta},\dots,\, \partial^{(n)} \Omega|_{t=\theta}
$$
are algebraically independent over $\overline{K}$, we have $\dim \Gamma_{\rho_n C}=n+1$ by Theorem \ref{ThmPapanikolas}. Hence the closed immersion \eqref{MaurischatImmersion} is an isomorphism. We note that the group $\Gamma_{\rho_n C}(\overline{\mathbb{F}_q(t)})$ of $\overline{\mathbb{F}_q(t)}$-valued points is commutative. Indeed, if we take two elements
\begin{equation}
    \begin{pmatrix}
   a_0 & 0 &\cdots&\cdots&\cdots&0\\
   a_1 & a_0&0&\cdots&\cdots&0\\
   a_2 & a_1& a_0&0&\cdots&0\\
   \vdots&\ddots& \ddots&\ddots&&\vdots \\
   a_{n-1}& a_{n-2}&\cdots &a_1&a_0&0\\
   a_n&a_{n-1}&\cdots &\cdots&a_1&a_0
\end{pmatrix} \text{ and }
\begin{pmatrix}
   a_0^\prime & 0 &\cdots&\cdots&\cdots&0\\
   a_1\prime & a_0\prime&0&\cdots&\cdots&0\\
   a_2\prime & a_1\prime& a_0\prime&0&\cdots&0\\
   \vdots&\ddots& \ddots&\ddots&&\vdots \\
   a_{n-1}\prime& a_{n-2}\prime&\cdots &a_1\prime&a_0\prime&0\\
   a_n\prime&a_{n-1}\prime&\cdots &\cdots&a_1\prime&a_0\prime
\end{pmatrix}
\end{equation}
of the group, then the product is the lower triangle matrix whose $(i,\,j)$-entry is
\begin{equation}
    \sum_{l=0}^{i-j} a_l a_{i-j-l}^\prime
\end{equation}
for $ i \geq j$.
\end{example}
For a non-negative integers $0 \leq n^\prime \leq n$ and pre-$t$-motive $P$, the prolongation $\rho_{n^\prime}P$ is a sub-pre-$t$-motive of $\rho_{n}P$ (cf. \cite[Remark 3.2]{Maurischat2018}). In particular, the prolongation $\rho_{n^\prime}C$ is a sub-pre-$t$-motive of $\rho_{n}C$ and hence $\langle \rho_{n^\prime}C \rangle $ can be seen as a full subcategory of $\langle \rho_{n}C \rangle$. Therefore we have faithfully flat morphism $\pi: \Gamma_{\rho_{n}C} \twoheadrightarrow \Gamma_{\rho_{n^\prime}C}$. This surjection is described concretely by the following proposision:
\begin{Proposition}\label{ExamplerhoCsurj}
     The morphism $\pi:\Gamma_{\rho_{n}C} \twoheadrightarrow \Gamma_{\rho_{n^\prime}C}$ above is given by
\begin{equation}\label{explicitmorph}
    \begin{pmatrix}
   a_0 & 0 &\cdots&0\\
   a_1 & a_0&\cdots&0\\
   \vdots& \vdots&\ddots&\vdots \\
   a_{n}& a_{n-1}&\cdots&a_0
\end{pmatrix}
\mapsto
\begin{pmatrix}
   a_0 & 0 &\cdots&0\\
   a_1 & a_0&\cdots&0\\
   \vdots& \vdots&\ddots&\vdots \\
   a_{n^\prime}& a_{n^\prime-1}&\cdots&a_0
\end{pmatrix}.
\end{equation}
\end{Proposition}
\begin{proof}
In order to deduce this proposition by Lemma \ref{LemmaTannakianProj}, we further consider the pre-$t$-motive $(\rho_{n^\prime}C)\oplus(\rho_{n}C)$. Both of the Tannakian categories $\langle \rho_{n^\prime}C \rangle $ and $\langle \rho_{n}C \rangle$ are subcategories of $\langle (\rho_{n^\prime}C)\oplus(\rho_{n}C) \rangle$, so we have faithfully flat morphisms $\pi^\prime$ and $\pi^{\prime \prime}$ which make the following diagram commutative:
\begin{center}
\begin{tikzpicture}[auto]
 \node (33) at (3, 3) {$\Gamma_{(\rho_{n^\prime}C)\oplus(\rho_{n}C)}$};
\node (30) at (3, 0) {$\Gamma_{\rho_{n}C}$}; \node (60) at (6, 0) {$\Gamma_{\rho_{n^\prime}C}$.};

\draw[->>] (33) to node {$\pi^{\prime \prime}$}(60);

\draw[->>] (30) to node {$\pi$}(60);

\draw[->>] (33) to node {$\pi^\prime$} (30);

\label{LemmaDiagram2}
\end{tikzpicture}.
\end{center}
By Lemma \ref{LemmaDiagram}, these morphisms are given by
\begin{align}
    \pi^{\prime}:\begin{pmatrix}
   a_0 & 0 &\cdots&0\\
   a_1 & a_0&\cdots&0\\
   \vdots& \vdots&\ddots&\vdots \\
   a_{n}& a_{n-1}&\cdots&a_0
\end{pmatrix}
\oplus
\begin{pmatrix}
   a_0 & 0 &\cdots&0\\
   a_1 & a_0&\cdots&0\\
   \vdots& \vdots&\ddots&\vdots \\
   a_{n^\prime}& a_{n^\prime-1}&\cdots&a_0
\end{pmatrix} \mapsto \begin{pmatrix}
   a_0 & 0 &\cdots&0\\
   a_1 & a_0&\cdots&0\\
   \vdots& \vdots&\ddots&\vdots \\
   a_{n}& a_{n-1}&\cdots&a_0
\end{pmatrix}\\
\intertext{and}
\pi^{\prime \prime}:\begin{pmatrix}
   a_0 & 0 &\cdots&0\\
   a_1 & a_0&\cdots&0\\
   \vdots& \vdots&\ddots&\vdots \\
   a_{n}& a_{n-1}&\cdots&a_0
\end{pmatrix}
\oplus
\begin{pmatrix}
   a_0 & 0 &\cdots&0\\
   a_1 & a_0&\cdots&0\\
   \vdots& \vdots&\ddots&\vdots \\
   a_{n^\prime}& a_{n^\prime-1}&\cdots&a_0
\end{pmatrix} \mapsto \begin{pmatrix}
   a_0 & 0 &\cdots&0\\
   a_1 & a_0&\cdots&0\\
   \vdots& \vdots&\ddots&\vdots \\
   a_{n^\prime}& a_{n^\prime-1}&\cdots&a_0
\end{pmatrix}.
\end{align}
We note that $\pi^{\prime}$ is an isomorphism since the two subcategories $\langle \rho_{n}C \rangle$ and $\langle (\rho_{n^\prime}C)\oplus(\rho_{n}C) \rangle$ are equal as subcategories of $\mathcal{R}$. As $t$-motivic Galois groups are reduced, the morphism $\pi$ can be written as \eqref{explicitmorph}.
\end{proof}

We define $ \mathcal{U}_{\rho_n C}$ to be the kernel of the mophism $\Gamma_{\rho_n C} \twoheadrightarrow \Gamma_{\rho_0C}=\Gamma_C \simeq \mathbb{G}_m$. We can show that
\begin{equation}\label{EqexplicitUrhonC}
   \mathcal{U}_{\rho_n C}\simeq  \left \{   \begin{pmatrix}
   1 & 0 &\cdots&\cdots&\cdots&0\\
   a_1 & 1&0&\cdots&\cdots&0\\
   a_2 & a_1& 1&0&\cdots&0\\
   \vdots&\ddots& \ddots&\ddots&&\vdots \\
   a_{n-1}& a_{n-2}&\cdots &a_1&1&0\\
   a_n&a_{n-1}&\cdots &\cdots&a_1&1
\end{pmatrix} \, \middle|\,a_1,\,\dots,\,a_n \in \mathbb{A}
\right\}.
\end{equation}

\begin{example}
    Let us put $\mathbf{s}=(s,\,s^\prime) \in \mathbb{Z}^2$ and $\mathbf{u}=(u,\,u^\prime)  \in \overline{K}[t]^2$ such that we have
    \begin{equation}
        ||u||_\infty <|\theta|^{\frac{sq}{q-1}} \text{ and } ||u^\prime||_\infty <|\theta|^{\frac{s^\prime q}{q-1}},
    \end{equation} 
    We can consider the pre-$t$-motive $M[\mathbf{u},\,\mathbf{s}]$ in Subsection \ref{subsectionATseries}. For $n\geq0$, its $n$-th prolongation $\rho_n M[\mathbf{u},\,\mathbf{s}]$ is represented by the matrix $\rho_n\Phi[\mathbf{u},\,\mathbf{s}]$, where $\Phi[\mathbf{u},\,\mathbf{s}]$ is the matrix defined in \eqref{PhiUS}. This has a rigid analytic trivialization $\rho_n \Psi[\mathbf{u},\,\mathbf{s}]$, where
        \begin{equation}
             \Psi[\mathbf{u},\,\mathbf{s}]:=
                \begin{pmatrix}
                    \Omega^{s+s^\prime}&0&0\\
                    \Omega^{s+s^\prime}\mathcal{L}_{(u),\,(s)}&\Omega^{s^\prime}&0 \\
                    \Omega^{s+s^\prime}\mathcal{L}_{\mathbf{u},\,\mathbf{s}}&\Omega^{s^\prime}\mathcal{L}_{(u^\prime),\,(s^\prime)}&1
                \end{pmatrix},
        \end{equation}
    see \eqref{PsiUS}. By definition, the matrix $\partial^{(n)}\Psi[\mathbf{u},\,\mathbf{s}]$ is equal to
    \begin{center}
        \begin{equation}\tiny{
                \begin{pmatrix}
                    \displaystyle{\sum_{i_1+\cdots+i_{s+s^\prime}=n}}(\partial^{(i_1)}\Omega)\cdots(\partial^{(i_{s+s^\prime})}\Omega)&0&0\\
                     \displaystyle{\sum_{i_1+\cdots+i_{s+s^\prime}+j=n}}(\partial^{(i_1)}\Omega)\cdots(\partial^{(i_{s+s^\prime})}\Omega)(\partial^{(j)}\mathcal{L}_{(u),\,(s)})&\displaystyle{\sum_{i_1+\cdots+i_{s^\prime}=n}}(\partial^{(i_1)}\Omega)\cdots(\partial^{(i_{s^\prime)}}\Omega)& 0\\
                    \displaystyle{\sum_{i_1+\cdots+i_{s+s^\prime}+j=n}}(\partial^{(i_1)}\Omega)\cdots(\partial^{(i_{s+s^\prime})}\Omega)(\partial^{(j)}\mathcal{L}_{\mathbf{u},\,\mathbf{s}})&\displaystyle{\sum_{i_1+\cdots+i_{s^\prime}+j=n}}(\partial^{(i_1)}\Omega)\cdots(\partial^{(i_{s^\prime})}\Omega)(\partial^{(j)}\mathcal{L}_{(u^\prime),\,(s^\prime)})&0
                \end{pmatrix}}
        \end{equation}
    \end{center}
    for $n \geq 1$. Therefore we can interpret the special values $\partial^{(n)}\mathcal{L}_{\mathbf{u},\,\mathbf{s}}|_{t=\theta}$ in terms of periods of the pre-$t$-motives $\rho_n M[\mathbf{u},\,\mathbf{s}]$ and of $\rho_n C$. 
\end{example}

We can obtain the period interpretations of the special values at $t=\theta$ of hyperderivatives $ \partial^{(n)}\mathcal{L}_{\mathbf{u},\mathbf{s}}(t)$ of $t$-motivic CMPLs for arbitrary $\mathbf{s}=(s_1,\,\dots,\,s_r) \in \mathbb{Z}_{\geq 1}^r$ and $\mathbf{u}=(u_1,\,\dots,\,u_r) \in \overline{K}[t]^r$ with 
\begin{equation}
        ||u_i||_\infty<|\theta|_\infty^{\frac{s_i q}{q-1}}
    \end{equation}
in the same manner, and those of Anderson-Thakur series by putting $u_i=H_{s_i-1}$.
By virtue of Proposition \ref{Uchino}, this means that we have period interpretations of Taylor coefficients at $t=\theta$ of $t$-motivic CMPLs and of Anderson-Thakur series.

\section{Carliz zeta values and Taylor coefficients of the power series $\Omega$}\label{sectionfirststep}

Our aim in this section is to show that there exists no algebraic relation among Taylor coefficients $\alpha_0,\,\alpha_1,\,\dots$ of the power series $\Omega$ (see Example \ref{Omega}) and Carlitz zeta values.
Namely, we will show that the set
\begin{equation}
    \{ \alpha_0,\,\dots,\,\alpha_n,\,\zeta_A(s_1),\,\dots,\,\zeta_A(s_r)\}
\end{equation}
is an algebraically independent set for any $n\geq0$ and the choice of positive integers $(s_1,\,\dots,\,s_r)$ such that $(q-1)\nmid s_i$ for $1 \leq i\leq r$ and the ratio $s_i /s_j$ is not in the set $p^{\mathbb{Z}}$ of integer powers of $p$. 
More generaly, we take polynomials $u_1,\,\dots,\,u_r \in \overline{K}[t]$ such that
    \begin{equation}\label{conditiononui_firststep1}
        ||u_i||_\infty<|\theta|_\infty^{\frac{s_i q}{q-1}}
    \end{equation}
    and
    \begin{equation}\label{conditiononui_firststep2}
    \mathcal{L}_{(u_i),\,(s_i)}(t)|_{t=\theta}\neq 0
    \end{equation}
for each $1 \leq i \leq r$ and prove the algebraic independence of the set 
\begin{equation}
    \{ \alpha_0,\,\dots,\,\alpha_n,\,\mathcal{L}_{(u_1),\,(s_1)}|_{t=\theta},\,\dots,\,\mathcal{L}_{(u_r),\,(s_r)}|_{t=\theta}\}.
\end{equation}

In order to obtain the aforementioned algebraically independence result, we will construct a certain pre-$t$-motive $M(n)$, which depends also on the choice of $s_1,\,\dots,\,s_r$ and $u_1,\,\dots,\,u_r$, for each $n$ and describe the structure of $t$-motivic Galois group of $M(n)$ explicitly. We construct an explicit algebraic group $G(n)$ for each $n \geq0$ and we will show that the $t$-motivic Galois group $\Gamma_{M(n)}$ is isomorphic to $G(n)$ by induction on $n$.

\begin{definition}\label{Defpretfirststep}
    Take an index $(s_1,\,\dots,\,s_r)\in \mathbb{Z}_{>0}^r$ and polynomials $u_1,\,\dots,\,u_r \in \overline{K}[t]$. For $1 \leq i \leq r$, write $M[i]$ for the pre-$t$-motive $M[(u_i);(s_i)]$ in Section \ref{subsectionATseries}. 
    For $n \geq 0$, we put \begin{equation}
        M(n):=\rho_n C \oplus \bigoplus_{i=1}^rM[i].
    \end{equation}
\end{definition}

The pre-$t$-motive $M(n)$ is represented by the matrix
\begin{equation}
    \Phi(n):=\begin{pmatrix}
    t-\theta& 0&0 &\cdots & 0\\
    1& t-\theta & 0&\cdots & 0  \\
    0 &1 & t-\theta & \ddots&\vdots\\
    \vdots &\ddots & \ddots &\ddots&0\\
    0& 0 &\cdots & 1&t-\theta
\end{pmatrix}\oplus \bigoplus_{i=1}^r
\begin{pmatrix}
    (t-\theta)^{s_i}&0\\
    (t-\theta)^{s_i}u_i^{(-1)}&1
\end{pmatrix}
\end{equation}
in $\operatorname{GL}_{n+1+2r}(\overline{K}(t)) \cap \operatorname{Mat}_{n+1+2r}(\overline{K}[t])$. If we have Inequalities \eqref{conditiononui_firststep1} and \eqref{conditiononui_firststep2}
    for each $1 \leq i \leq r$, then the representing matrix $\Phi(n)$ has a rigid analytic trivialization
\begin{equation}
\Psi(n):=\begin{pmatrix}
   \Omega0 & 0 &\cdots&\cdots&0\\
   \partial^{(1)}\Omega & \Omega&0&\cdots&0\\
   \partial^{(2)}\Omega & \partial^{(1)}\Omega & \Omega&\ddots&\vdots\\
   \vdots&\vdots& \ddots&\ddots&0 \\
   \partial^{(n)}\Omega & \partial^{(n-1)}\Omega &\cdots &\partial^{(1)}\Omega &\Omega
\end{pmatrix}\oplus \bigoplus_{i_1}^r \begin{pmatrix}
    \Omega^{s_i}&0\\
    \Omega^{s_i}\mathcal{L}_{(u_i),\,(s_i)}&1
\end{pmatrix},
\end{equation}
which is an element of $\operatorname{GL}_{n+1+2r}(\mathbb{L}) \cap \operatorname{Mat}_{n+1+2r}(\mathbb{T})$.

\begin{definition}
For $n \geq 0$ and an index $(s_1,\,\dots,\,s_r)$, let us put $G(n)$ to be the algebraic subgroup of $\operatorname{GL}_{n+1+2r}$ over $\overline{\mathbb{F}_q(t)}$ given as follows:
\begin{equation}
    G(n)(\overline{\mathbb{F}_q(t)}):=
    \left\{
        \begin{pmatrix}
   a_0 & 0 &\cdots&\cdots&0\\
   a_1 & a_0&0&\cdots&0\\
   a_2 & a_1& a_0&\ddots&\vdots\\
   \vdots&\vdots& \ddots&\ddots&0 \\
   a_{n}& a_{n-1}&\cdots &a_1&a_0\\
\end{pmatrix}
        \oplus \bigoplus_{i=1}^r 
        \begin{pmatrix}
            a_0^{s_i}& 0\\ 
            a_0^{s_i} x_i& 1
        \end{pmatrix}
    \,\middle | \,
    a_0 \neq 0
    \right\}.
\end{equation}
\end{definition}
We note that the dimension of $G(n)$ is $n+1+r$. The algebraic groups $G(n)$ will help us to study the structure of $t$-motivic Galois groups $\Gamma_{M(n)}$. We begin by obtaining the following inclusion:

\begin{Proposition}\label{firststeponesideinclusion}
    If the polynomials $u_1,\,\dots,\,u_r \in \overline{K}[t]$ are chosen so that Inequalities \eqref{conditiononui_firststep1} and \eqref{conditiononui_firststep1}
    hold for each $1 \leq i \leq r$, we have an closed immersion 
\begin{equation}
    \Gamma_{M(n)} \subset G(n)
\end{equation}
of algebaic groups over $\overline{\mathbb{F}_q(t)}$.
\end{Proposition}
\begin{proof}
By calculating $\widetilde{\Psi}(n)$ (here, we use the notation in Section \ref{subsectionPapanikolas}), we notice that $\widetilde{\Psi}(n) \in G(n)(\mathbb{L}\otimes_{\overline{K}(t)}\mathbb{L})$.
 As $\Gamma_{\Psi(n)}$ is characterized to be the smallest closed subscheme of $\operatorname{GL}_{n+1+2r}$ such that $\widetilde{\Psi}(n) \in \Gamma_{\Psi(n)}(\mathbb{L}\otimes_{\overline{\mathbb{F}_q(t)}}\mathbb{L}) $, we have $\Gamma_{\Psi(n)}\subset G(n)$. We can identify $\Gamma_{M(n)}$ as $\Gamma_{\Psi(n)}$ by Theorem \ref{ThmPapanikolas}.
\end{proof}
 
In the rest of this section, we assume that polynomials $u_1,\,\dots,\,u_r \in \overline{K}[t]$ are chosen so that we have Inequalities \eqref{conditiononui_firststep1} and \eqref{conditiononui_firststep2} for each $1 \leq i \leq r$.
We will prove that, for each $n \geq 0$, the $t$-motivic Galois group $\Gamma_{M(n)}$ equals $G(n)$ if $(q-1) \nmid s_i$ for each $1 \leq i \leq r$ and the ratio $s_i/s_j$ is not an integer power of $p$ for each $i \neq j$. 
In order to study the structure of algebraic groups $\Gamma_{M(n)}$, we consider their maximal unipotent subgroups instead.
Since the pre-$t$-motive $M(n)$ has the Carlitz motive as its sub-pre-$t$-motive, Tannakian duality yields the surjection $\pi_n:\Gamma_{M(n)} \twoheadrightarrow \Gamma_{C}$. We let $\mathcal{U}_{M(n)}$ be the kernel of $\pi_n$.
Recall that the unipotent algebraic group $\mathcal{U}_{\rho_n C}$ was defined to be the kernel of the surjection $\Gamma_{\rho_n C} \twoheadrightarrow \Gamma_C$ in Proposition \ref{ExamplerhoCsurj} for $n \geq 0$. 
\begin{Lemma}\label{Lemmadiagramchase}
    We have surjective homomorphisms $\Gamma_{M(n)} \twoheadrightarrow \Gamma_{\rho_{n^\prime}C}$
    and this induces  a surjection $\mathcal{U}_{M(n)} \twoheadrightarrow \mathcal{U}_{\rho_{n^\prime}C}$ for $0 \leq n^\prime \leq n$.
\end{Lemma}
\begin{proof}
    As $\rho_{n^\prime} C$ is a sub-pre-$t$-motive of $\Gamma_{M(n)}$, we have a surjective homomorphism $\Gamma_{M(n)} \twoheadrightarrow \Gamma_{\rho_{n^\prime} C}$ by Tannakian duality. We can prove the surjectivity of the homomorphism $\mathcal{U}_{M(n)} \twoheadrightarrow \mathcal{U}_{\rho_{n^\prime}C}$ by the following commutative diagram:
    \begin{equation}\label{Diagramchase}
                \begin{tikzpicture}[auto] 
                   \node (22) at (0, 2) {$\mathcal{U}_{M(n)}$}; \node (42) at (2, 2) {$\Gamma_{ M(n)}$}; \node (62) at (4, 2) {$\Gamma_C$};
                    \node (20) at (0, 0) {$\mathcal{U}_{\rho_{n^\prime}C}$}; \node (40) at (2, 0) {$\Gamma_{\rho_{n^\prime}C}$}; \node (60) at (4, 0) {$\Gamma_C$};
    
                    \draw[-] (62) to node[sloped] {$=$}(60);
                    \draw[->>] (42) to node {}(40);
                    \draw[->] (22) to node {}(20);

                    \draw[{Hooks[right]}->] (22) to node {}(42);
                    \draw[{Hooks[right]}->] (20) to node {}(40);
    
                    \draw[->>] (42) to node {}(62);
                    \draw[->>] (40) to node {}(60);
                \end{tikzpicture}.
            \end{equation}
            If we take an element $\xi$ in $\mathcal{U}_{\rho_{n^\prime} C}(\overline{\mathbb{F}_q(t)})$, we have $\eta \in \Gamma_{M(n)}(\overline{\mathbb{F}_q(t)})$ which is mapped to $\xi \in \mathcal{U}_{\rho_{n^\prime} C}(\overline{\mathbb{F}_q(t)}) \subset \Gamma_{\rho_{n^\prime}C}(\overline{\mathbb{F}_q(t)})$. As horizontal lines of the diagram above are exact, $\xi$ goes to $1$ in $\Gamma_C(\overline{\mathbb{F}_q(t)}) \simeq \mathbb{G}_m(\overline{\mathbb{F}_q(t)})$ and so does $\eta$ by the commutativity of the diagram. Hence $\eta$ is an element of $\mathcal{U}_{M(n)}(\overline{\mathbb{F}_q(t)})$ and we obtain the surjectivity of $\mathcal{U}_{M(n)} \rightarrow \mathcal{U}_{\rho_{n^\prime} C}$.
\end{proof}
For $n \geq 1 $, the surjection $\Gamma_{\rho_n C} \twoheadrightarrow \Gamma_{\rho_{n-1}C}$ (see Proposition \ref{ExamplerhoCsurj} again) induces an surjection $\mathcal{U}_{\rho_n C}\twoheadrightarrow \mathcal{U}_{\rho_{n-1} C}$ and we define $\mathcal{V}_{\rho_n C}$ to be its kernel. We have
\begin{equation}\label{EqdescriptionVrhonC}
    \mathcal{V}_{\rho_n C}= \left\{ \begin{pmatrix}
   1 & 0 &\cdots&\cdots&0\\
   0& 1&0&\cdots&0\\
   \vdots& \ddots&\ddots&\ddots&\vdots \\
   0& 0&\cdots&1&0\\
   a_n&0 &\cdots&0&1
\end{pmatrix}
    \right\} \simeq \mathbb{G}_a. 
\end{equation}

Let us consider the homomorphism $G(n) \rightarrow \mathbb{G}_m$ given by
\begin{equation}
    \begin{pmatrix}
   a_0 & 0 &\cdots&\cdots&0\\
   a_1 & a_0&0&\cdots&0\\
   a_2 & a_1& a_0&\ddots&\vdots\\
   \vdots&\vdots& \ddots&\ddots&0 \\
   a_{n}& a_{n-1}&\cdots &a_1&a_0\\
\end{pmatrix}
        \otimes \bigoplus_{i=1}^r 
        \begin{pmatrix}
            a_0^{s_i}& 0\\ 
            a_0^{s_i} x_i& 1
        \end{pmatrix} \mapsto (a_0).
\end{equation}
We let 
\begin{equation}\label{EqexplicitU(n)}
    U(n):= \left \{  \begin{pmatrix}
   1 & 0 &\cdots&\cdots&0\\
   a_1 & 1&0&\cdots&0\\
   a_2 & a_1& 1&\ddots&\vdots\\
   \vdots&\vdots& \ddots&\ddots&0 \\
   a_{n}& a_{n-1}&\cdots &a_1&1\\
\end{pmatrix}
        \oplus \bigoplus_{i=1}^r 
        \begin{pmatrix}
            1& 0\\ 
             x_i& 1
        \end{pmatrix}\right\}
\end{equation}
be the kernel of this surjection. 

We assume $(q-1) \nmid s_i$ for $1 \leq i \leq r$ and $s_i/s_j \not \in p^{\mathbb{Z}}$ for each $1 \leq i <j \leq r$, and recall that polynomials $u_1,\,\dots,\,u_r \in \overline{K}[t]$ are chosen so that we have Inequalities \eqref{conditiononui_firststep1} and \eqref{conditiononui_firststep2} for each $1 \leq i \leq r$. Let us prove that the algebraic group $\Gamma_{M(n)}$ equals to $G(n)$ by induction on $n$. 
To begin with, we note that Chang and Yu~(\cite[\S\S  4.2]{Chang2007}) obtained the equality $\Gamma_{M(0)}=G(0)$. Indeed, as they proved
\begin{equation}
    \trdeg_{\overline{K}}\overline{K}(\tilde{\pi},\,\mathcal{L}_{(u_1),\,(s_1)}|_{t=\theta},\,\dots,\,\mathcal{L}_{(u_r),\,(s_r)}|_{t=\theta})=1+r,
\end{equation}
it follows from Theorem \ref{ThmPapanikolas} that $\dim \Gamma_{M(0)}=1+r=\dim G(0)$. Therefore, the closed immersion in Proposition \ref{firststeponesideinclusion} is an isomorphism for $n=0$, hence the equalities
\begin{align}
    \Gamma_{M(0)}&=G(0)=\left\{ (a_0)\oplus \begin{pmatrix}
        a_0^{s_1}&0\\
        a_0^{s_1}x_1&1
    \end{pmatrix} \oplus \cdots \oplus \begin{pmatrix}
        a_0^{s_r}&0\\
        a_0^{s_r}x_r&1
    \end{pmatrix}\,\middle | \,
    a_0 \neq 0
    \right\} \text{ and}\label{GM0}\\
    \mathcal{U}_{M(0)}&=U(0)=\left\{ (1)\oplus \begin{pmatrix}
        1&0\\
        x_1&1
    \end{pmatrix} \oplus \cdots \oplus \begin{pmatrix}
        1&0\\
        x_r&1
    \end{pmatrix}
    \right\} \simeq \mathbb{G}_a^r. \label{UM0}
\end{align}

Let us consider the case where $n \geq 1$ and write $V(n)$ for the kernel of the homomorphism $U(n) \twoheadrightarrow \mathcal{U}_{\rho_{n-1} C}$ given by 
\begin{equation}
    \begin{pmatrix}
   1 & 0 &\cdots&\cdots&0\\
   a_1 & 1&0&\cdots&0\\
   a_2 & a_1& 1&\ddots&\vdots\\
   \vdots&\vdots& \ddots&\ddots&0 \\
   a_{n}& a_{n-1}&\cdots &a_1&1\\
\end{pmatrix}
        \otimes \bigoplus_{i=1}^r 
        \begin{pmatrix}
            1& 0\\ 
             x_i& 1
        \end{pmatrix} \mapsto \begin{pmatrix}
   1 & 0 &\cdots&\cdots&0\\
   a_1 & 1&0&\cdots&0\\
   a_2 & a_1& 1&\ddots&\vdots\\
   \vdots&\vdots& \ddots&\ddots&0 \\
   a_{n-1}& a_{n-2}&\cdots &a_1&1\\
\end{pmatrix}
\end{equation}
The kernel $V(n)$ can be described as follows:
\begin{equation}
   V(n)= \left \{ \begin{pmatrix}
   1 & 0 &\cdots&\cdots&0\\
   0& 1&0&\cdots&0\\
   \vdots& \ddots&\ddots&\ddots&\vdots \\
   0& 0&\cdots&1&0\\
   a_n&0 &\cdots&0&1
\end{pmatrix}
\oplus\begin{pmatrix}
        1&0\\
        x_1&1
    \end{pmatrix} \oplus \cdots \oplus \begin{pmatrix}
        1&0\\
        x_r&1
    \end{pmatrix}
\right\}
\end{equation}
and hence we can obtain identification 
\begin{equation}\label{directproductV(n)}
    V(n)\simeq \mathcal{V}_{\rho_n C} \times U(0)
\end{equation}
in the evident way and also have
\begin{equation}\label{directproductV(n)2}
    V(n)\simeq \mathcal{V}_{\rho_n C} \times \mathcal{U}_{M(0)} 
\end{equation}
if we have Equation \eqref{UM0}, see also Equation \eqref{EqdescriptionVrhonC}.
We define $\mathcal{V}_{M(n)}$ to be the kernel of the surjection $\mathcal{U}_{M(n)} \twoheadrightarrow \mathcal{U}_{\rho_{n-1} C}$ given in Lemma \ref{Lemmadiagramchase}; then $\mathcal{V}_{M(n)}$ is regarded as an algebraic subgroup of $V(n)$, see the following diagram of exact sequences:
\begin{equation}    \begin{tikzpicture}[auto] 
                   \node (22) at (-1, 2) {$\mathcal{V}_{M(n)}$}; \node (42) at (3.5, 2) {$\mathcal{U}_{M(n)}$}; \node (62) at (8, 2) {$\mathcal{U}_{\rho_{n-1} C}$};
                    \node (20) at (-1, 0) {$V(n)$}; \node (40) at (3.5, 0) {$U(n)$}; \node (60) at (8, 0) {$\mathcal{U}_{\rho_{n-1} C}$};
    
                    \draw[-] (62) to node[sloped] {$=$}(60);
                    \draw[{Hooks[right]}->] (42) to node {}(40);
                    \draw[{Hooks[right]}->] (22) to node {}(20);

                    \draw[{Hooks[right]}->] (22) to node {}(42);
                    \draw[{Hooks[right]}->] (20) to node {}(40);
    
                    \draw[->>] (42) to node {}(62);
                    \draw[->>] (40) to node {}(60);

                \end{tikzpicture}.
            \end{equation}

In order to show that the inclusion $\Gamma_{M(n)}\hookrightarrow G(n)$ in Proposition \ref{firststeponesideinclusion} is an isomorphism, we will verify the equality $\mathcal{V}_{M(n)} = V(n)$ in the case where positive integers $s_1,\,\dots,\,s_r$ are chosen so that $(q-1) \nmid s_i$ for $1 \leq i \leq r$ and $s_i/s_j \not \in p^{\mathbb{Z}}$ for each $1 \leq i <j \leq r$.
One of our primary tools is the following lemma:
\begin{Lemma}\label{LemmaAction}
     Consider positive integers $s_1,\,\dots,\,s_r$ and algebraic group $V$ isomorphic to $\mathbb{G}_a^{r+1}$ equipped with the $\mathbb{G}_m$-action given by
     \begin{equation}
         a.(x_0,\,x_1,\,\dots,\,x_r):=(x_0,\,a^{s_1}x_1,\,\dots,\,a^{s_r}x_r)
     \end{equation}
     for any $a \in \overline{\mathbb{F}_q(t)}^\times$ and $x_0,\,x_1,\,\dots,\,x_r \in \overline{\mathbb{F}_q(t)}$.
     Let $W$ be an algebraic subgroup of $V$ such that the group $W(\overline{\mathbb{F}_q(t)})$ is closed under the $\mathbb{G}_m(\overline{\mathbb{F}_q(t)})$-action.
     Assume that the restriction of the $0$-th projection $V \twoheadrightarrow \mathbb{G}_a$ to the algebraic subgroup $W$ is sujective and also that $W$ is mapped onto $\mathbb{G}_a^r$ by the projection $(x_0,\,x_1,\,\dots,\,x_r) \mapsto (x_1,\,x_2,\,\dots,\,x_r)$.
     If the ratio $s_i/s_j$ is not an integer power of $p$ for each $1 \leq i <j \leq r$, then it follows that $W=V$.
    \end{Lemma}
    \begin{proof}
        As we have a surjection $W \twoheadrightarrow \mathbb{G}_a^r$ by the assumption, the codimension of $W$ in $V$ is $0$ or $1$.
        We recall the result of Conrad~\cite[Corollary 1.8]{Conrad2015} on the structure of direct products of $\mathbb{G}_a$.
        His result gives us polynomials $P_0(X),\,P_1(X),\,\dots,\,P_r(X) \in \overline{\mathbb{F}_q(t)}[X]$ such that the algebraic subset $W(\overline{\mathbb{F}_q(t)})$ is given as follows:
            \begin{equation}
                \{ (x_0,x_1,\,\dots,\,x_r) \in V(\overline{\mathbb{F}_q(t)})
                \mid P_0(x_0)+P_1(x_1)+\cdots+P_r(x_r)=0\}.
            \end{equation}
        Moreover, we may assume that the polynomial $P_i(X)$ is an additive polynomial for $0 \leq i \leq r$,  that is, the polynomial $P_i(X)$ is of the form
            \begin{equation}
               P_i(X)=b_0X^{p^0}+b_1X^{p^1}+\cdots+b_{m_i}X^{p^{m_i}} 
            \end{equation}
            for some $b_0,\,\dots,\,b_{m_i}\in \overline{\mathbb{F}_q(t)}$. 

            As $W$ is mapped onto $\mathbb{G}_a^r$ by the projection $V \twoheadrightarrow \mathbb{G}_a^r$ given by $(x_0,\,x_1,\,\dots,\,x_r) $ $\mapsto (x_1,\,x_2,\,\dots,\,x_r)$, we can pick $b \in \overline{\mathbb{F}_q(t)}$ such that $(b,\,1,\,\dots,\,1) \in W (\overline{\mathbb{F}_q(t)})$. 
            For any element $a \in \overline{\mathbb{F}_q(t)}$, we have $$a.(b,\,1,\,\dots,\,1) =(b,\,a^{s_1},\,a^{s_2},\,\dots,\,a^{s_r})\in W (\overline{\mathbb{F}_q(t)})$$ as $W (\overline{\mathbb{F}_q(t)})$ was assumed to be closed under the $\mathbb{G}_m(\overline{\mathbb{F}_q(t)})$-action and hence it follows that
            \begin{equation}\label{polynomina}
                P_0(b)+P_1(a^{s_1})+\cdots+P_r(a^{s_r})=0.
            \end{equation}
            Since we arbitrarily chose $a$ from the infinite set $\overline{\mathbb{F}_q(t)}^\times$, the left hand side of Equation \eqref{polynomina} must be a trivial as a polynomial in $a$.
            For $1 \leq i \leq r$, we have
            \begin{equation}
                \deg_a P_i(a^{s_i})=s_ip^{m_i}
            \end{equation}
            for some $m_i \geq 0$ if the polynomial $P_i(X)$ in $X$ is non-zero. As we assumed that $s_i/s_j \not \in p^{\mathbb{Z}}$ for any  $1 \leq i <j \leq r$, it follows that $$P_1(X)=P_2(X)=\cdots=P_r(X)=0,$$
            hence we have 
            \begin{equation}
                W(\overline{\mathbb{F}_q(t)})=\{(x_0,x_1,\,\dots,\,x_r) \in V(\overline{\mathbb{F}_q(t)})
                \mid P_0(x_0)=0\}.
            \end{equation}
            As we have assumed that $W$ is mapped onto $\mathbb{G}_a$ by the $0$-th projection, we also have $P_0(X)=0$. Therefore, we can conclude that $W=V$.
    \end{proof}

Recall that we have Equation \eqref{UM0} and the direct product decomposition $V(n)\simeq \mathcal{V}_{\rho_n C} \times \mathcal{U}_{M(0)}\simeq \mathbb{G}_a \times \mathbb{G}_a^r$ if we assume that positive integers $s_1,\,\dots,\,s_r$ are chosen so that $(q-1) \nmid s_i$ for $1 \leq i \leq r$ and $s_i/s_j \not \in p^{\mathbb{Z}}$ for each $1 \leq i <j \leq r$.
We will confirm that the inclusion $\mathcal{V}_{M(n)} \subset V(n)$ satisfies the conditions required in Lemma \ref{LemmaAction} in order to obtain the equality $\mathcal{V}_{M(n)}=V(n)$.
To begin with, we check that the restriction to $\mathcal{V}_{M(n)}$ of the projection $ \mathcal{V}_{\rho_n C} \times \mathcal{U}_{M(0)} \twoheadrightarrow \mathcal{V}_{\rho_n C}$ is surjective. 
This surjectivity follows from Lemma \ref{Lemmafirstproj} given below, which can be formulated and be verified without those assumptions on  positive integers $s_1,\,\dots,\,s_r$.

Recalling that we have explicitly described the structure of algebraic group $\mathcal{U}_{\rho_n C}$ in Equation \eqref{EqexplicitUrhonC} using a result of Maurischat~(\cite{Maurischat2018}), we can consider the homomorphism $\varphi:U(n) \twoheadrightarrow \mathcal{U}_{\rho_n C}$ given by 
\begin{equation}
    \begin{pmatrix}
   1 & 0 &\cdots&\cdots&0\\
   a_1 & 1&0&\cdots&0\\
   a_2 & a_1& 1&\ddots&\vdots\\
   \vdots&\vdots& \ddots&\ddots&0 \\
   a_{n}& a_{n-1}&\cdots &a_1&1\\
\end{pmatrix}
        \otimes \bigoplus_{i=1}^r 
        \begin{pmatrix}
            1& 0\\ 
             x_i& 1
        \end{pmatrix} \mapsto \begin{pmatrix}
   1 & 0 &\cdots&\cdots&0\\
   a_1 & 1&0&\cdots&0\\
   a_2 & a_1& 1&\ddots&\vdots\\
   \vdots&\vdots& \ddots&\ddots&0 \\
   a_{n}& a_{n-1}&\cdots &a_1&1\\
\end{pmatrix}.
\end{equation}
Then we also have $\varphi|_{V(n)}:V(n)\twoheadrightarrow \mathcal{V}_{\rho_n C}$ by the definition of algebraic groups $V(n)$ and Equation \eqref{EqdescriptionVrhonC}.
We note that the homomorphism $\varphi|_{V(n)}:V(n)\twoheadrightarrow \mathcal{V}_{\rho_n C}$ coincides with the projection $ \mathcal{V}_{\rho_n C} \times \mathcal{U}_{M(0)} \twoheadrightarrow \mathcal{V}_{\rho_n C}$ if we have Equation \eqref{directproductV(n)}.

\begin{Lemma}\label{Lemmafirstproj}
The algebraic subgroup $\mathcal{V}_{M(n)} \subset V(n) $ is mapped onto $\mathcal{V}_{\rho_n C}$ by the homomorphism $\varphi|_{V(n)}:V(n)\twoheadrightarrow \mathcal{V}_{\rho_n C}$.
\end{Lemma}
\begin{proof}
    We have obtained the closed immersion of $\Gamma_{M(n)}$ into $ G(n)$ in Proposition \ref{firststeponesideinclusion} and we write $\iota$ for the restriction $\mathcal{U}_{M(n)} \hookrightarrow U(n)$.
   Using Lemma \ref{LemmaTannakianProj}, we can show that the composition $\varphi\circ \iota:\mathcal{U}_{M(n)} \rightarrow \mathcal{U}_{\rho_n C}$ coincides with the surjection given in Lemma \ref{Lemmadiagramchase}. 
    By considering the following commutative diagram of exact sequences, we can show that the restriction $\varphi|_{\mathcal{V}_{M(n)}}$ to the algebraic subgroup $\mathcal{V}_{M(n)}$ is surjective onto $\mathcal{V}_{\rho_n C}$:
    \begin{equation}
        \begin{tikzpicture}[auto] 
                   \node (22) at (0, 1.5) {$\mathcal{V}_{M(n)}$}; \node (42) at (4, 1.5) {$\mathcal{U}_{M(n)}$}; \node (62) at (8, 1.5) {$\mathcal{U}_{\rho_{n-1} C}$};
                    \node (20) at (0, 0) {$V(n)$}; \node (40) at (4, 0) {$U(n)$}; \node (60) at (8, 0) {$\mathcal{U}_{\rho_{n-1} C}$};
                    \node (2-2) at (0, -1.5) {$\mathcal{V}_{\rho_n C}$}; \node (4-2) at (4, -1.5) {$\mathcal{U}_{\rho_n C}$}; \node (6-2) at (8, -1.5) {$\mathcal{U}_{\rho_{n-1} C}$};
                    ;

                    \draw[-] (62) to node[sloped] {$=$}(60);
                    \draw[-] (60) to node[sloped] {$=$}(6-2);
                    
                    \draw[->>] (40) to node {$\varphi$}(4-2);

                    \draw[{Hooks[right]}->] (42) to node {}(40);
                    \draw[{Hooks[right]}->] (22) to node {}(20);

                    \draw[{Hooks[right]}->] (22) to node {}(42);
                    \draw[{Hooks[right]}->] (20) to node {}(40);
                    \draw[{Hooks[right]}->] (2-2) to node {}(4-2);
    
                    \draw[->>] (42) to node {}(62);
                    \draw[->>] (40) to node {}(60);
                    \draw[->>] (4-2) to node {}(6-2);

                    \draw[->>] (20) to node {$\varphi \mid_{V(n)}$}(2-2);

                    \draw[white,line width=6pt] (42) to [bend right=50] node { } (4-2);
                    \draw[->] (42) to [bend right=50] node {} (4-2);
                    \draw[-] (4-2) to [bend left=50] node {$\begin{matrix}\\ \\ \varphi\circ \iota\end{matrix}$} (42);
                \end{tikzpicture}.
    \end{equation} 
\end{proof}

We assume that positive integers $s_1,\,\dots,\,s_r$ are chosen so that $(q-1) \nmid s_i$ for $1 \leq i \leq r$ and $s_i/s_j \not \in p^{\mathbb{Z}}$ for each $1 \leq i <j \leq r$.
By the assumtion on positive integers $s_1,\,\dots,\,s_r$, we have direct product decomposition of the algebraic group $V(n)$ as in Equation \eqref{directproductV(n)2}.
We show also that the algebraic subgroup $\mathcal{V}_{M(n)} \subset V(n)\simeq \mathcal{V}_{\rho_n C} \times \mathcal{U}_{M(0)}$ is mapped onto $ \mathcal{U}_{M(0)}$ by the projection. In order to do that, we need the induction hypothesis, which asserts that the closed immersion $\Gamma_{M(n-1)} \hookrightarrow G(n-1)$ in Proposition \ref{firststeponesideinclusion} is an isomorphism. 
\begin{Lemma}\label{lemmareminderproj}
    Assume that $(q-1) \nmid s_i$ for $1 \leq i \leq r$ and $s_i/s_j \not \in p^{\mathbb{Z}}$ for each $1 \leq i <j \leq r$.
    If the algebraic group $\Gamma_{M(n-1)}$ equals to $G(n-1)$, then the restriction to the subgroup $\mathcal{V}_{M(n)}$ of the projection $V(n) \simeq \mathcal{V}_{\rho_n C} \times \mathcal{U}_{M(0)}\twoheadrightarrow \mathcal{U}_{M(0)} $ is surjective for $n \geq 1$.
\end{Lemma}
\begin{proof}
    As we have assumed that $\Gamma_{M(n-1)}=G(n-1)$, 
    we can consider the homomorphism $\overline{\psi}:G(n)\twoheadrightarrow \Gamma_{M(n-1)}$ given by
    \begin{align}
        &\begin{pmatrix}
   a_0 & 0 &\cdots&\cdots&0\\
   a_1 & a_0&0&\cdots&0\\
   a_2 & a_1& a_0&\ddots&\vdots\\
   \vdots&\vdots& \ddots&\ddots&0 \\
   a_{n}& a_{n-1}&\cdots &a_1&a_0\\
\end{pmatrix}
        \oplus \begin{pmatrix}
        a_0^{s_1}&0\\
        a_0^{s_1}x_1&1
    \end{pmatrix} \oplus \cdots \oplus \begin{pmatrix}
        a_0^{s_r}&0\\
        a_0^{s_r}x_r&1
    \end{pmatrix}\\ 
        &\quad \mapsto \begin{pmatrix}
   a_0 & 0 &\cdots&\cdots&0\\
   a_1 & a_0&0&\cdots&0\\
   a_2 & a_1& a_0&\ddots&\vdots\\
   \vdots&\vdots& \ddots&\ddots&0 \\
   a_{n-1}& a_{n-2}&\cdots &a_1&a_0\\
\end{pmatrix}\oplus \begin{pmatrix}
        a_0^{s_1}&0\\
        a_0^{s_1}x_1&1
    \end{pmatrix} \oplus \cdots \oplus \begin{pmatrix}
        a_0^{s_r}&0\\
        a_0^{s_r}x_r&1
    \end{pmatrix}
    \end{align}
    and its restriction $\psi:U(n)\twoheadrightarrow \mathcal{U}_{M(n-1)}$ since the equality $\mathcal{U}_{M(n-1)}=U(n-1)$ follows from $\Gamma_{M(n-1)}=G(n-1)$.
    We write $\overline{\iota}$ for the closed immersion of $\Gamma_{M(n)} $ into $ G(n)$ given in Proposition \ref{firststeponesideinclusion} and write $\iota:\mathcal{U}_{M(n)} \hookrightarrow U(n)$ for the restriction.
    Let us consider the following commutative diagram:
    \begin{equation}
        \begin{tikzpicture}[auto] 
                   \node (22) at (0, 1.5) {$\mathcal{U}_{M(n)}$}; \node (42) at (4, 1.5) {$\Gamma_{ M(n)}$}; \node (62) at (8, 1.5) {$\Gamma_C$};
                    \node (20) at (0, 0) {$U(n)$}; \node (40) at (4, 0) {$G(n)$}; \node (60) at (8, 0) {$\mathbb{G}_m$};
                    \node (2-2) at (0, -1.5) {$\mathcal{U}_{M(n-1)}$}; \node (4-2) at (4, -1.5) {$\Gamma_{M(n-1)}$}; \node (6-2) at (8, -1.5) {$\Gamma_C$};
                   
                    \draw[-] (62) to node[sloped] {$=$}(60);
                    \draw[-] (60) to node[sloped] {$=$}(6-2);
                    
                    \draw[{Hooks[right]}->] (42) to node {$\overline{\iota}$}(40);
                    \draw[{Hooks[right]}->] (22) to node {$\iota$}(20);

                    \draw[{Hooks[right]}->] (22) to node {}(42);
                    \draw[{Hooks[right]}->] (20) to node {}(40);
                    \draw[{Hooks[right]}->] (2-2) to node {}(4-2);
                    
                    \draw[->>] (42) to node {}(62);
                    \draw[->>] (40) to node {}(60);
                    \draw[->>] (4-2) to node {}(6-2);
                    
                    \draw[->>] (20) to node {$\psi$}(2-2);
                    \draw[->>] (40) to node {$\overline{\psi}$}(4-2);

                    \draw[white,line width=6pt] (42) to [bend right=50] node { } (4-2);
                    \draw[->] (42) to [bend right=50] node {} (4-2);
                    \draw[-] (4-2) to [bend left=50] node {$\begin{matrix}\\ \\ \\ \overline{\psi}\circ \overline{\iota} \end{matrix}$} (42);
                    \draw[->] (22) to [bend right=50] node {} (2-2);
                    \draw[-] (2-2) to [bend left=50] node {$\psi\circ \iota$} (22);
                \end{tikzpicture}.
    \end{equation}
    Lemma \ref{LemmaTannakianProj} shows that the composition $\overline{\psi}\circ \overline{\iota}:\Gamma_{M(n)} \rightarrow \Gamma_{M(n-1)}$ in the diagram coincides with the surjection given by Tannakian duality; note that $M(n-1)$ is a sub-pre-$t$-motive of $M(n)$. Hence it follows from the commutativity of the above diagram that the composition $\psi\circ \iota$
    is the surjective onto $\mathcal{U}_{M(n-1)}$. 
    
    By assumptions, the algebraic groups $\mathcal{U}_{M(i)}$ are equal to algebraic groups $U(i)$ with the explicit description in Equation \eqref{EqexplicitU(n)} for $i=0$ and $i=n-1$. Thus the kernel of the surjection $\mathcal{U}_{M(n-1)} \rightarrow \mathcal{U}_{\rho_{n-1} C}$ given in Lemma \ref{Lemmadiagramchase} can be identified with $\mathcal{U}_{M(0)}$ in the evident manner, see Equation \eqref{EqexplicitU(n)} for the structure of the algebraic group $\mathcal{U}_{M(n-1)} \simeq U(n-1)$ and see Example \ref{Maurischatresult} for that of $\mathcal{U}_{\rho_{n-1} C}$.
    We consider the following commutative diagram, whose horizontal lines are exact:
    \begin{equation}
        \begin{tikzpicture}[auto] 
                   \node (22) at (0, 1.5) {$\mathcal{V}_{M(n)}$}; \node (42) at (4, 1.5) {$\mathcal{U}_{M(n)}$}; \node (62) at (8, 1.5) {$\mathcal{U}_{\rho_{n-1} C}$};
                    \node (20) at (0, 0) {$V(n)$}; \node (40) at (4, 0) {$U(n)$}; \node (60) at (8, 0) {$\mathcal{U}_{\rho_{n-1} C}$};
                    \node (2-2) at (0, -1.5) {$\mathcal{U}_{M(0)}$}; \node (4-2) at (4, -1.5) {$\mathcal{U}_{M(n-1)}$}; \node (6-2) at (8, -1.5) {$\mathcal{U}_{\rho_{n-1} C}$};
                    ;

                    \draw[-] (62) to node[sloped] {$=$}(60);
                    \draw[-] (60) to node[sloped] {$=$}(6-2);
                    
                    \draw[->>] (40) to node {$\psi$}(4-2);
                    
                    \draw[{Hooks[right]}->] (42) to node {$\iota$}(40);
                    \draw[{Hooks[right]}->] (22) to node {$\iota \mid_{\mathcal{V}_{M(n)}}$}(20);

                    \draw[{Hooks[right]}->] (22) to node {}(42);
                    \draw[{Hooks[right]}->] (20) to node {}(40);
                    \draw[{Hooks[right]}->] (2-2) to node {}(4-2);
                   
                    \draw[->>] (42) to node {}(62);
                    \draw[->>] (40) to node {}(60);
                    \draw[->>] (4-2) to node {}(6-2);
                    
                    \draw[->>] (20) to node {$\psi \mid_{V(n)}$}(2-2);
                    
                    \draw[white,line width=6pt] (42) to [bend right=50] node { } (4-2);
                    \draw[->] (42) to [bend right=50] node {} (4-2);
                    \draw[-] (4-2) to [bend left=50] node {$\begin{matrix}\\ \\ \psi\circ \iota\end{matrix}$} (42);
                \end{tikzpicture}.
    \end{equation}
    The composition $(\psi|_{V(n)}) \circ(\iota|_{\mathcal{V}_{M(n)}}) :\mathcal{V}_{M(n)} \rightarrow \mathcal{U}_{M(0)}$ is verified to be surjective by the diagram chasing.
    Let us recall the direct product decomposition in Equation \eqref{directproductV(n)2}.
    By the definition of $\psi$, we can show that the homomorphism $\psi|_{V(n)}$ coincides with the projection $\mathcal{V}_{\rho_n C} \times \mathcal{U}_{M(0)}\twoheadrightarrow \mathcal{U}_{M(0)} $, hence the desired surjectivity.
\end{proof}
Let us continue to assume that  $(q-1) \nmid s_i$ for $1 \leq i \leq r$ and $s_i/s_j \not \in p^{\mathbb{Z}}$ for each $1 \leq i <j \leq r$. Considering the conjugate action of $\Gamma_{M(n)}$ on $\mathcal{U}_{M(n)}$
given by the exact sequence
\begin{equation}\label{ss3exactseq}
    1 \hookrightarrow \mathcal{U}_{M(n)} \hookrightarrow \Gamma_{M(n)} \xrightarrow{\pi} \Gamma_C \twoheadrightarrow 1,
\end{equation}
we can show that the inclusion $\mathcal{V}_{M(n)} \subset V(n) \simeq \mathcal{V}_{\rho_n C} \times \mathcal{U}_{M(0)}$ satisfies the remaining assumption of Lemma \ref{LemmaAction} as follows:
\begin{Lemma}\label{LemmaActionfirststep}
    Assume that positive integers $s_1,\,\dots,\,s_r$ are chosen so that $(q-1) \nmid s_i$ for $1 \leq i \leq r$ and $s_i/s_j \not \in p^{\mathbb{Z}}$ for each $1 \leq i <j \leq r$, and identify the algebraic group $V(n)$ with the direct product $\mathcal{V}_{\rho_n C}\times \mathcal{U}_{M(0)}\simeq \mathbb{G}_a^{r+1}$ in the evident way (see Equation \eqref{directproductV(n)2}). 
    If the group $\mathbb{G}_m(\overline{\mathbb{F}_q(t)})$ acts on $$\big(V(n)\big) (\overline{\mathbb{F}_q(t)})\simeq \mathbb{G}_a^{r+1}(\overline{\mathbb{F}_q(t)})$$ by
    \begin{equation}
         a.(x_0,\,x_1,\,\dots,\,x_r):=(x_0,\,a^{s_1}x_1,\,\dots,\,a^{s_r}x_r)
     \end{equation}
     for any $a \in \overline{\mathbb{F}_q(t)}^\times$ and $x_1,\,\dots,\,x_r \in \overline{\mathbb{F}_q(t)}$, then the subgroup $\mathcal{V}_{M(n)}(\overline{\mathbb{F}_q(t)})$ of $\big(V(n)\big) (\overline{\mathbb{F}_q(t)})$ is closed under this action.
\end{Lemma}
\begin{proof}   The exact sequence \eqref{ss3exactseq}
shows that the group $\Gamma_C(\overline{\mathbb{F}_q(t)})$ of $\overline{\mathbb{F}_q(t)}$-valued points acts on $\mathcal{U}_{M(n)}(\overline{\mathbb{F}_q(t)})$ as follows: for $a \in\Gamma_C(\overline{\mathbb{F}_q(t)})$ and $R \in \mathcal{U}_n(\overline{\mathbb{F}_q(t)})$, we define $a . R:=\widetilde{a}^{-1}R\widetilde{a}$ where $\widetilde{a}$ is an arbitrary element in the preimage $\pi^{-1}(a)$, here $\overline{\mathbb{F}_q(t)}$ is a fixed algebraic closure of $\mathbb{F}_q(t)$. 
Recall that the $t$-motivic Galois group $\Gamma_C$ is isomorphic to the multiplicative group $\mathbb{G}_m$. If we identify the group $\Gamma_C(\overline{\mathbb{F}_q(t)})$ with $\mathbb{G}_m(\overline{\mathbb{F}_q(t)})$ and take arbitrary $a \in  \mathbb{G}_m(\overline{\mathbb{F}_q(t)}) =\overline{\mathbb{F}_q(t)} \setminus \{0\}$, then the action is explicitly described as follows:
\begin{small}
\begin{align}
  &\quad\   a.\begin{pmatrix}
   1 & 0 &\cdots&\cdots&\cdots&0\\
   a_1 & 1&0&\cdots&\cdots&0\\
   a_2 & a_1& 1&0&\cdots&0\\
   \vdots&\ddots& \ddots&\ddots&&\vdots \\
   a_{n-1}& a_{n-2}&\cdots &a_1&1&0\\
   a_n&a_{n-1}&\cdots &\cdots&a_1&1
\end{pmatrix}\oplus \bigoplus_{1\leq i\leq r}\begin{pmatrix}
1 & 0\\
 x_{i}& 1
\end{pmatrix}\\
&= Q^{-1}\begin{pmatrix}
   1 & 0 &\cdots&\cdots&0\\
   a_1 & 1&0&\cdots&0\\
   \vdots& \vdots&\ddots&&\vdots \\
   a_{n-1}& a_{n-2}&\cdots&1&0\\
   a_n&a_{n-1} &\cdots&a_1&1
\end{pmatrix}Q\oplus \bigoplus_{1 \leq i\leq r} \left\{ \begin{pmatrix}
a^{s_i} & 0\\
 b_j& 1
\end{pmatrix}^{-1}\begin{pmatrix}
1 & 0\\
 x_{i}& 1
\end{pmatrix}\begin{pmatrix}
a^{s_i} & 0\\
 b_i& 1
\end{pmatrix}\right\} \\
&=\begin{pmatrix}
   1 & 0 &\cdots&\cdots&\cdots&0\\
   a_1 & 1&0&\cdots&\cdots&0\\
   a_2 & a_1& 1&0&\cdots&0\\
   \vdots&\ddots& \ddots&\ddots&&\vdots \\
   a_{n-1}& a_{n-2}&\cdots &a_1&1&0\\
   a_n&a_{n-1}&\cdots &\cdots&a_1&1
\end{pmatrix}\oplus \bigoplus_{1 \leq i\leq r}\begin{pmatrix}
1 & 0\\
 a^{s_i}x_{i}& 1
\end{pmatrix}. \label{ExplicitAction}
\end{align}
\end{small}
Here, the lower triangular matrix $Q \in \Gamma_{\rho_n C} (\overline{\mathbb{F}_q(t)})$ and elements $b_1,\,\dots,\,b_j \in \overline{\mathbb{F}_q(t)}$ are chosen so that 
\begin{equation}
    Q \oplus \begin{pmatrix}
a^{s_1} & 0\\
 b_1& 1
\end{pmatrix}\oplus \cdots \oplus  \begin{pmatrix}
a^{s_i} & 0\\
 b_i& 1
\end{pmatrix} \in \pi^{-1}(a).
\end{equation}
Note also that the exact sequence
\begin{equation}
    1 \hookrightarrow \mathcal{U}_{\rho_{n-1} C}\hookrightarrow \Gamma_{\rho_{n-1}C}\twoheadrightarrow \Gamma_C \twoheadrightarrow 1
\end{equation}
yields the $\Gamma_C(\overline{\mathbb{F}_q(t)})$-action on $U_{n-1}(\overline{\mathbb{F}_q(t)})$ by the similar way (the action is trivial in fact since the group $\Gamma_{\rho_{n-1}C}(\overline{\mathbb{F}_q(t)})$ is commutative, see Example \ref{Maurischatresult}).
As Diagram \eqref{Diagramchase} shows that the surjection $\mathcal{U}_{M(n)} \twoheadrightarrow \mathcal{U}_{\rho_{n-1} C}$ is $\Gamma_C(\overline{\mathbb{F}_q(t)})$-equivariant,
it follows that the subgroup $\mathcal{V}_{M(n)}(\overline{\mathbb{F}_q(t)})$ is closed under the action on $\mathcal{U}_n(\overline{\mathbb{F}_q(t)})$ of $\Gamma_C(\overline{\mathbb{F}_q(t)})$ and Equation \eqref{ExplicitAction} shows the desired result.
\end{proof}
Now we are ready to apply Lemma \ref{LemmaAction} and verify the equality $\mathcal{V}_{M(n)}=V(n)$.
We can obtain the following result on $t$-motivic Galois groups of pre-$t$-motive $M(n)$.
\begin{Proposition}\label{PropMainBfirststep}
    Let us take an index $(s_1,\,\dots,\,s_r)\in \mathbb{Z}_{>0}^r$. Assume that $(q-1) \nmid s_i$ for each $1 \leq i \leq r$ and also that the ratio $s_i/s_j$ is not an integer power of $p$ for each $i \neq j$.
    For given polynomials $u_1,\,\dots,\,u_r$ in $\overline{K}[t]$ such that \begin{equation}
    ||u_i||_\infty<|\theta|_\infty^{\frac{s_i q}{q-1}}
    \end{equation}
    and $\mathcal{L}_{(u_i),\,(s_i)}(t)|_{t=\theta}\neq 0$ hold for each $1 \leq i \leq r$,
    we have 
    \begin{equation}\label{firststepGroupeq}
        \Gamma_{\Psi(n)} = G(n) 
    \end{equation}
    for $n \geq0$.
\end{Proposition}
\begin{proof}
  We prove the proposition by the induction on $n$. We recall that Chang and Yu~\cite{Chang2007} proved that $\Gamma_{M(0)}=G(0)$. We consider the case $n \geq 1$.
  If we assume that $\Gamma_{M(n-1)}=G(n-1)$ holds, then Lemmas \ref{Lemmafirstproj}, \ref{lemmareminderproj}, and \ref{LemmaActionfirststep} enable us to apply Lemma \ref{LemmaAction} and to obtain $\mathcal{V}_{M(n)}=V(n) \simeq \mathbb{G}_a^{r+1}$. Now, the exact sequence 
  \begin{equation}
    1 \hookrightarrow \mathcal{V}_{M(n)}\hookrightarrow \mathcal{U}_{M(n) }\twoheadrightarrow \mathcal{U}_{\rho_{n-1}}\twoheadrightarrow 1
\end{equation}
shows that $\dim \mathcal{U}_{M(n)}=\dim \mathcal{V}_{M(n)}+ \dim \mathcal{U}_{\rho_{n-1}}=(r+1)+(n-1)=n+r$ and the exact sequence
\begin{equation}
    1 \hookrightarrow \mathcal{U}_{M(n)} \hookrightarrow \Gamma_{M(n)} \twoheadrightarrow \Gamma_C \twoheadrightarrow 1
\end{equation}
implies $\dim \Gamma_{M(n)}= \dim \mathcal{U}_{M(n)}  + \dim \Gamma_C=n+1+r$. Hence we have $\dim \Gamma_{M(n)} =\dim G(n)$ and it follows that the closed immersion in Proposition \ref{firststeponesideinclusion} is an isomorphism as $G(n)$ is smooth and connected.
\end{proof}
Equality \eqref{firststepGroupeq} of algebraic groups corresponds via Papanikolas' theory (Theorem \ref{ThmPapanikolas}) to the following algebraic independence. 
This might be seen as a partial result of Theorem \ref{MainBpartial}.
\begin{Corollary}
    Take an index $(s_1,\,\dots,\,s_r)\in \mathbb{Z}_{>0}^r$. We assume that $(q-1) \nmid s_i$ for each $1 \leq i \leq r$ and the ratio $s_i/s_j$ is not an integer power of $p$ for each $i \neq j$.
    Let us consider the Taylor expansion of $\Omega$ (see Example \ref{Omega}):
\begin{equation}
    \Omega=\sum \alpha_n (t-\theta)^n.
\end{equation}
    \begin{enumerate}
    \item
        For given polynomials $u_1,\,\dots,\,u_r$ in $\overline{K}[t]$ such that \begin{equation}
    ||u_i||_\infty<|\theta|_\infty^{\frac{s_i q}{q-1}}
    \end{equation}
    and $\mathcal{L}_{(u_i),\,(s_i)}(t)|_{t=\theta}\neq 0$ for each $1 \leq i \leq r$,
    we consider the Taylor expansions
        \begin{equation}
        \mathcal{L}_{(u_{i}),(s_{i})}(t)=\sum_{n=0}^{\infty} \alpha_{i,\,n}(t-\theta)^n
    \end{equation}
    of $t$-motivic CMPL for each $1 \leq i \leq r$.  Then the field 
    \begin{equation}
        \overline{K}( \alpha_{n^\prime},\, \alpha_{i,\,0} \mid 0\leq n^\prime  \leq n,  1 \leq i \leq r)
    \end{equation}
     has a transcendental degree $n+1+r$ over $\overline{K}$ for each $n\geq0$.
    \item 
        Considering the Taylor expansions
\begin{equation}
     \zeta_A^{\mathrm{AT}}\left((s_i)\right)=\sum \beta_{i,\,n}(t-\theta)^n
\end{equation}
of Anderson-Thakur series for $1 \leq i \leq r$, we have
    \begin{equation}
        \trdeg_{\overline{K}}\overline{K}( \alpha_{n^\prime} ,\, \beta_{i,\,0} \mid 0\leq n^\prime  \leq n,\,1 \leq i \leq r)=n+1+r.
    \end{equation}
    \end{enumerate}
\end{Corollary}

\begin{proof}
    By Theorem \ref{ThmPapanikolas} and Proposition \ref{Uchino}, we have
    \begin{align}
       &\quad\ \trdeg_{\overline{K}}\overline{K}( \alpha_{n^\prime},\, \alpha_{i,\,0} \mid 0\leq n^\prime  \leq n,  1 \leq i \leq r)\\
       &=\trdeg_{\overline{K}}\overline{K}( \partial^{(n^\prime)}\Omega ,\, \mathcal{L}_{(u_i),\,(s_i)}|_{t=\theta} \mid 0\leq n^\prime  \leq n,  1 \leq i \leq r)\\
       &=\dim \Gamma_{M(n)}=\dim G(n)=n+1+r.
    \end{align}
    The assertion (2) is proven if we put $u_i$ to be the Anderson-Thakur polynomial $H_{s_i-1}$.
\end{proof}

\section{Taylor coefficients of the power series $\Omega$ and of Anderson-Thakur series}\label{sectionproof}
The main goal in this section is to study transcendence properties of Taylor coefficients of $t$-motivic CMPLs and to prove Theorem \ref{MainB}. 
We construct certain pre-$t$-motives $M(i,\,m)$ and observe in Proposition \ref{PropDimTrans} that the calculations of their $t$-motivic Galois groups $\Gamma_{M(i,\,m)}$ give us the desired algebraic independence result Theorem \ref{MainB}.
In preparation for the calculation of $t$-motivic Galois groups, 
we construct explicit algebraic varieties $G_{i,\,m}$, which will be proven in fact to equal the $t$-motivic Galois group $\Gamma_{M(i,\,m)}$ in the end of this section.

\subsection{Simple examples}\label{Subsectionsimplestexample}

In order to help readers to follow the calculation for the general cases, we treat with some special cases.
Let us take positive integers $s_1$ and $s_2$ such that $(q-1) \nmid s_i$ for $i = 1,\,2$ and $s_1/s_2 \not \in p^{\mathbb{Z}}$. 
We also take polynomials $u_1,\,u_2$ in $\overline{K}[t]$ such that 
\begin{equation}
    ||u_i||_\infty<|\theta|_\infty^{\frac{s_i q}{q-1}}
\end{equation}
for $i=1,\,2$.
We consider $t$-motivic CMPL's $\mathcal{L}_1:=\mathcal{L}_{(u_1),\,(s_1)}$, $\mathcal{L}_2:=\mathcal{L}_{(u_2),\,(s_2)}$, and $\mathcal{L}_3:=\mathcal{L}_{(u_1,\,u_2),\,(s_1,\,s_2)}$ and their Taylor expansions
\begin{equation}
    \mathcal{L}_i=\sum_{n\geq 0}\beta_{i,\,n}(t-\theta)^n,\quad \text{($i=1,\,2,\,3$)}.
\end{equation}
We assume $\mathcal{L}_i|_{t=\theta}\neq 0$ for $i=1,\,2$. 
Let us recall that the power series $\Omega$ was defined in Example \ref{Omega} and consider its Taylor expansion
\begin{equation}
    \Omega(t)=\sum \alpha_n (t-\theta)^n.
\end{equation}

We let $\Phi(2,\,0)$ to be the matrix
\begin{equation}
    \begin{pmatrix}
        t-\theta& 0 &0\\
        1& t-\theta &0\\
        0&1&t-\theta
    \end{pmatrix}
    \oplus
    \begin{pmatrix}
        (t-\theta)^{s_1}&0\\
        (t-\theta)^{s_1} u_1^{(-1)}&1
    \end{pmatrix}
    \oplus
    \begin{pmatrix}
        (t-\theta)^{s_2}&0\\
        (t-\theta)^{s_2} u_2^{(-1)}&1
    \end{pmatrix}.
\end{equation}
and let $M(2,\,0)$ be the pre-$t$-motive represented by $\Phi(2,\,0)$. We note that $M(2,\,0)$ is equal to the pre-$t$-motive $M(2)$ in Section \ref{sectionfirststep} and therefore we have
\begin{equation}\label{Eqsimplestexample0}
    \Gamma_{M(2,\,0)}=G_{2,\,0}:=
    \left\{ \begin{pmatrix}
        a_0&0&0\\
        a_1&a_0&0\\
        a_2&a_1&a_0
    \end{pmatrix}
    \oplus
    \begin{pmatrix}
        a_0^{s_1}&0\\
        a_0^{s_1}x_1&1
    \end{pmatrix}
    \oplus
    \begin{pmatrix}
        a_0^{s_2}&0\\
        a_0^{s_2}x_2&1
    \end{pmatrix}
    \, \middle| \,
    a_0 \neq 0
    \right\}
\end{equation}
by Proposition \ref{PropMainBfirststep}, hence the algebraic independence
\begin{equation}
        \trdeg_{\overline{K}}\overline{K}(\alpha_0,\,\alpha_1,\,\alpha_2,\,\beta_{1,\,0},\,\beta_{2,\,0})=5.
    \end{equation}

\begin{example}\label{simplestexample1}
    Our goal is to obtain
    \begin{equation}\label{Eqsimplestexample1}
        \trdeg_{\overline{K}}\overline{K}(\alpha_0,\,\alpha_1,\,\alpha_2,\,\beta_{1,\,0},\,\beta_{2,\,0},\,\beta_{3,\,0})=6.
    \end{equation}
We put
\begin{align}
    \Phi(3,\,0)&:=
    \begin{pmatrix}
        t-\theta& 0 &0\\
        1& t-\theta &0\\
        0&1&t-\theta
    \end{pmatrix}
    \oplus
    \begin{pmatrix}
        (t-\theta)^{s_1}&0\\
        (t-\theta)^{s_1} u_1^{(-1)}&1
    \end{pmatrix}\\
    & \quad \ \ \oplus
    \begin{pmatrix}
        (t-\theta)^{s_2}&0\\
        (t-\theta)^{s_2} u_2^{(-1)}&1
    \end{pmatrix}
    \oplus 
    \begin{pmatrix}
        (t-\theta)^{s_1+s_2}& 0&0\\
        (t-\theta)^{s_1+s_2}u_1^{(-1)}&(t-\theta)^{s_2}&0\\
        0&(t-\theta)^{s_2} u_2^{(-1)}&1
    \end{pmatrix}
\end{align}
and let $M(3,\,0)$ be the pre-$t$-motive defined by this matrix. The matrix $\Phi(3,\,0) \in \operatorname{GL}_{10}(\overline{K}(t)) \cap \operatorname{Mat}_{10}(\overline{K}[t])$ has a rigid analytic trivialization
\begin{align}
    \Psi(3,\,0)&:=
    \begin{pmatrix}
        \Omega& 0 &0\\
        \partial^{(1)} \Omega& \Omega &0\\
        \partial^{(2)} \Omega&\partial^{(1)} \Omega&\Omega
    \end{pmatrix}
    \oplus
    \begin{pmatrix}
        \Omega^{s_1}&0\\
        \Omega^{s_1} \mathcal{L}_1&1
    \end{pmatrix}\\
    & \quad \ \ \oplus
    \begin{pmatrix}
        \Omega^{s_2}&0\\
        \Omega^{s_2} \mathcal{L}_2&1
    \end{pmatrix}
    \oplus 
    \begin{pmatrix}
        \Omega^{s_1+s_2}& 0&0\\
        \Omega^{s_1+s_2}\mathcal{L}_1&\Omega^{s_2}&0\\
        \Omega^{s_1+s_2}\mathcal{L}_3&\Omega^{s_2} \mathcal{L}_2&1
    \end{pmatrix}
\end{align}
 in $\operatorname{GL}_{10}(\mathbb{L}) \cap \operatorname{Mat}_{10}(\mathbb{T})$, see Section \ref{subsectionATseries}. We can show that the matrix $\widetilde{\Psi}(3,\,0)$ is an $\mathbb{L} \otimes_{\overline{K}(t)}\mathbb{L}$-valued point of $G_{3,\,0} \subset \operatorname{GL}_{10/\overline{\mathbb{F}_q(t)}}$ given by
\begin{small}\begin{equation}
    \left\{ \begin{pmatrix}
        a_0&0&0\\
        a_1&a_0&0\\
        a_2&a_1&a_0
    \end{pmatrix}
    \oplus
    \begin{pmatrix}
        a_0^{s_1}&0\\
        a_0^{s_1}x_1&1
    \end{pmatrix}
    \oplus
    \begin{pmatrix}
        a_0^{s_2}&0\\
        a_0^{s_2}x_2&1
    \end{pmatrix}
    \oplus
    \begin{pmatrix}
        a_0^{s_1+s_2}&0&0\\
        a_0^{s_1+s_2} x_1& a_0^{s_2}&0\\
        a_0^{s_1+s_2} x_3&a_0^{s_2}x_2&1
    \end{pmatrix}
    \, \middle| \,
    a_0 \neq 0
    \right\},
\end{equation}
\end{small}
hence we have $ \Gamma_{\Psi(3,\,0)}\subset G_{3,\,0}$ as $\Gamma_{\Psi(3,\,0)}$ was characterized to be the smallest closed subscheme of $\operatorname{GL}_{10/\overline{\mathbb{F}_q(t)}}$ which has $\widetilde{\Psi}(3,\,0)$ as an $\mathbb{L} \otimes_{\overline{K}(t)}\mathbb{L}$-valued point. We identify $\Gamma_{M(3,\,0)}$ with $\Gamma_{\Psi(3,\,0)}$ by Theorem \ref{ThmPapanikolas}.

As the pre-$t$-motive $M(2,\,0)$ is a direct summand of $M(3,\,0)$, we have a surjective homomorphism $\varphi_{3,\,0}:\Gamma_{M(3,\,0)}\twoheadrightarrow \Gamma_{M(2,\,0)}$, whose kernel is denoted by $\mathcal{V}_{3,\,0}$. Lemma \ref{LemmaTannakianProj} tells us that the surjection $\varphi_{3,\,0}$ is given by
\begin{align}
    \begin{pmatrix}
        a_0&0&0\\
        a_1&a_0&0\\
        a_2&a_1&a_0
    \end{pmatrix}
    \oplus
    \begin{pmatrix}
        a_0^{s_1}&0\\
        a_0^{s_1}x_1&1
    \end{pmatrix}
    \oplus
    \begin{pmatrix}
        a_0^{s_2}&0\\
        a_0^{s_2}x_2&1
    \end{pmatrix}
    \oplus
    \begin{pmatrix}
        a_0^{s_1+s_2}&0&0\\
        a_0^{s_1+s_2} x_1& a_0^{s_2}&0\\
        a_0^{s_1+s_2} x_3&a_0^{s_2}x_2&1
    \end{pmatrix}\\
    \mapsto 
    \begin{pmatrix}
        a_0&0&0\\
        a_1&a_0&0\\
        a_2&a_1&a_0
    \end{pmatrix}
    \oplus
    \begin{pmatrix}
        a_0^{s_1}&0\\
        a_0^{s_1}x_1&1
    \end{pmatrix}
    \oplus
    \begin{pmatrix}
        a_0^{s_2}&0\\
        a_0^{s_2}x_2&1
    \end{pmatrix}.
\end{align}
Hence the kernel $\mathcal{V}_{3,\,0}$ is a closed subvariety of 
\begin{equation}
    V_{3,\,0}:=\left \{
    I_3 \oplus I_2 \oplus I_2 \oplus
    \begin{pmatrix}
        1&0&0\\
        0&1&0\\
        x_3&0&1
    \end{pmatrix}
    \right \} \simeq \mathbb{G}_a.
\end{equation}

We take arbitrary $b \in \overline{\mathbb{F}_q(t)}$. Equation \eqref{Eqsimplestexample0} enables us to pick
\begin{equation}
    Q_b^{\prime}:=
    I_3 \oplus
    \begin{pmatrix}
        1&0\\
        b&1
    \end{pmatrix}
    \oplus I_2 \in \Gamma_{M(2,\,0)}(\overline{\mathbb{F}_q(t)}) 
\end{equation}
(by putting $a_0=1$, $x_1=b$, and $a_1=a_2=x_2=0$) and
\begin{equation}
    R^\prime:=I_3 \oplus I_2 \oplus
    \begin{pmatrix}
        1&0\\
        1&1
    \end{pmatrix}  \in \Gamma_{M(2,\,0)}(\overline{\mathbb{F}_q(t)}).
\end{equation}
By the surjectivity of $\varphi_{3,\,0}$, we can take elements $Q_b$ and $R$ of $\Gamma_{M(3,\,0)}(\overline{\mathbb{F}_q(t)})$ which are respectively mapped to $Q_b^\prime$ and $R^\prime$ by $\varphi_{3,\,0}$.
By the constructions, there exist elements $y_Q$ and $y_R$ of $\overline{\mathbb{F}_q(t)}$ such that
\begin{equation}
    Q_b=
    I_3 \oplus
    \begin{pmatrix}
        1&0\\
        b&1
    \end{pmatrix}
    \oplus I_2 \oplus
    \begin{pmatrix}
        1&0&0\\
        b&1&0\\
        y_Q&0&1
    \end{pmatrix}
\end{equation}
and
\begin{equation}
    R=I_3 \oplus I_2 \oplus
    \begin{pmatrix}
        1&0\\
        1&1
    \end{pmatrix}\oplus
    \begin{pmatrix}
        1&0&0\\
        0&1&0\\
        y_R&1&1
    \end{pmatrix}.
\end{equation}
We can calculate the commutator as follows:
\begin{equation}\label{commutator1}
    RQ_bR^{-1}Q_b^{-1}
    =I_3\oplus I_2 \oplus I_2 \oplus
    \begin{pmatrix}
        1&0&0\\
        0&1&0\\
        b&0&1
    \end{pmatrix} \in \mathcal{V}_{3,\,0}(\overline{\mathbb{F}_q(t)}),
\end{equation}
hence it follows that $\mathcal{V}_{3,\,0}=V_{3,\,0}$ as $b$ is arbitrarily chosen. Since $\Gamma_{M(2,\,0)}$ is equal to the algebraic set $G_{2,\,0}$, we obtain
\begin{equation}
    \dim \Gamma_{M(3,\,0)}=\dim \Gamma_{M(2,\,0)}+\dim \mathcal{V}_{3,\,0}= \dim G_{2,\,0}+ \dim V_{3,\,0}=6.
\end{equation}
Therefore, 
\begin{equation}
    \trdeg_{\overline{K}}\overline{K}(\alpha_0,\,\alpha_1,\,\alpha_2,\,\beta_{1,\,0},\,\beta_{2,\,0},\,\beta_{3,\,0})=\dim \Gamma_{M(3,\,0)}=6
\end{equation}
by Theorem \ref{ThmPapanikolas}. As $G_{3,\,0}$ is smooth and irreducible, we also have
\begin{equation}\label{Eqsimplestexample1group}
    \Gamma_{M(3,\,0)}=G_{3,\,0}.
\end{equation}
\end{example}

\begin{example}\label{simplestexample2}
    Let us further assume $p \nmid s_1$. Our goal here is to verify the equality
    \begin{equation}
        \trdeg_{\overline{K}}\overline{K}(\alpha_0,\,\alpha_1,\,\alpha_2,\,\beta_{1,\,0},\,\beta_{2,\,0},\,\beta_{3,\,0},\,\beta_{1,\,1})=7.
    \end{equation}
We define $\Phi(1,\,1) \in \operatorname{GL}_{12}(\overline{K}(t)) \cap \operatorname{Mat}_{12}(\overline{K}[t]) $ to be 
\begin{align}
    &\begin{pmatrix}
        t-\theta& 0 &0\\
        1& t-\theta &0\\
        0&1&t-\theta
    \end{pmatrix}
    \oplus
    \begin{pmatrix}
	(t-\theta)^s& 0&0&0\\
	(t-\theta)^{s} u_1^{(-1)} &1 &0&0\\
	\partial^{(1)}((t-\theta)^s) & 0&(t-\theta)^s&0\\
	\partial^{(1)}( (t-\theta)^{s} u_1^{(-1)}) &\partial^{(1)} ( 1 )&(t-\theta)^{s} u_1^{(-1)}&1 
\end{pmatrix} \\
    & \quad \ \ \oplus
    \begin{pmatrix}
        (t-\theta)^{s_2}&0\\
        (t-\theta)^{s_2} u_2^{(-1)}&1
    \end{pmatrix}
    \oplus 
    \begin{pmatrix}
        (t-\theta)^{s_1+s_2}& 0&0\\
        (t-\theta)^{s_1+s_2}u_1^{(-1)}&(t-\theta)^{s_2}&0\\
        0&(t-\theta)^{s_2} u_2^{(-1)}&1
    \end{pmatrix}
\end{align}
and consider the pre-$t$-motive $M(1,\,1)$ defined by $\Phi (1,\,1)$. The matrix $\Phi(1,\,1)$ has a rigid analytic trivialization
\begin{align}
    \Psi(1,\,1)&:=
    \begin{pmatrix}
        \Omega& 0 &0\\
        \partial^{(1)} \Omega& \Omega &0\\
        \partial^{(2)} \Omega&\partial^{(1)} \Omega&\Omega
    \end{pmatrix}
    \oplus
    \begin{pmatrix}
	\Omega^{s}& 0&0&0\\
	\Omega^{s} \mathcal{L}_1 &1 &0&0\\
	s\Omega^{s-1}\partial^{(1)}\Omega & 0&\Omega^{s}&0\\
	\partial^{(1)}( \Omega^{s}  \mathcal{L}_1 ) &0&\Omega^{s} \mathcal{L}_1 &1 
    \end{pmatrix}
    \\
    & \quad \,\ \oplus
    \begin{pmatrix}
        \Omega^{s_2}&0\\
        \Omega^{s_2} \mathcal{L}_2&1
    \end{pmatrix}
    \oplus 
    \begin{pmatrix}
        \Omega^{s_1+s_2}& 0&0\\
        \Omega^{s_1+s_2}\mathcal{L}_1&\Omega^{s_2}&0\\
        \Omega^{s_1+s_2}\mathcal{L}_3&\Omega^{s_2} \mathcal{L}_2&1
    \end{pmatrix} \in \operatorname{GL}_{12}(\mathbb{L}) \cap \operatorname{Mat}_{12}(\mathbb{T}).
\end{align}
By some calculation, we can see that the matrix $\widetilde{\Psi}(1,\,1)$ is in $G_{1,\,1}(\mathbb{L} \otimes_{\overline{K}(t)}\mathbb{L})$ where $G_{1,\,1} \subset \operatorname{GL}_{12/\overline{\mathbb{F}_q(t)}}$ is defined to be the algebraic subset consists of the matrices of the form
\begin{align}
     &\begin{pmatrix}
        a_0&0&0\\
        a_1&a_0&0\\
        a_2&a_1&a_0
    \end{pmatrix}
    \oplus
    \begin{pmatrix}
        a_0^{s_1}&0&0&0\\
        a_0^{s_1}x_1&1&0&0\\
        sa_0^{s-1}a_1&0&a_0^{s_1}&0\\
        x_4a_0^{s}+sx_1a_0^{s-1}a_1&0&a_0^{s_1}x_1&1
    \end{pmatrix}\\
    &\quad \oplus
    \begin{pmatrix}
        a_0^{s_2}&0\\
        a_0^{s_2}x_2&1
    \end{pmatrix}
    \oplus
    \begin{pmatrix}
        a_0^{s_1+s_2}&0&0\\
        a_0^{s_1+s_2} x_1& a_0^{s_2}&0\\
        a_0^{s_1+s_2} x_3&a_0^{s_2}x_2&1
    \end{pmatrix}
\end{align}
Therefore, the algebraic group $ \Gamma_{\Psi(1,\,1)}$ is a closed subvariety of $ G_{1,\,1}$ since $\Gamma_{\Psi(1,\,1)}$ is the smallest closed subscheme of $\operatorname{GL}_{12/\overline{\mathbb{F}_q(t)}}$ which has $\widetilde{\Psi}(1,\,1)$ as an $\mathbb{L} \otimes_{\overline{K}(t)}\mathbb{L}$-valued point. 
By virtue of Theorem \ref{ThmPapanikolas}, we identify $\Gamma_{M(1,\,1)}$ with $\Gamma_{\Psi(1,\,1)}$.

Since the pre-$t$-motive $M(1,\,1)$ has $M(3,\,0)$ as a sub-pre-$t$-motive, we have a faithfully flat homomorphism $\varphi_{1,\,1}:\Gamma_{M(1,\,1)}\twoheadrightarrow \Gamma_{M(3,\,0)}$. 
Because of Lemma \ref{LemmaTannakianProj},  surjection $\varphi_{3,\,0}$ is described as follow:
\begin{align}
    &\begin{pmatrix}
        a_0&0&0\\
        a_1&a_0&0\\
        a_2&a_1&a_0
    \end{pmatrix}
    \oplus
    \begin{pmatrix}
        a_0^{s_1}&0&0&0\\
        a_0^{s_1}x_1&1&0&0\\
        sa_0^{s-1}a_1&0&a_0^{s_1}&0\\
        x_4a_0^{s}+sx_1a_0^{s-1}a_1&0&a_0^{s_1}x_1&1
    \end{pmatrix}\\
    &\quad \oplus
    \begin{pmatrix}
        a_0^{s_2}&0\\
        a_0^{s_2}x_2&1
    \end{pmatrix}
    \oplus
    \begin{pmatrix}
        a_0^{s_1+s_2}&0&0\\
        a_0^{s_1+s_2} x_1& a_0^{s_2}&0\\
        a_0^{s_1+s_2} x_3&a_0^{s_2}x_2&1
    \end{pmatrix}\\
    &\mapsto
    \begin{pmatrix}
        a_0&0&0\\
        a_1&a_0&0\\
        a_2&a_1&a_0
    \end{pmatrix}
    \oplus
    \begin{pmatrix}
        a_0^{s_1}&0\\
        a_0^{s_1}x_1&1
    \end{pmatrix}
    \oplus
    \begin{pmatrix}
        a_0^{s_2}&0\\
        a_0^{s_2}x_2&1
    \end{pmatrix}
    \oplus
    \begin{pmatrix}
        a_0^{s_1+s_2}&0&0\\
        a_0^{s_1+s_2} x_1& a_0^{s_2}&0\\
        a_0^{s_1+s_2} x_3&a_0^{s_2}x_2&1
    \end{pmatrix}.
\end{align}
Hence the kernel $\mathcal{V}_{1,\,1}$ of $\varphi_{1,\,1}$ is a closed subvariety of 
\begin{equation}
    V_{1,\,1}:=\left \{
    I_3 \oplus 
    \begin{pmatrix}
	   1& 0&0&0\\
	   0 &1 &0&0\\
	   0 & 0&1&0\\
	   x_4&0&0 &1 
    \end{pmatrix}
    \oplus I_2 \oplus I_3
    \right \},
\end{equation}
which is isomorphic to $\mathbb{G}_a$.

Let us take arbitrary $b \in \overline{\mathbb{F}_q(t)}$. By Equation \eqref{Eqsimplestexample1group}, we can pick
\begin{equation}
    Q_b^{\prime}:=
    \begin{pmatrix}
        1&0&0\\
        b&1&0\\
        0&b&1
    \end{pmatrix} \oplus I_2 \oplus I_2 \oplus I_3 
    \in \Gamma_{M(3,\,0)}(\overline{\mathbb{F}_q(t)}) 
\end{equation}
by putting $a_0=1$, $a_1=b$, and $a_2=x_1=x_2=x_3=0$. We can also pick
\begin{equation}
    R^\prime:=I_3 \oplus 
    \begin{pmatrix}
        1&0\\
        1&1
    \end{pmatrix} \oplus I_2 \oplus 
    \begin{pmatrix}
        1&0&0\\
        1&1&0\\
        0&0&1
    \end{pmatrix}
\end{equation}
from $\Gamma_{M(3,\,0)}(\overline{\mathbb{F}_q(t)})$.
By the surjectivity of $\varphi_{1,\,1}$, there are $Q_b$ and $R$ in the group $\Gamma_{M(1,\,1)}(\overline{\mathbb{F}_q(t)})$ which are mapped to $Q_b^\prime$ and $R^\prime$ by $\varphi_{1,\,1}$, respectively.
By the constructions, there exist $y_Q$ and $y_R$ such that
\begin{equation}
    Q_b=
    \begin{pmatrix}
        1&0&0\\
        b&1&0\\
        0&b&1
    \end{pmatrix} \oplus
    \begin{pmatrix}
	   1& 0&0&0\\
	   0 &1 &0&0\\
	   s_1b & 0&1&0\\
	   y_Q&0&0 &1 
    \end{pmatrix} \otimes I_2 \otimes I_3
\end{equation}
and
\begin{equation}
    R=I_3 \oplus 
    \begin{pmatrix}
	   1& 0&0&0\\
	   1 &1 &0&0\\
	   0 & 0&1&0\\
	   y_R&0&1 &1 
    \end{pmatrix} \otimes I_2 \oplus
    \begin{pmatrix}
        1&0&0\\
        1&1&0\\
        0&0&1
    \end{pmatrix}.
\end{equation}
The commutator is
\begin{equation}\label{commutator2}
    RQ_bR^{-1}Q_b^{-1}
    =I_3\oplus 
    \begin{pmatrix}
	   1& 0&0&0\\
	   0 &1 &0&0\\
	   0 & 0&1&0\\
	   s_1b&0&0 &1 
    \end{pmatrix}
    \oplus I_2 \oplus
    I_3 \in \mathcal{V}_{1,\,1}(\overline{\mathbb{F}_q(t)}).
\end{equation}
Therefore, we have $\mathcal{V}_{1,\,1}=V_{1,\,1}$ as $b$ is arbitrarily chosen and $s_1$ is assumed not to be a multiple of $p$. Since we have proven that $\Gamma_{M(3,\,0)}$ is equal to the algebraic set $G_{3,\,0}$, we obtain
\begin{equation}
    \dim \Gamma_{M(1,\,1)}=\dim \Gamma_{M(3,\,0)}+\dim \mathcal{V}_{1,\,1}= 7.
\end{equation}
By Theorem \ref{ThmPapanikolas}, we can conclude that
\begin{equation}
    \trdeg_{\overline{K}}\overline{K}(\alpha_0,\,\alpha_1,\,\alpha_2,\,\beta_{1,\,0},\,\beta_{2,\,0},\,\beta_{3,\,0},\,\beta_{1,\,1})=\dim \Gamma_{M(1,\,1)}=7
\end{equation}
and $\Gamma_{M(1,\,1)}=G_{1,\,1}$.
\end{example}

\subsection{The pre-$t$-motives considered}

In this subsection, we consider certain pre-$t$-motives which have Taylor coefficients of $t$-motivic CMPL's as their periods.
We take and fix a non-negative integer $n$.

\begin{definition}\label{Defsubprime}
For each index $\mathbf{s}=(s_1,\,\dots,\,s_r)\in \mathbb{Z}_{\geq 1}^r$, we put
\begin{equation}
\operatorname{Sub}^\prime(\mathbf{s}):= \{ (s_{i_1},\,s_{i_2},\,\dots,\,s_{i_d}) \mid 1\leq d \leq r,\, 1\leq i_1<\cdots<i_d\leq r \}. \label{eqsubprime}
\end{equation}
 
\end{definition}
Let us take an index $\mathbf{s} =(s_1,\,\dots,\,s_r)\in \mathbb{Z}_{\geq 1}^r$.
Then we enumerate elements of the set $\operatorname{Sub}^\prime(\mathbf{s})$ as $ \mathbf{s}^{[1]},$ $\mathbf{s}^{[2]},\dots,$ $\mathbf{s}^{[\#\operatorname{Sub}^\prime(\mathbf{s})]}$ so that $\operatorname{dep}(\mathbf{s}^{[i]})\leq\operatorname{dep}(\mathbf{s}^{[i+1]})$ for all $1 \leq i\leq \#\operatorname{Sub}^\prime(\mathbf{s})-1$.
For simplicity, we assume that the positive integers $s_1,\,\dots,\,s_r$ are pairwise distinct.
We also take elements $u_1,\,\dots,\,u_r$ in $\overline{K}[t]$. 
We assume that we have 
 \begin{equation}
    ||u_i||_\infty<|\theta|_\infty^{\frac{s_i q}{q-1}}
 \end{equation}
for all $1\leq i\leq r$.
Let us take $1 \leq i\leq \#\operatorname{Sub}^\prime(\mathbf{s})-1$. If $\mathbf{s}^{[i]}=(s_{i_1},\,s_{i_2},\,\dots,\,s_{i_d})$, then we write $\mathbf{u}^{[i]}:=(u_{i_1},\,u_{i_2},\,\dots,$ $u_{i_d})$ by an abuse of language.
We define $\Phi_i$ to be the matrix $\Phi[\mathbf{u}^{[i]};\mathbf{s}^{[i]}]$ defined as in Equation \eqref{PhiUS} and $M[i]$ to be the pre-$t$-motive $M[\mathbf{u}^{[i]};\mathbf{s}^{[i]}]$, whose $\sigma$-action is represented by $\Phi_i$. We note that the matrix $\Psi_i:=\Psi[\mathbf{u}^{[i]};\mathbf{s}^{[i]}]$ defined as \eqref{PsiUS}  is a rigid analytic trivialization of $\Phi_i$.

\begin{definition}\label{mainpretmotive}
    For any integers $0 \leq m \leq n$ and $1 \leq i \leq \#\operatorname{Sub}^\prime(\mathbf{s})$, we define the pre-$t$-motive $M(i,\,m)$ to be
\begin{align}
\begin{cases}
\rho_n C \oplus \displaystyle{\bigoplus_{1 \leq j\leq i} M[j]} & \text{if  $m=0$,}\\
\rho_n C \oplus \displaystyle{\bigoplus_{1 \leq j\leq i} \rho_m M[j] \oplus \bigoplus_{i<j \leq \#\operatorname{Sub}^\prime(\mathbf{s})}\rho_{m-1}M[j] }& \text{if  $m \geq 1$}.
\end{cases}
\end{align}
We further put
\begin{equation}\label{explainMprime}
M^\prime(i,\,m):=
\begin{cases}
M(i-1,\,m)& \text{ $i > 1$},\\
M(\#\operatorname{Sub}^\prime(\mathbf{s}),\,m-1)& \text{$i=1$},
\end{cases}
\end{equation}
for $(i,\,m)\neq (1,\,0)$.
\end{definition}

We put
\begin{equation}
N=N_{n,\,m,\,i}=(n+1)+(m+1)\sum_{1 \leq j\leq i} (\operatorname{dep}\mathbf{s}^{[j]}+1)+m\sum_{i<j \leq \#\operatorname{Sub}^\prime(\mathbf{s})} (\operatorname{dep}\mathbf{s}^{[j]}+1)
\end{equation} for any integers $0 \leq m \leq n$ and $1 \leq i \leq \#\operatorname{Sub}^\prime(\mathbf{s})$. Then we can take a representing matrix 
\begin{align} 
\Phi(i,\,m):=\begin{cases}
\rho_n (t-\theta) \oplus\displaystyle{ \bigoplus_{1 \leq j\leq i} \Phi_j }& \text{if $m=0$,}\\
\rho_n (t-\theta) \oplus \displaystyle{\bigoplus_{1 \leq j\leq i} \rho_m \Phi_j \oplus \bigoplus_{i<j \leq \#\operatorname{Sub}^\prime(\mathbf{s})}\rho_{m-1}\Phi_j }& \text{if $m \geq 1$}
\end{cases}\label{Phiim}
\end{align}
in $\operatorname{GL}_N(\overline{K}[t])$ of $M(i,\,m)$ and its rigid analytic trivialization
\begin{align}\label{Psiim}
\Psi(i,\,m):=
\begin{cases}
\rho_n (\Omega) \oplus\displaystyle{ \bigoplus_{j\leq i}\Psi_j} & \text{if $m=0$,}\\
\rho_n (\Omega) \oplus \displaystyle{ \bigoplus_{j\leq i} \rho_m \Psi_j \oplus \bigoplus_{i<j \leq \#\operatorname{Sub}^\prime(\mathbf{s})}\rho_{m-1}\Psi_j}& \text{if $m \geq1$}
\end{cases}
\end{align}
in $\operatorname{GL}_N(\mathbb{T})$.
We note that the pre-$t$-motive $M(r,\,0)$ coincides with $M(n)$ in Section \ref{sectionfirststep} if $s_1,\,\dots,\,s_r$ are distinct positive integers.

These pre-$t$-motives $M(i,\,m)$ have Taylor coefficients of $t$-motivic CMPLs as their periods.  So the study of their $t$-motivic Galois group tells us about the algebraic independence of these Taylor coefficients via Theorem \ref{ThmPapanikolas} as the following examples suggest.  
The purpose of this section is to explicitly determine the $t$-motivic Galois group of $M(i,\,m)$ and study the transcendence property of Taylor coefficients of $t$-motivic CMPLs.
\begin{example}
    Let us consider the case where $n=2$ and $r=2$. 
    We put $\mathbf{s}^{[1]}=(s_1),\,\mathbf{s}^{[2]}=(s_2),\,\mathbf{s}^{[3]}=(s_1,\,s_2)$.
    The pre-$t$-motives $M(2,0)$, $M(3,0)$, and $M(1,\,1)$ appeared in Subsection \ref{Subsectionsimplestexample}.
The representing matrix $\Phi(2,\,1)$ of the pre-$t$-motive $M(2,\,1)$ is written as follows:
    \begin{align}
        &\begin{pmatrix}
        t-\theta& 0 &0\\
        1& t-\theta &0\\
        0&1&t-\theta
    \end{pmatrix}
        \oplus
        \begin{pmatrix}
            \begin{matrix}
                (t-\theta)^{s_1}&0\\
                (t-\theta)^{s_1} u_{1}^{(-1)}&1
            \end{matrix} &0\\
            \begin{matrix}
                s_1(t-\theta)^{s_1-1}&0\\
                \partial^{(1)}\left((t-\theta)^{s_1} u_{1}^{(-1)}\right)&0
            \end{matrix}&
            \begin{matrix}
                (t-\theta)^{s_1}&0\\
                (t-\theta)^{s_1} u_{1}^{(-1)}&1
            \end{matrix}\\
        \end{pmatrix}\\
        &\quad \ \ \oplus  \begin{pmatrix}
            \begin{matrix}
                (t-\theta)^{s_2}&0\\
                (t-\theta)^{s_2} u_{2}^{(-1)}&1
            \end{matrix} &0\\
            \begin{matrix}
                s_2(t-\theta)^{s_2-1}&0\\
                \partial^{(1)}\left((t-\theta)^{s_2} u_{2}^{(-1)}\right)&0
            \end{matrix}&
            \begin{matrix}
                (t-\theta)^{s_2}&0\\
                (t-\theta)^{s_2} u_{2}^{(-1)}&1
            \end{matrix}
            \end{pmatrix}\\
            &\quad \ \ \ \oplus 
            \begin{pmatrix}
                (t-\theta)^{s_1+s_2}&&\\
                (t-\theta)^{s_1+s_2}u_1^{(-1)}&(t-\theta)^{s_2}&\\
                &(t-\theta)^{s_2}u_2^{(-1)}&1
            \end{pmatrix}
        
    \end{align}
    The matrices $\Phi(2,\,1)$ has a rigid analytic trivialization
    \begin{align}
        &\begin{pmatrix}
        \Omega& 0 &0\\
        \partial^{(1)} \Omega& \Omega &0\\
        \partial^{(2)} \Omega&\partial^{(1)} \Omega&\Omega
    \end{pmatrix}
        \oplus
        \begin{pmatrix}
            \begin{matrix}
                \Omega^{s_1}&0\\
                \Omega^{s_1} \mathcal{L}_{(u_1),\,(s_1)}&1
            \end{matrix} &0\\
            \begin{matrix}
                s_1\Omega^{s_1-1}\partial^{(1)}\Omega&0\\
                \partial^{(1)}\left(\Omega^{s_1} \mathcal{L}_{(u_1),\,(s_1)}\right)&0
            \end{matrix}&
            \begin{matrix}
                \Omega^{s_1}&0\\
                \Omega^{s_1} \mathcal{L}_{(u_1),\,(s_1)}&1
            \end{matrix}\\
        \end{pmatrix}\\
        &\quad \ \ \oplus  \begin{pmatrix}
            \begin{matrix}
                \Omega^{s_2}&0\\
                \Omega^{s_2} \mathcal{L}_{(u_2),\,(s_2)}&1
            \end{matrix} &0\\
            \begin{matrix}
                s_2\Omega^{s_2-1}\partial^{(1)}\Omega&0\\
                \partial^{(1)}\left(\Omega^{s_2} \mathcal{L}_{(u_2),\,(s_2)}\right)&0
            \end{matrix}&
            \begin{matrix}
                \Omega^{s_2}&0\\
                \Omega^{s_2} \mathcal{L}_{(u_2),\,(s_2)}&1
            \end{matrix}
            \end{pmatrix}\\
            &\quad \ \ \ \oplus 
            \begin{pmatrix}
                \Omega^{s_1+s_2}&&\\
                \Omega^{s_1+s_2}\mathcal{L}_{(u_1),\,(s_1)}&\Omega^{s_2}&\\
                \Omega^{s_1+s_2}\mathcal{L}_{(u_1,\,u_2),\,(s_1,\,s_2)}&\Omega^{s_2}\mathcal{L}_{(u_2),\,(s_2)}&1
            \end{pmatrix}
        
    \end{align}
    Therefore, Theorem \ref{ThmPapanikolas} implies that the transcendental degree of the field generated by $\widetilde{\pi},\,\partial^{(1)}\Omega |_{t=\theta},\,\mathcal{L}_{(u_1),\,(s_1)}|_{t=\theta},\,\mathcal{L}_{(u_2),\,(s_2)}|_{t=\theta},\,\mathcal{L}_{(u_1,\,u_2),(s_1,\,s_2)}|_{t=\theta}$, $\partial^{(1)}\mathcal{L}_{(u_1),\,(s_1)}|_{t=\theta}$, and $\partial^{(1)}\mathcal{L}_{(u_2),\,(s_2)}|_{t=\theta}$ over $\overline{K}$ is equal to the dimension of the algebraic group $\Gamma_{M(2,\,1)}$.
    
    Similarly, the pre-$t$-motive $M(3,\,1)$ is represented by the matrix $\Phi(3,\,1)$ which is written as
\begin{align}
        &\begin{pmatrix}
        t-\theta& 0 &0\\
        1& t-\theta &0\\
        0&1&t-\theta
    \end{pmatrix}
        \oplus
        \begin{pmatrix}
            \begin{matrix}
                (t-\theta)^{s_1}&0\\
                (t-\theta)^{s_1} u_{1}^{(-1)}&1
            \end{matrix} &O\\
            \begin{matrix}
                s_1(t-\theta)^{s_1-1}&0\\
                \partial^{(1)}\left((t-\theta)^{s_1} u_{1}^{(-1)}\right)&0
            \end{matrix}&
            \begin{matrix}
                (t-\theta)^{s_1}&0\\
                (t-\theta)^{s_1} u_{1}^{(-1)}&1
            \end{matrix}\\
        \end{pmatrix}\\
        &\quad \ \ \oplus  \begin{pmatrix}
            \begin{matrix}
                (t-\theta)^{s_2}&0\\
                (t-\theta)^{s_2} u_{12}^{(-1)}&1
            \end{matrix} &O\\
            \begin{matrix}
                s_2(t-\theta)^{s_2-1}&0\\
                \partial^{(1)}\left((t-\theta)^{s_2} u_{2}^{(-1)}\right)&0
            \end{matrix}&
            \begin{matrix}
                (t-\theta)^{s_2}&0\\
                (t-\theta)^{s_2} u_{2}^{(-1)}&1
            \end{matrix}
            \end{pmatrix}\oplus \begin{pmatrix}
             \Phi_3&O\\
            \partial^{(1)} \Phi_3 &\Phi_3
            \end{pmatrix}
    \end{align}
    where
    \begin{align}
        \Phi_3:&= \begin{pmatrix}
                (t-\theta)^{s_1+s_2}&&\\
                (t-\theta)^{s_1+s_2}u_1^{(-1)}&(t-\theta)^{s_2}&\\
                &(t-\theta)^{s_2}u_2^{(-1)}&1
            \end{pmatrix} \text{ and }\\
            \partial^{(1)} \Phi_3&=\begin{pmatrix}
                (s_1+s_2)(t-\theta)^{s_1+s_2-1}&&\\
                \partial^{(1)}\left((t-\theta)^{s_1+s_2}u_1^{(-1)}\right)&s_2(t-\theta)^{s_2-1}&\\
                &\partial^{(1)}\left((t-\theta)^{s_2}u_2^{(-1)}\right)&0
            \end{pmatrix},
    \end{align}
    and has a rigid analytic trivialization $\Psi(3,\,1)$ given by
    \begin{align}
        &\begin{pmatrix}
        \Omega& 0 &0\\
        \partial^{(1)} \Omega& \Omega &0\\
        \partial^{(2)} \Omega&\partial^{(1)} \Omega&\Omega
    \end{pmatrix}
        \oplus
        \begin{pmatrix}
            \begin{matrix}
                \Omega^{s_1}&0\\
                \Omega^{s_1} \mathcal{L}_{(u_1),\,(s_1)}&1
            \end{matrix} &O\\
            \begin{matrix}
                s_1\Omega^{s_1-1}\partial^{(1)} \Omega&0\\
                \partial^{(1)}\left(\Omega^{s_1} \mathcal{L}_{(u_1),\,(s_1)}\right)&0
            \end{matrix}&
            \begin{matrix}
                \Omega^{s_1}&0\\
                \Omega^{s_1} \mathcal{L}_{(u_1),\,(s_1)}&1
            \end{matrix}\\
        \end{pmatrix}\\
        &\quad \ \ \oplus  \begin{pmatrix}
            \begin{matrix}
                \Omega^{s_2}&0\\
                \Omega^{s_2} \mathcal{L}_{(u_2),\,(s_2)}&1
            \end{matrix} &0\\
            \begin{matrix}
                s_2\Omega^{s_2-1}\partial^{(1)} \Omega&0\\
                \partial^{(1)}\left(\Omega^{s_2} \mathcal{L}_{(u_2),\,(s_2)}\right)&0
            \end{matrix}&
            \begin{matrix}
                \Omega^{s_2}&0\\
                \Omega^{s_2} \mathcal{L}_{(u_2),\,(s_2)}&1
            \end{matrix}
            \end{pmatrix}\oplus \begin{pmatrix}
             \Psi_3&O\\
            \partial^{(1)} \Psi_3 &\Psi_3
            \end{pmatrix}
    \end{align}
    where
    \begin{align}
        \Psi_3:&=\begin{pmatrix}
                \Omega^{s_1+s_2}&&\\
                \Omega^{s_1+s_2}\mathcal{L}_{(u_1),\,(s_1)}&\Omega^{s_2}&\\
                \Omega^{s_1+s_2}\mathcal{L}_{(u_1,\,u_2),\,(s_1,\,s_2)}&\Omega^{s_2}\mathcal{L}_{(u_2),\,(s_2)}&1
            \end{pmatrix} \text{ and}\\
             \partial^{(1)} \Psi_3&=\begin{pmatrix}
                \partial^{(1)}\Omega^{s_1+s_2}&&\\
                \partial^{(1)}\left(\Omega^{s_1+s_2}\mathcal{L}_{(u_1),\,(s_1)}\right)&\partial^{(1)}\Omega^{s_2}&\\
                \partial^{(1)}\left(\Omega^{s_1+s_2}\mathcal{L}_{(u_1,\,u_2),\,(s_1,\,s_2)}\right)&\partial^{(1)}\left(\Omega^{s_2}\mathcal{L}_{(u_2),\,(s_2)}\right)&0
                
            \end{pmatrix}.
    \end{align}
    Hence, the transcendental degree of the field generated over $\overline{K}$ by $\widetilde{\pi},\,\partial^{(1)}\Omega |_{t=\theta},$ $\mathcal{L}_{(u_1),\,(s_1)}|_{t=\theta}$, $\mathcal{L}_{(u_2),\,(s_2)}|_{t=\theta}$, $\mathcal{L}_{(u_1,\,u_2),(s_1,\,s_2)}|_{t=\theta}$, $\partial^{(1)}\mathcal{L}_{(u_1),\,(s_1)}|_{t=\theta}$, $\partial^{(1)}\mathcal{L}_{(u_2),\,(s_2)}|_{t=\theta}$, and $\partial^{(1)} \mathcal{L}_{(u_1,\,u_2)}(s_1,$ $s_2)|_{t=\theta}$ is equal to $\dim \Gamma_{M(3,\,1)}$.
\end{example}

More generally, we have the following proposition. This motivates us to consider the dimensions of $t$-motivic Galois groups of aforementioned pre-$t$-motives $M(i,\,m)$. 
Let us recall that we write $\mathcal{L}_{\mathbf{u}^\prime,\,\mathbf{s}^\prime}$ for the $t$-motivic CMPL as in Section \ref{subsectionATseries} for each tuple $\mathbf{u}^\prime=(u_1,\,\dots,\,u_d) \in \overline{K}[t]$ and $\mathbf{s}^\prime=(s_1,\,\dots,\,s_d) \in \mathbb{Z}_{\geq 1}^r$ ($r \geq 1)$.
\begin{Proposition}\label{PropDimTrans}
    Take an index $\mathbf{s}=(s_1,\,\dots,\,s_r) \in \mathbb{Z}_{\geq 1}^r$ and polynomials $u_1,\,\dots,\,u_r \in \overline{K}[t]$ such that $s_i \neq s_j$, $\mathcal{L}_{(u_i),\,(s_i)}(t)|_{t=\theta}\neq 0$, and $||u_i||_\infty<|\theta|_\infty^{\frac{s_i q}{q-1}}$ for all $1\leq i<j \leq r$. If we consider the Taylor expansions
    \begin{equation}
        \Omega(t)=\sum_{n=0}^{\infty} \alpha_n (t-\theta)^n,\quad \mathcal{L}_{(u_{i_1},\,\dots,\,u_{i_d}),(s_{i_1},\,\dots,\,s_{i_d})}(t)=\sum_{n=0}^{\infty} \alpha_{(s_{i_1},\,\dots,\,s_{i_d}),\,n}(t-\theta)^n
    \end{equation}
for each $(s_{i_1},\,\dots,\,s_{i_d}) \in \operatorname{Sub}^\prime(\mathbf{s})$,
    then the field generated by the set
    \begin{equation}
        \left \{ \alpha_{n^\prime},\, \alpha_{\mathbf{s}^{[i^\prime]},\,m},\, \alpha_{\mathbf{s}^{[i^{\prime \prime}]},\,m^{\prime \prime}}\,\middle|\, \substack{0\leq n^\prime  \leq n,\\
        1 \leq i^\prime \leq i,\,0 \leq m^\prime \leq m\\
        i < i^{\prime \prime} \leq \#\operatorname{Sub}^{\prime}(\mathbf{s}),\,0 \leq m^{\prime \prime} < m } \right\}
    \end{equation}
    over $\overline{K}$ has transcendental degree equal to the dimension of the $t$-motivic Galois group $\Gamma_{M(i,\,m)}$ for any $1\leq i \leq \#\operatorname{Sub}^{\prime}(\mathbf{s})$ and $1 \leq m \leq n$.
\end{Proposition}
\begin{proof}
Theorem \ref{ThmPapanikolas} and the Leibniz rule \eqref{Leibniz} show that $\dim \Gamma_{M(i,\,m)}$ is equal to the transcendental degree of
\begin{equation}
        \overline{K} \left( \partial^{(n^\prime)}\Omega|_{t=\theta},
\partial^{(m^\prime)}\mathcal{L}_{\mathbf{u}^{[i^\prime]},\,\mathbf{s}^{[i^\prime]}},\,
\partial^{(m^{\prime \prime})}\mathcal{L}_{\mathbf{u}^{[i^{\prime \prime}]},\,\mathbf{s}^{[i^{\prime \prime}]}}\,\middle|\, \substack{0\leq n^\prime  \leq n,\\
        1 \leq i^\prime \leq i,\,0 \leq m^\prime \leq m\\
        i \leq i^{\prime \prime} \leq \#\operatorname{Sub}^{\prime}(\mathbf{s}),\,0 \leq m^{\prime \prime} < m } \right).
    \end{equation}
over $\overline{K}$. By Proposition \ref{Uchino}, this field is equal to 
\begin{equation}
        \overline{K}\left ( \alpha_{n^\prime},\, \alpha_{\mathbf{s}^{[i^\prime]},\,m},\, \alpha_{\mathbf{s}^{[i^{\prime \prime}]},\,m^{\prime \prime}}\,\middle|\, \substack{0\leq n^\prime  \leq n,\\
        1 \leq i^\prime \leq i,\,0 \leq m^\prime \leq m\\
        i \leq i^{\prime \prime} \leq \#\operatorname{Sub}^{\prime}(\mathbf{s}),\,0 \leq m^{\prime \prime} < m } \right).
    \end{equation}
\end{proof}

Let us consider the relationships among pre-$t$-motives $M(i,\,m)$. We recall that $\rho_{m-1} M[i]$ is a sub-pre-$t$-motive of $\rho_m M[i]$ for each $1 \leq i \leq \#\operatorname{Sub}^\prime(\mathbf{s})$ and $m \geq 1$ (cf. \cite[Remark 3.2]{Maurischat2018}). Therefore, the Tannakian subcategory $\langle M^\prime(i,\,m) \rangle$, which is generated by $M^\prime(i,\,m)$ can be seen as full Tannakian subcategory of $\langle M(i,\,m) \rangle$ for $1
 \leq i \leq 2^r-1$ and $0 \leq m \leq n$. 

\begin{definition}\label{mathcalVimmorph}
    We define $\varphi_{i,\,m}$ to be the faithfully flat homomorphism 
\begin{equation}\label{Eqphiim}
     \Gamma_{M(i,\,m)}\twoheadrightarrow\Gamma_{M^\prime(i,\,m)}
\end{equation}
of algebraic groups given by Tannakian duality (\cite[Proposition 2.21]{Deligne1982}). We also write $\varphi_{i,\,m}$ for the faithfully flat morphism from $\Gamma_{\Psi(i,\,m)}$ to $\Gamma_{\Psi^\prime(i,\,m)}$ given by the homomorphism above and Theorem \ref{ThmPapanikolas}.
\end{definition}

We will study the structures of $t$-motivic Galois groups by the induction on tuples $(i,\,m)$ where the order is given by
\begin{equation}
    (i_1,\,m_1) \geq (i_2 ,\,m_2) \text{ if ``$m_1> m_2$'' or ``$m_1=m_2$ and $i_1 \geq i_2$''}.
\end{equation}
Homomorphisms $\varphi_{i,\,m}$ play important roles in the induction argument.

\begin{definition}\label{mathcalVim}
    We let $\mathcal{V}_{i,\,m}$ be the kernel of $\varphi_{i,\,m}$ for  $1
 \leq i \leq \#\operatorname{Sub}^\prime(\mathbf{s})$ and $0 \leq m \leq n$ with $(i,\,m)\neq (1,\,0)$.
\end{definition}

\subsection{Explicit descriptions of $t$-motivic Galois groups}\label{subsectionexplicit}

Our aim in this subsection is define algebraic varieties (Definition \ref{explicitform}) which we use to study the $t$-motivic Galois groups of pre-$t$-motives constructed in the previous subsection.

We take an index $\mathbf{s}=(s_1,\,\dots,\,s_r)$ and fix an enumeration $\mathbf{s}^{[1]},\,\mathbf{s}^{[2]},\,\dots,$ $\,\mathbf{s}^{[\#\operatorname{Sub}^\prime(\mathbf{s})]}$ of the set 
\begin{equation}
\operatorname{Sub}^\prime(\mathbf{s}):= \{ (s_{i_1},\,s_{i_2},\,\dots,\,s_{i_d}) \mid 1\leq d \leq r,\, 1\leq i_1<\cdots<i_d\leq r \} 
\end{equation}
such that we have $\operatorname{dep}(\mathbf{s}^{[i]})\leq\operatorname{dep}(\mathbf{s}^{[i+1]})$ for each $1 \leq i \leq \#\operatorname{Sub}^\prime(\mathbf{s})-1$.

Let us take variables $a_0,\,a_1,\,\dots,\,a_n,$ $ x_{m,\, \mathbf{s}^{\prime}}$ ($0 \leq m \leq n$, $\mathbf{s}^{\prime} \in \operatorname{Sub}^\prime(\mathbf{s})$).
For $m \geq 1$ and $s \geq 1$, we define the polynomials $\left\{\underline{D}_{s}
^{(0)}, \underline{D}_{s}
^{(1)},\ldots, \underline{D}_{s}
^{(m)}\right\}$ by the following expansion

\begin{align}
 \prod_{i=1}^s
\begin{pmatrix}
   X_{0,\,i} & 0 &\cdots&\cdots&0\\
   X_{1,\,i} & X_{0,\,i}&0&\cdots&0\\
   X_{2,\,i} & X_{1,\,i}& X_{0,\,i}&\ddots&\vdots\\
   \vdots&\vdots& \ddots&\ddots&0 \\
   X_{m,\,i}& X_{m-1,\,i}&\cdots &X_{1,\,i}&X_{0,\,i}\\
\end{pmatrix}
=:\begin{pmatrix}
   \underline{D}_s^{(0)} & 0 &\cdots&0\\
   \underline{D}_s^{(1)}  & \underline{D}_s^{(0)} &\cdots&0\\
   \vdots& \vdots&\ddots&\vdots \\
   \underline{D}_s^{(m)} & \underline{D}_s^{(m-1)} &\cdots&\underline{D}_s^{(0)}
\end{pmatrix}.
\end{align}
We mention that $\underline{D}_{s}^{(m)}$ is a polynomial in the variable $\{X_{0,1},\ldots,X_{0,s},X_{1,1}$, $\ldots $, $ X_{1,s},\ldots, X_{m,1},\ldots,X_{m,s} \}$. Without confusion with matrix product, we write 
\[\underline{D}_{s}
^{(m)}=\underline{D}_{s}
^{(m)} \begin{pmatrix}
   X_{0,\,1} &X_{1,\,1}& \cdots& X_{m,\,1}\\
   \vdots&\vdots &\cdots&\vdots \\
    X_{0,\,s} &X_{1,\,s}& \cdots& X_{m,\,s}  
\end{pmatrix} \]to emphasize that it is a polynomial in the variables above. We note further that we can explicitly write down $\underline{D}_{s}^{(m)}$ as
\begin{equation}
\underline{D}_{s}
^{(m)} \begin{pmatrix}
   X_{0,\,1} &X_{1,\,1}& \cdots& X_{m,\,1}\\
   \vdots&\vdots &\cdots&\vdots \\
    X_{0,\,s} &X_{1,\,s}& \cdots& X_{m,\,s}  
\end{pmatrix}=\sum_{j_1+j_2+\cdots+j_s=m}X_{j_1,\,1}X_{j_2,\,2}\cdots X_{j_s,\,s}.
\end{equation} 
It follows that once we replace $X_{i,j}$ by $\partial^{(i)}f_{j}$, the formula \eqref{Leibniz} implies that
\begin{align}
\underline{D}_{s}
^{(m)} \begin{pmatrix}
   \partial^{(0)}f_1 &\partial^{(1)}f_1& \cdots& \partial^{(m)}f_1\\
   \vdots&\vdots &\cdots&\vdots \\
    \partial^{(0)}f_s &\partial^{(1)}f_s& \cdots& \partial^{(m)}f_s  
\end{pmatrix}&=\sum_{j_1+j_2+\cdots+j_s=m}\partial^{(j_1)}f_1\partial^{(j_2)}f_2\cdots \partial^{(j_s)}f_s\\
&=\partial^{(m)}(f_1\cdots f_s)
\end{align}
for $f_1,\,\dots,\,f_s \in \mathbb{C}_\infty((t))$.

\begin{definition}\label{Smj}
For each $0 \leq m \leq n$ and $0 \leq j \leq 2^r-1$, we write $S_{m,\,j}$ for the lower triangle matrix of size $d^\prime+1$ whose $(k,\,l)$-entry is
\begin{align}
&\underline{D}_{s_{j_l^\prime}+\cdots+s_{j_{d^\prime}^\prime}+1}^{(m)}
\begin{pmatrix}
   a_0 &a_1& \cdots& a_m\\
   \vdots&\vdots &\vdots &\vdots\\
   a_0 &a_1& \cdots& a_m\\
   x_{0,\,(s_{j_l^\prime},\dots,\,s_{j_{k-1}^\prime})} & x_{1,\,(s_{j_l^\prime},\dots,\,s_{j_{k-1}^\prime})} & \cdots& x_{m,\,(s_{j_l^\prime},\dots,\,s_{j_{k-1}^\prime})}\\
   \end{pmatrix} \text{(if $k>l$)}\\
  & \underline{D}_{s_{j_l^\prime}+\cdots+s_{j_{d^\prime}^\prime}}^{(m)}
\begin{pmatrix}
   a_0 &a_1& \cdots& a_m\\
   \vdots&\vdots &\vdots &\vdots\\
   a_0 &a_1& \cdots& a_m
   \end{pmatrix} \text{ (if $k=l$)}
\end{align}
where $\mathbf{s}^{[j]}=(s_{j_1^\prime},\dots,\,s_{j_{d^\prime}^\prime})$ and $1\leq l \leq k \leq d^\prime+1$. The entry in the lower right corner is $1$ if $m=0$ and is $0$ if $m \geq1$.
\end{definition}

\begin{example}If $\mathbf{s}^{[j]}=(s,\,s^\prime) \in \mathbb{Z}^2$, then $S_{m,\,j}$ is given as follows:

        \begin{equation}
                \begin{pmatrix}
                    \displaystyle{\sum_{i_1+\cdots+i_{s+s^\prime}=m}}a_{i_1}\cdots a_{i_{s+s^\prime}}&0&0\\
                     \displaystyle{\sum_{i_1+\cdots+i_{s+s^\prime}+i=m}}a_{i_1}\cdots a_{i_{s+s^\prime}}x_{i,\,(s)}&\displaystyle{\sum_{i_1+\cdots+i_{s^\prime}=m}}a_{i_1}\cdots a_{i_{s^\prime}}&0 \\
                    \displaystyle{\sum_{i_1+\cdots+i_{s+s^\prime}+i=m}}a_{i_1}\cdots a_{i_{s+s^\prime}}x_{i,\,\mathbf{s}^{[j]}}&\displaystyle{\sum_{i_1+\cdots+i_{s^\prime}+i=m}}a_{i_1}\cdots a_{i_{s^\prime}}x_{i,\,(s^\prime)}&0
                \end{pmatrix}
        \end{equation}
    for $m \geq 1$.
\end{example}

Definition \ref{Smj} helps us to study the structure of the $t$-motivic Galois group of pre-$t$-motives $M(i,\,m)$. Using Definition \ref{Smj}, we define explicit algebraic varieties $G_{i,\,m}$ below and show later that these varieties
contains the $t$-motivic Galois group $\Gamma_{M(i,\,m)}$ of pre-$t$-motives $M(i,\,m)$ as closed subvarieties.

\begin{definition}\label{explicitform}
Let us fix $n$ and take arbitrary $0 \leq m \leq n$ and $1 \leq i \leq \# \operatorname{Sub}^{\prime}(\mathbf{s})$. Recall that we put
\begin{equation}
N=N_{n,\,m,\,i}=(n+1)+(m+1)\sum_{1 \leq j\leq i} (\operatorname{dep}\mathbf{s}^{[j]}+1)+m\sum_{i<j \leq \#\operatorname{Sub}^\prime(\mathbf{s})} (\operatorname{dep}\mathbf{s}^{[j]}+1).
\end{equation}
We define $G_{i,\,m}$ to be the closed algebraic subset of $\operatorname{GL}_N$ over $\overline{\mathbb{F}_q(t)}$ consisting of matrices of the form
\begin{align}\label{elementsofGim}
\begin{pmatrix}
   a_0 & 0 &\cdots&\cdots&0\\
   a_1 & a_0&0&\cdots&0\\
   a_2 & a_1& a_0&\ddots&\vdots\\
   \vdots&\vdots& \ddots&\ddots&0 \\
   a_{n}& a_{n-1}&\cdots &a_1&a_0\\
\end{pmatrix}
\oplus \bigoplus_{1 \leq j\leq i}\begin{pmatrix}
   S_{0,\,j} & 0 &\cdots&0\\
    S_{1,\,j}  &  S_{0,\,j} &\cdots&0\\
   \vdots& \vdots&\ddots&\vdots \\
    S_{m,\,j} &  S_{m-1,\,j} &\cdots& S_{0,\,j} 
\end{pmatrix}\\
 \oplus \bigoplus_{i<j \leq \#\operatorname{Sub}^\prime(\mathbf{s})}
\begin{pmatrix}
   S_{0,\,j} & 0 &\cdots&0\\
    S_{1,\,j}  &  S_{0,\,j} &\cdots&0\\
   \vdots& \vdots&\ddots&\vdots \\
    S_{m-1,\,j} &  S_{m-2,\,j} &\cdots& S_{0,\,j} 
\end{pmatrix}.
\end{align}
for $m \geq 1$. 
In the case $m=0$, $G_{i,\,m}$ is defined to be the closed algebraic subset of $\operatorname{GL}_N$ over $\overline{\mathbb{F}_q(t)}$ consisting of matrices of the form
\begin{equation}
\begin{pmatrix}
   a_0 & 0 &\cdots&\cdots&0\\
   a_1 & a_0&0&\cdots&0\\
   a_2 & a_1& a_0&\ddots&\vdots\\
   \vdots&\vdots& \ddots&\ddots&0 \\
   a_{n}& a_{n-1}&\cdots &a_1&a_0\\
\end{pmatrix}
\oplus \bigoplus_{1 \leq j\leq i}S_{0,\,j}.
\end{equation}

The variety $G(r,\,0)$ can be identified with the algebraic group $G(n)$ in Section \ref{sectionfirststep} if $s_1,\,\dots,\,s_r$ are distinct positive integers.

We let $G_{i,\,m}^{\prime}$ be the algebraic variety

\begin{equation}\label{explainGprime1}
\begin{cases}
G_{i-1,\,m}& \text{if $i > 1$},\\
G_{\#\operatorname{Sub}^\prime(\mathbf{s}),\,m-1}& \text{if $i=1$},
\end{cases}
\end{equation}
and let $\Pi_{i,\,m}:G_{i,\,m} \twoheadrightarrow G_{i,\,m}^{\prime}$ be the projections for $1 \leq i \leq \#\operatorname{Sub}^\prime(\mathbf{s})$ and $0\leq m \leq n$ with $(i,\,m) \neq (1,0)$. 
\end{definition}
For instance, the projection $\Pi_{i,\,m}$ is given by
\begin{small}
\begin{align}\label{elementsofGim2}
\begin{pmatrix}
   a_0 & 0 &\cdots&\cdots&0\\
   a_1 & a_0&0&\cdots&0\\
   \vdots& \vdots&\ddots&\ddots&\vdots \\
   a_{n-1}& a_{n-2}&\cdots&a_0&0\\
   a_n&a_{n-1} &\cdots&a_1&a_0
\end{pmatrix}
\oplus \bigoplus_{1 \leq j\leq i}\begin{pmatrix}
   S_{0,\,j} & 0 &\cdots&0\\
    S_{1,\,j}  &  S_{0,\,j} &\cdots&0\\
   \vdots& \vdots&\ddots&\vdots \\
    S_{m,\,j} &  S_{m-1,\,j} &\cdots& S_{0,\,j} 
\end{pmatrix}\\
 \oplus \bigoplus_{i<j \leq \#\operatorname{Sub}^\prime(\mathbf{s})}
\begin{pmatrix}
   S_{0,\,j} & 0 &\cdots&0\\
    S_{1,\,j}  &  S_{0,\,j} &\cdots&0\\
   \vdots& \vdots&\ddots&\vdots \\
    S_{m-1,\,j} &  S_{m-2,\,j} &\cdots& S_{0,\,j} 
\end{pmatrix}\\
=\begin{pmatrix}
   a_0 & 0 &\cdots&\cdots&0\\
   a_1 & a_0&0&\cdots&0\\
   \vdots& \vdots&\ddots&\ddots&\vdots \\
   a_{n-1}& a_{n-2}&\cdots&a_0&0\\
   a_n&a_{n-1} &\cdots&a_1&a_0
\end{pmatrix}
\oplus \bigoplus_{1 \leq j\leq i-1}\begin{pmatrix}
   S_{0,\,j} & 0 &\cdots&0\\
    S_{1,\,j}  &  S_{0,\,j} &\cdots&0\\
   \vdots& \vdots&\ddots&\vdots \\
    S_{m,\,j} &  S_{m-1,\,j} &\cdots& S_{0,\,j} 
\end{pmatrix}\oplus\\
 \begin{pmatrix}
   S_{0,\,i} & 0 &\cdots&0\\
    S_{1,\,i}  &  S_{0,\,i} &\cdots&0\\
   \vdots& \vdots&\ddots&\vdots \\
    S_{m,\,i} &  S_{m-1,\,i} &\cdots& S_{0,\,i} 
\end{pmatrix}\oplus \bigoplus_{i<j \leq \#\operatorname{Sub}^\prime(\mathbf{s})}
\begin{pmatrix}
   S_{0,\,j} & 0 &\cdots&0\\
    S_{1,\,j}  &  S_{0,\,j} &\cdots&0\\
   \vdots& \vdots&\ddots&\vdots \\
    S_{m-1,\,j} &  S_{m-2,\,j} &\cdots& S_{0,\,j} 
\end{pmatrix}
\\\mapsto 
\begin{pmatrix}
   a_0 & 0 &\cdots&\cdots&0\\
   a_1 & a_0&0&\cdots&0\\
   \vdots& \vdots&\ddots&\ddots&\vdots \\
   a_{n-1}& a_{n-2}&\cdots&a_0&0\\
   a_n&a_{n-1} &\cdots&a_1&a_0
\end{pmatrix}
\oplus \bigoplus_{1 \leq j\leq i-1}\begin{pmatrix}
   S_{0,\,j} & 0 &\cdots&0\\
    S_{1,\,j}  &  S_{0,\,j} &\cdots&0\\
   \vdots& \vdots&\ddots&\vdots \\
    S_{m,\,j} &  S_{m-1,\,j} &\cdots& S_{0,\,j} 
\end{pmatrix}\oplus\\
 \begin{pmatrix}
   S_{0,\,i} & 0 &\cdots&0\\
    S_{1,\,i}  &  S_{0,\,i} &\cdots&0\\
   \vdots& \vdots&\ddots&\vdots \\
    S_{m-1,\,i} &  S_{m-2,\,i} &\cdots& S_{0,\,i} 
\end{pmatrix}\oplus \bigoplus_{i<j \leq \#\operatorname{Sub}^\prime(\mathbf{s})}
\begin{pmatrix}
   S_{0,\,j} & 0 &\cdots&0\\
    S_{1,\,j}  &  S_{0,\,j} &\cdots&0\\
   \vdots& \vdots&\ddots&\vdots \\
    S_{m-1,\,j} &  S_{m-2,\,j} &\cdots& S_{0,\,j} 
\end{pmatrix}\\
=
\begin{pmatrix}
   a_0 & 0 &\cdots&\cdots&0\\
   a_1 & a_0&0&\cdots&0\\
   \vdots& \vdots&\ddots&\ddots&\vdots \\
   a_{n-1}& a_{n-2}&\cdots&a_0&0\\
   a_n&a_{n-1} &\cdots&a_1&a_0
\end{pmatrix}\oplus \bigoplus_{1 \leq j\leq i-1}\begin{pmatrix}
   S_{0,\,j} & 0 &\cdots&0\\
    S_{1,\,j}  &  S_{0,\,j} &\cdots&0\\
   \vdots& \vdots&\ddots&\vdots \\
    S_{m,\,j} &  S_{m-1,\,j} &\cdots& S_{0,\,j} 
\end{pmatrix}\\
 \oplus \bigoplus_{i-1<j \leq \#\operatorname{Sub}^\prime(\mathbf{s})}
\begin{pmatrix}
   S_{0,\,j} & 0 &\cdots&0\\
    S_{1,\,j}  &  S_{0,\,j} &\cdots&0\\
   \vdots& \vdots&\ddots&\vdots \\
    S_{m-1,\,j} &  S_{m-2,\,j} &\cdots& S_{0,\,j} 
\end{pmatrix}

\end{align}
\end{small}
for $m \geq1$.

The variety $G_{i,\,m}$ is isomorphic to 
\begin{equation}\operatorname{Spec}\overline{\mathbb{F}_q(t)}[a_0,\,a_0^{-1},\,a_1,\,\dots,\,a_n ,\,x_{0,\,1},\,\dots,\,x_{m,\,i}].
\end{equation}
Hence the it is smooth and irreduceible.

The key point of the proof of Theorem \ref{MainB} is to show
that, under the conditions of Theorem \ref{MainB}, the $t$-motivic Galois group of pre-$t$-motive $M(i,\,m)$ is equal to  $G_{i,\,m}$ explicitly defined above.
In what follows, our central target is to prove this equality, and we first confirm the following inclusion.

\begin{Proposition}\label{iota}
Take an index $\mathbf{s}=(s_1,\,\dots,\,s_r)$ so that positive integers $s_1,\,\dots,\,s_r$ are pairwise distinct and fix the enumeration $\mathbf{s}^{[1]},\,\mathbf{s}^{[2]},\,\dots,\,\mathbf{s}^{[\# \operatorname{Sub}^\prime(\mathbf{s})]}$ of the set $\operatorname{Sub}^\prime(\mathbf{s})$.
Also take polynomials $u_1,\,\dots,\,u_r$ in $\overline{K}[t]$ such that we have
\begin{equation}
        ||u_i||_\infty<|\theta|_\infty^{\frac{s_i q}{q-1}},\quad
     \text{($1 \leq i \leq r$)}.
\end{equation}
For $1 \leq i \leq  r$ and $0 \leq m \leq n$, we let $M(i,\,m)$ be the pre-$t$-motive defined in Definition \ref{mainpretmotive} and $G_{i,\,m}$ be the algebraic variety in Definition \ref{explicitform}.
Then the $t$-motivic Galois group $\Gamma_{M(i,\,m)}$ is isomorphic to a closed subscheme over $\overline{\mathbb{F}_q(t)}$ of $  G_{i,\,m}$ for any $1 \leq i \leq \#\operatorname{Sub}^\prime(\mathbf{s})$ and $0 \leq m \leq n$.
\end{Proposition}
We write $\iota :\Gamma_{M(i,\,m)} \hookrightarrow G_{i,\,m}$ for the closed immersion given by the proposition above. 
\begin{proof}
Note first that the hyperderivative $\partial^{(m)}$ on $\mathbb{L}$ can be extended to $\mathbb{L} \otimes_{\overline{K}(t)} \mathbb{L}$ as follows:
\begin{equation}
    \partial^{(m)}(f\otimes g):=\sum_{m^\prime=0}^m \partial^{(m^\prime)}f \otimes \partial^{(m-m^\prime)}g, \quad(f,\,g\in \mathbb{L}),
\end{equation}
that is, we have $\partial^{(m)}f\otimes1=\partial^{(m)}(f\otimes 1)$ and $1 \otimes \partial^{(m)}f=\partial^{(m)}(1 \otimes f)$, for each $m >0$. Hyperderivatives $\partial^{(0)},\,\partial^{(1)},\,\partial^{(2)},\dots$ on $\mathbb{L} \otimes_{\overline{K}(t)} \mathbb{L}$ also satisfy Leibniz rule \eqref{Leibniz}. Indeed, for any $r,\,m \geq 0$ and $f_1,\,\dots,\,f_r,\,g_1,\,\dots,\,g_r \in \mathbb{L}$, we have
\begin{align}
    & \quad \ \ \partial^{(m)}\Big( (f_1\otimes g_1)\cdots(f_r\otimes g_r) \Big)=\partial^{(m)}(f_1\cdots f_r\otimes g_1 \cdots g_r)\\
    &=\sum_{m^\prime=0}^{m}
\Big(\partial^{(m^\prime)}(f_1\cdots f_r)\otimes\partial^{(m-m^\prime)} (g_1 \cdots g_r)\Big)\\
&=\sum_{m^\prime=0}^{m}\left(\sum_{j_1+\cdots+j_r=m^\prime}\partial^{(j_1)}f_1\cdots\partial^{(j_r)}f_r \right)\otimes\left(\sum_{j_1^\prime+\cdots+j_r^\prime=m-m^\prime}\partial^{(j_1^\prime)}g_1\cdots\partial^{(j_r^\prime)}g_r \right)\\
&=\sum_{j_1+\cdots+j_r+j_1^\prime+\cdots+j_r^\prime=m}(\partial^{(j_1)}f_1\cdots\partial^{(j_r)}f_r)\otimes(\partial^{(j_1^\prime)}g_1\cdots\partial^{(j_r^\prime)}g_r)\\
&=\sum_{j_1+\cdots+j_r+j_1^\prime+\cdots+j_r^\prime=m}(\partial^{(j_1)}f_1 \otimes \partial^{(j_1^\prime)}g_1)\cdots (\partial^{(j_r)}f_r \otimes \partial^{(j_r^\prime)}g_r)\\
&=\sum_{m_1+\cdots+m_r=m}\left(\sum_{\substack{j_1+j_1^\prime=m_1, \dots, j_r+j_r^\prime=m_r}}(\partial^{(j_1)}f_1 \otimes \partial^{(j_1^\prime)}g_1)\cdots (\partial^{(j_r)}f_r \otimes \partial^{(j_r^\prime)}g_r)\right)\\
&=\sum_{m_1+\cdots+m_r=m}\partial^{(m_1)}(f_1\otimes g_1)\cdots \partial^{(m_r)}(f_1 \otimes g_r).
\end{align}

It follows that such as Section \ref{subsectionprolongation}, we  obtain an $\mathbb{L} \otimes_{\overline{K}(t)} \mathbb{L}$-algebra homomorphism
\begin{equation}
    \rho_{m}:\operatorname{Mat}_{d\times d}\left(\mathbb{L} \otimes_{\overline{K}(t)} \mathbb{L}\right) \rightarrow \operatorname{Mat}_{d(m+1)\times d(m+1)}\left(\mathbb{L} \otimes_{\overline{K}(t)} \mathbb{L}\right)
\end{equation}
     by putting
\begin{equation}
    \rho_m(X):=
    \begin{pmatrix}
   X & 0 &\cdots&\cdots&0\\
   \partial^{(1)}X & X&0&\cdots&0\\
   \vdots& \vdots&\ddots&&\vdots \\
   \partial^{(m-1)}X & \partial^{(m-2)}X&\cdots&X&0\\
   \partial^{(m)}X &\partial^{(m-1)}X &\cdots&\partial^{(1)}X &X
\end{pmatrix}.
\end{equation}
for $d \geq 1$ and $m \geq 0$.

We recall that the matrix $\Psi(i,\,m)$ in equation \eqref{Psiim} is a rigid analytic trivialization of a defining matrix of the pre-$t$-motive $M(i,\,m)$.
Recalling the notation $\widetilde{\Psi}$ defined in~\eqref{eqtilde}, we can see that

\begin{equation}
    \widetilde{\Psi(i,\,m)}=
    \begin{cases}
\widetilde{\rho_n \left((\Omega)\right) }\oplus\displaystyle{ \bigoplus_{j\leq i}\widetilde{\Psi}_j }& \text{if $m=0$,}\\
\widetilde{{\rho_n \left((\Omega)\right)}} \oplus \displaystyle{\bigoplus_{j\leq i} \widetilde{\rho_m (\Psi_j)} \oplus \bigoplus_{i<j \leq \#\operatorname{Sub}^\prime(\mathbf{s})}\widetilde{\rho_{m-1}\left(\Psi_j\right)}}& \text{if $m \geq 1$},
\end{cases}
\end{equation}
is equal to
\begin{equation}
    \begin{cases}
\rho_n \left(\widetilde{(\Omega)} \right)\oplus \displaystyle{\bigoplus_{j\leq i}\widetilde{\Psi}_j} & \text{if $m=0$,}\\
\rho_n \left(\widetilde{(\Omega)} \right)\oplus \displaystyle{\bigoplus_{j\leq i} \rho_m \left(\widetilde{\Psi}_j \right)\oplus \bigoplus_{i<j \leq \#\operatorname{Sub}^\prime(\mathbf{s})}\rho_{m-1}\left(\widetilde{\Psi}_j\right)}& \text{if $m \geq 1$}
\end{cases}
\end{equation}
since $\rho_m$ is ring homomorphism for each $m \geq 0$.
Hence we can conclude that $\widetilde{\Psi(i,\,m)}$ is an element of $G_{i,\,m}(\mathbb{L}\otimes_{\overline{K}(t)}\mathbb{L})$ by substituting $a_{n^\prime}$ by $\partial^{(n^\prime)}(\Omega^{-1}\otimes \Omega)$ and $x_{n^\prime,\,\mathbf{s}^{[j]}}$ by
\begin{equation}
\partial^{(n^\prime)}\left(\sum_{n=1}^{d+1}\sum_{m=0}^{d+1-n}(-1)^m \sum_{\substack{n=k_0<k_1<\cdots\\
    \cdots<k_{m-1}<k_m=d+1}}L_{k_1,\,k_0}^{[j]}L_{k_2,\,k_1}^{[j]}\cdots L_{k_m,\,k_m-1}^{[j]}\otimes \Omega^{s_1+\cdots+s_d}L_{n,\,1}^{[j]}\right),
\end{equation}
for $1 \leq j \leq \#I$. Here $L_{k,\,k^\prime}^{[j]}$ (depending on $\mathbf{s}^{[j]}$) denotes the series \begin{equation}
\mathcal{L}_{(u_{j_{k^\prime}},\,u_{j_{k^\prime+1}},\,\dots,\,u_{k-1}),\,(s_{j_{k^\prime}},\,s_{j_{k^\prime+1}},\,\dots,\,s_{j_{k-1}})}
\end{equation}
for $1\leq k^\prime< k\leq d+1$, under the notation $\mathbf{s}=(s_1,\,\dots,\,s_r)$, $1 \leq j_1<\cdots<j_d\leq r$, and $\mathbf{s}^{[j]}=(s_{j_1},\,\dots,\,s_{j_d})\in \operatorname{Sub}^{\prime}(\mathbf{s})$. Hence we further have $\Gamma_{\Psi(i,\,m)} \subset G_{i,\,m}$ by the definition (see equation \eqref{GammaPsi}). As Theorem \ref{ThmPapanikolas} shows $\Gamma_{\Psi(i,\,m)} \simeq \Gamma_{M(i,\,m)}$, the $t$-motivic Galois group $\Gamma_{M(i,\,m)}$ is isomorphic to a closed subscheme of $G_{i,\,m}$.

\end{proof}

The dimension of the variety $G_{i,\,m}$ is \begin{equation}
        n+1+(m+1)i+m(\#\operatorname{Sub}^\prime(\mathbf{s})-i).\label{dimGim}
    \end{equation} 
Hence we have 
\begin{equation}\label{dimGimsucsessor}
    \dim G_{i,\,m}=\dim G_{i,\,m}^\prime+1
\end{equation}

We define the subvariety $V_{i,\,m} \subset G_{i,\,m}$ to be the preimage of the identity element via the morphisms $\Pi_{i,\,m}:G_{i,\,m}\twoheadrightarrow G_{i,\,m}^\prime$ (see Definition \ref{explicitform}).
This is a closed subgroup of $\operatorname{GL}_{N_{n,\,m,\,j}}$ and isomorphic to the additive group $\mathbb{G}_a$ as $V_{i,\,m}$ is the algebraic subset consisting of the matrices of the form
\begin{small}
    \begin{equation} \label{vectorpart}
        I_{n+1}\oplus \bigoplus_{1\leq j <i} I_{(m+1)(\dep \mathbf{s}^{[j]}+1)}\oplus
        \begin{pmatrix}
        1&0 & \cdots& 0 \\
        0& \ddots&&\\
        0&&\ddots&\\
        x_{m,\,\mathbf{s}_j}& 0&\cdots&1
    \end{pmatrix}\oplus \bigoplus_{i<j\leq \#\operatorname{Sub}^\prime(\mathbf{s})}I_{(m)(\dep \mathbf{s}^{[j]}+1)}.
    \end{equation}
    \end{small}

\begin{Lemma}\label{LemmaVinclusion}
    The algebraic group $\mathcal{V}_{i,\,m}$ in Definition \ref{mathcalVim} is a closed subgroup of $V_{i,\,m}$ for $(i,\,m)\neq (1,\,0)$.
\end{Lemma}
\begin{proof}
Lemma \ref{LemmaTannakianProj} shows the commutativity of the following diagram:
\begin{center}
\begin{tikzpicture}[auto]
\node (03) at (0, 3) {$\Gamma_{M(i,\,m)}$}; \node (33) at (3, 3) {$\Gamma_{M^\prime(i,\,m)}$}; 
\node (00) at (0, 0) {$G_{i,\,m}$}; \node (30) at (3, 0) {$G_{i,\,m}^\prime$,};

\draw[->>] (03) to node {$\varphi_{i,\,m}$}(33);
\draw[->>] (00) to node {$\Pi_{i,\,m}$}(30);

\draw[{Hooks[right]}->] (03) to node {$\iota$} (00);
\draw[{Hooks[right]}->] (33) to node {$\iota$} (30);

\end{tikzpicture}.

\end{center}
which shows that the algebraic group $\mathcal{V}_{i,\,m}$ is contained in $V_{i,\,m}\simeq \mathbb{G}_a$.
\end{proof}

\subsection{Determination of the $t$-motivic Galois groups}\label{sectionproof2}
Let us take an index $\mathbf{s}=(s_1,\,\dots,\,s_r)\in \mathbb{Z}_{>0}^r$ and polyomials $u_1,\,\dots,\,u_r \in \overline{K}[t]$ such that
\begin{equation}        ||u_i||_\infty<|\theta|_\infty^{\frac{s_i q}{q-1}}
    \end{equation}
hold for all $1 \leq i \leq r$.
We enumerate the set
\begin{equation}
\operatorname{Sub}^\prime(\mathbf{s}):= \{ (s_{i_1},\,s_{i_2},\,\dots,\,s_{i_d}) \mid 1\leq d \leq r,\, 1\leq i_1<\cdots<i_d\leq r \} 
\end{equation}
as $\mathbf{s}^{[1]},\,\mathbf{s}^{[2]},\,\dots,\,\mathbf{s}^{[\#\operatorname{Sub}^\prime(\mathbf{s})]}$ so that $\operatorname{dep}(\mathbf{s}^{[i]})\leq\operatorname{dep}(\mathbf{s}^{[i+1]})$ for all $1 \leq i \leq \#\operatorname{Sub}^\prime(\mathbf{s})-1$. 
We also choose a non-negative integer $n$.
Our aim in this subsection is to verify the equality
\begin{equation}\label{localintro}
    \Gamma_{M(i,\,m)}=G_{i,\,m}
\end{equation}
for any $1 \leq i \leq \#\operatorname{Sub}^\prime(\mathbf{s})$ and $0 \leq m \leq n$ in the case where we have $$\mathcal{L}_{(u_i),\,(s_i)}(t)|_{t=\theta}\neq 0$$ for $1 \leq i \leq r$ and the index $(s_1,\,\dots,\,s_r)\in \mathbb{Z}_{>0}^r$ is chosen so that $p \nmid s_i,\,(q-1) \nmid s_i$ for $1 \leq i \leq r$ and $s_1,\,\dots,\,s_r$ are distinct.
Equality \eqref{localintro} will contribute to the proof of Theorem \ref{MainB}, which is one of our main goals of this chapter.
As we have proved in Proposition \ref{PropMainBfirststep} that $\Gamma_{M(i,\,0)}=G_{i,\,0}$ for $0 \leq i \leq r$,  we next prove the desired equality for the case where at least one of the conditions $m\geq 1$ or $i \geq r+1$ holds.

The proof of Equation \eqref{localintro} will be done by induction on tuples $(i,\,m)$ where the order is given by
\begin{equation}
    (i_1,\,m_1) \geq (i_2 ,\,m_2) \text{ if ``$m_1> m_2$'' or ``$m_1=m_2$ and $i_1 \geq i_2$''}.
\end{equation}
We recall that the symbols $M^\prime(i,\,m)$ and $G_{i,\,m}^\prime$ are respectively defined in Equations \eqref{explainMprime} and \eqref{explainGprime1} for $(i,\,m) \neq (1,\,0)$. 
Note that what we have to show is that the equation $\Gamma_{M^\prime(i,\,m)}=G_{i,\,m}^\prime$ yields $\Gamma_{M(i,\,m)}=G_{i,\,m}$ if we have $m\geq 1$ or $i \geq r+1$ because of Proposition \ref{PropMainBfirststep}.
The steps of induction can be grouped to two arrays. The one array consists of steps with $r+1 \leq i \leq 2^r-1$, which are dealt with in Proposition \ref{Propcasehigherdepth} and correspond to Example  
\ref{simplestexample1}.
The other one consists of steps with $m \geq 1$ and $1 \leq i \leq r$. These steps are handled in Proposition \ref{Propcasedep1} and correspond to Example \ref{simplestexample2}.

\begin{Proposition}\label{Propcasehigherdepth}
    We take an index $\mathbf{s}=(s_1,\,\dots,\,s_r)\in \mathbb{Z}_{>0}^r$ such that $(q-1) \nmid s_i$ for $1 \leq i \leq r$ and $s_i /s_{j^\prime} \neq p^{\mathbb{Z}}$ for $1 \leq j <j^\prime \leq r$.  
    Take also polynomials $u_1,\,\dots,\,u_r \in \overline{K}[t]$ such that $\mathcal{L}_{(u_i),\,(s_i)}(t)|_{t=\theta}\neq 0$ and
    \begin{equation}        ||u_i||_\infty<|\theta|_\infty^{\frac{s_i q}{q-1}}
    \end{equation}
    for all $1 \leq i \leq r$. If we have $\Gamma_{M^\prime(i,\,m)}=G^\prime_{i,\,m}$ for some $r+1 \leq i \leq 2^r-1$ and $0 \leq m \leq n$, then the equality $\Gamma_{M(i,\,m)}=G_{i,\,m}$ also holds.
\end{Proposition}

\begin{proof}

We note that our arguments are similar to those in Example \ref{simplestexample1}.
We recall that we have an exact sequence
\begin{equation}\label{exactseqinduction}
    1 \rightarrow \mathcal{V}_{i,\,m} \rightarrow \Gamma_{M(i,\,m)} \xrightarrow{\varphi_{i.\,m}} \Gamma_{M^\prime(i,\,m)} \rightarrow 1,
\end{equation}
see Definitions \ref{mathcalVimmorph} and \ref{mathcalVim}.
Our first goal is to show that the dimension of $\mathcal{V}_{i,\,m}$ is $1$. Similarly to the arguments in Example \ref{simplestexample1}, we will achieve this by choosing appropriate matrices $R$ and $Q_b$ from $\Gamma_{M(i,\,m)}$ and calculating their commutator.

Based on the assumption $\Gamma_{M^\prime(i,\,m)}=G^\prime_{i,\,m}$, we can take $R^\prime \in \Gamma_{M^\prime(i,\,m)}( \overline{\mathbb{F}_q(t)})$ given by $a_0=1, a_1=\cdots=a_n=0$, $x_{m,\,(s_{i_2},\dots,\,s_{i_d})}=1$, and $x_{m^\prime,\,\mathbf{s}^{j}} =0$ for other $m^\prime$ and $1 \leq j \leq 2^r-1$ in the notation of Definition \ref{explicitform}. Then we take $R$ to be any element in the preimage  $\varphi_{i,\,m}^{-1}(R^\prime)\subset \Gamma_{M(i,\,m)}(\overline{\mathbb{F}_q(t)})$. 
We recall the notation given in  Subsection \ref{subsectionexplicit}.
 
Then the following value
\begin{equation}
    \underline{D}_{s+1}^{(m^\prime)}
\begin{pmatrix}
   a_0 &a_1& \cdots& a_{m^\prime}\\
   \vdots&\vdots &\vdots \\
   a_0 &a_1& \cdots& a_{m^\prime}\\
   x_{0,\,\mathbf{s}^{[j]}} & x_{1,\,\mathbf{s}^{[j]}} & \cdots& x_{m^\prime,\,\mathbf{s}^{[j]}}
   \end{pmatrix}
   :=\sum_{\substack{l_1,\,\dots,\,l_s,\,l^\prime \geq0\\l_1+\dots+l_s+l^\prime=m^\prime}}a_{l_1}\cdots a_{l_s}x_{l^\prime,\,\mathbf{s}^{[j]}}
\end{equation}
for $1 \leq j \leq 2^r-1$, $s \geq 0$, and $0\leq m^\prime \leq m$, equals $1$ only if $\mathbf{s}^{[j]}=(s_{i_2},\dots,\,s_{i_d})$ and $m^\prime =m$. In other cases, it is $0$.
We also obtain 
\begin{equation}
\underline{D}_{s}^{(m^\prime)}
\begin{pmatrix}
   a_0 &a_1& \cdots& a_{m^\prime}\\
   \vdots&\vdots &\vdots &\vdots\\
   a_0 &a_1& \cdots& a_{m^\prime}
   \end{pmatrix} =\begin{cases}
       1& \text{ if } m^\prime=0;\\
       0& \text{ if } 1\leq m^\prime \leq n
   \end{cases}\label{polynomi1or0}
\end{equation}
for each $0 \leq m^\prime \leq n$ and $s \geq 1$.

Therefore, $R$ is of the form
\begin{equation}
I_{n+1}\oplus
\bigoplus_{j<i}R_j
\oplus
\begin{pmatrix}
   1 & 0 &\cdots&\cdots&0\\
   0& 1&0&\cdots&0\\
   \vdots& \vdots&\ddots&&\vdots \\
   0& 0&\cdots&1&0\\
   y&1 &\cdots&0&1
\end{pmatrix}
\oplus \bigoplus_{j>i} I_{(\dep \mathbf{s}^{[j]}+1)m}
\end{equation}
where $R_j \in \operatorname{GL}_{(\dep \mathbf{s}^{[j]}+1)(m+1)}$, for $j<i$, are some lower triangle matrix and $y$ is some element in $\overline{\mathbb{F}_q(t)}$.
By Definitions \ref{Smj} and \ref{explicitform}, we can see that the block $R_j$ is equal to
\begin{equation}
\begin{pmatrix}
   1 & 0 &\cdots&\cdots&0\\
   0& 1&0&\cdots&0\\
   \vdots& \vdots&\ddots&&\vdots \\
   0& 0&\cdots&1&0\\
   0&1 &\cdots&0&1
\end{pmatrix}
\end{equation}
if $\mathbf{s}^{[j]}=(s^\prime,\,s_{i_2},\,\dots,\,s_{i_d})$ for some $s^\prime$, and is equal to
\begin{equation}
\begin{pmatrix}
   1 & 0 &\cdots&\cdots&0\\
   0& 1&0&\cdots&0\\
   \vdots& \vdots&\ddots&&\vdots \\
   1& 0&\cdots&1&0\\
   0&0 &\cdots&0&1
\end{pmatrix}
\end{equation}
if $\mathbf{s}^{[j]}=(s_{i_2},\,\dots,\,s_{i_d}, s^\prime)$ for some $s^\prime$. 
The block $R_j$ is of the form
\begin{equation}
    \begin{pmatrix}
   1 & 0 &\cdots&\cdots&0\\
   0& 1&0&\cdots&0\\
   \vdots& \vdots&\ddots&&\vdots \\
   0& 0&\cdots&1&0\\
   1&0 &\cdots&0&1
\end{pmatrix}
\end{equation}
if $\mathbf{s}^{[j]}=(s_{i_2},\,\dots,\,s_{i_d})$.
For other $j <i$, we have $(s_{i_2},\,\dots,\,s_{i_d}) \not \in \operatorname{Sub}(\mathbf{s}^{[j]})$ by the enumeration of $\in \operatorname{Sub}(\mathbf{s})$ and hence the block $R_j$ is equal to the identity matrix.

We take arbitrary $b \in \overline{\mathbb{F}_q(t)}$ and let $Q_b^\prime \in \Gamma_{M^\prime(i,\,m)}(\overline{\mathbb{F}_q(t)})$ be the matrix given by putting $a_0=1$, $x_{0,\,(s_{i_1})}=b$, and putting other variables to be $0$. Let us take $Q_b\in \varphi_{i,\,m}^{-1}(Q_b^\prime) \subset \Gamma_{M(i,\,m)}(\overline{\mathbb{F}_q(t)})$.
We have
Equation \eqref{polynomi1or0}
for each $0 \leq m^\prime \leq m$ and $s \geq 1$ again.
We consider the value of the polynomial
\begin{equation}
    \underline{D}_{s+1}^{(m^\prime)}
\begin{pmatrix}
   a_0 &a_1& \cdots& a_{m^\prime}\\
   \vdots&\vdots && \vdots \\
   a_0 &a_1& \cdots& a_{m^\prime}\\
   x_{0,\,\mathbf{s}^{[j]}} & x_{1,\,\mathbf{s}^{[j]}} & \cdots& x_{m^\prime,\,\mathbf{s}^{[j]}}
   \end{pmatrix}
   :=\sum_{\substack{l_1,\,\dots,\,l_s,\,l^\prime \geq0\\l_1+\dots+l_s+l^\prime=m^\prime}}a_{l_1}\cdots a_{l_s}x_{l^\prime,\,\mathbf{s}^{[j]}}
\end{equation}
for $1 \leq j \leq 2^r-1$, $s \geq 0$, and $0\leq m^\prime \leq m$. This is equal to $b$ if $\mathbf{s}^{[j]}=(s_{i_1})$ and $m^\prime=0$. In the other cases, the value of the polynomial is $0$.

Then we can see that $Q_b$ is of the form
\begin{equation}
I_{n+1}\oplus
\bigoplus_{j<i}Q_j
\oplus \mathcal{X}(z)\oplus\bigoplus_{j>i}Q_j^\prime
\end{equation}
for some $z \in \overline{\mathbb{F}_q(t)}$ and $Q_j \in \operatorname{GL}_{(\dep \mathbf{s}^{[j]}+1)(m+1)}$, $Q_j^\prime \in \operatorname{GL}_{(\dep \mathbf{s}^{[j]}+1)(m)}$, where we put $\mathcal{X}(z)$ to be the matrix
\begin{equation}
    \begin{pmatrix}
   \begin{bmatrix}
   1 & 0 &\cdots&0\\
   b& 1&\ddots&0\\
   \vdots& \vdots&\ddots&\vdots \\
   0& 0&\cdots&1\\
\end{bmatrix}  & O &\cdots&\cdots&O\\
   O& \begin{bmatrix}
   1 & 0 &\cdots&0\\
   b& 1&\ddots&0\\
   \vdots& \vdots&\ddots&\vdots \\
   0& 0&\cdots&1\\
\end{bmatrix} &O&\cdots&O\\
   \vdots& \vdots&\ddots&&\vdots \\
   O& O&\cdots&
   \begin{bmatrix}
   1 & 0 &\cdots&0\\
   b& 1&\ddots&0\\
   \vdots& \vdots&\ddots&\vdots \\
   0& 0&\cdots&1\\
\end{bmatrix} &O\\
   \begin{bmatrix}
   0 & 0 &\cdots&0\\
   0& 0&\ddots&0\\
   \vdots& \vdots&\ddots&\vdots \\
   z& 0&\cdots&0\\
\end{bmatrix} &O &\cdots&O&\begin{bmatrix}
   1 & 0 &\cdots&0\\
   b& 1&\ddots&0\\
   \vdots& \vdots&\ddots&\vdots \\
   0& 0&\cdots&1\\
\end{bmatrix} 
\end{pmatrix}.
\end{equation}
 For $j<i$ such that $\mathbf{s}^{[j]}$ is of the form $(s_{i_2},\,\dots,\,s_{i_d})$, $(s^\prime,\,s_{i_2},\,\dots,\,s_{i_d})$, or $(s_{i_2},\,\dots,\,s_{i_d}, s^\prime)$ for some $s^\prime$, we have $(s_{i_1}) \not \in \operatorname{Sub}^\prime(\mathbf{s}^{[j]})$, so $Q_j$ is the identity matrix.

Similarly to the calculation in Equation \eqref{commutator1}, we can hence show that the commutator 
\begin{equation}
    RQ_bR^{-1}Q_b^{-1}\in   \Gamma_{M(i,\,m)}(\overline{\mathbb{F}_q(t)})
\end{equation}
is written as follows:
\begin{equation}
I_{n+1}\oplus \bigoplus_{j<i}I_{(\dep \mathbf{s}^{[j]}+1)(m+1)} \oplus
\begin{pmatrix}
   1 & 0 &\cdots&\cdots&0\\
   0 & 1&0&\cdots&0\\
   0 & 0& 1&\ddots&\vdots\\
   \vdots&\vdots& \ddots&\ddots&0 \\
   b& 0&\cdots &0&1\\
\end{pmatrix}
\oplus \bigoplus_{j>i} I_{(\dep \mathbf{s}^{[j]}+1)m},
\end{equation}
from which we obtain $RQ_bR^{-1}Q_b^{-1} \in  \mathcal{V}_{i,\,m}$.
As $b$ is an arbitrarily chosen element of $\overline{\mathbb{F}_q(t)}$, we notice that the inclusion $\mathcal{V}_{i,\,m} \subset V_{i,\,m}$ in Lemma \ref{LemmaVinclusion} is an isomorphism and hence we have $\dim \mathcal{V}_{i,\,m}=1$.

By Exact sequence \eqref{exactseqinduction}, we have
\begin{equation}
    \dim \Gamma_{M(i,\,m)}=\dim \Gamma_{M^\prime(i,\,m)} + \dim \mathcal{V}_{i,\,m}=\dim G_{i,\,m}^\prime+1=\dim G_{i,\,m},
\end{equation}
see Equation \eqref{dimGimsucsessor} for the last equality. Therefore, we can conclude that $\Gamma_{M(i,\,m)} =G_{i,\,m}$ since $G_{i,\,m}$ is smooth and irreducible.
\end{proof}

Using Proposition \ref{Propcasehigherdepth}, we can calculate the structures of certain $t$-motivic Galois groups as follows. The theorem below corresponds to an algebraic independence (Theorem \ref{MainBpartial}) via the theorem of Papanikolas.
\begin{Theorem}\label{ThmBGrouppartial}
    Take an index $\mathbf{s}=(s_1,\,\dots,\,s_r)\in \mathbb{Z}_{>0}^r$ such that $(q-1) \nmid s_i$ for $1 \leq i \leq r$ and $s_i/s_j \not \in p^{\mathbb{Z}}$ for $1 \leq i< j \leq r$. 
    For any $u_1,\,\dots,\,u_r \in \overline{K}[t]$ chosen so that $\mathcal{L}_{(u_i),\,(s_i)}(t)|_{t=\theta}\neq 0$ and
    \begin{equation}        ||u_i||_\infty<|\theta|_\infty^{\frac{s_i q}{q-1}}
    \end{equation}
    for each $1 \leq i \leq r$, we have 
    \begin{equation}\label{maineq}
        \Gamma_{\Psi(i,\,0)} = G_{i,\,0} 
    \end{equation}
    for $1 \leq i \leq 2^r-1$ where $\Psi(i,\,0)$ is the matrix given in Equation \eqref{Psiim} and $G_{i,\,0}$ is the algebraic variety introduced in Definition \ref{explicitform}.
\end{Theorem}

Propositions \ref{PropMainBfirststep} and \ref{Propcasehigherdepth} enable us to obtain Theorem \ref{ThmBGrouppartial} by induction on $i$. This yields algebraic independence of MZV's (which can be seen as $0$-th Taylor coefficients of Anderson-Thakur series) and Taylor coefficients of the power series $\Omega$ (Theorem \ref{MainBpartial}). We further prove the following proposition in order to deal with higher  Taylor coefficients of Anderson-Thakur series.
Recall that the set $\operatorname{Sub}^\prime (\mathbf{s})$ was enumerated so that we have $\operatorname{dep}(\mathbf{s}^{[i]}) \leq \operatorname{dep}(\mathbf{s}^{[i+1]})$ for $1 \leq \#\operatorname{Sub}^\prime (\mathbf{s})-1$. Hence $\operatorname{dep}(\mathbf{s}^{[i]})=1$ for $1 \leq i\ \leq r$ if $s_1,\,\dots,\,s_r$ are assumed to be pairwise distinct.

\begin{Proposition}\label{Propcasedep1}
    Take an index $\mathbf{s}=(s_1,\,\dots,\,s_r)\in \mathbb{Z}_{>0}^r$ such that $(q-1) \nmid s_i$ for $1 \leq i \leq r$ and $s_1,\,\dots,\,s_r$ are distinct positive integers.  
    Also take $u_1,\,\dots,\,u_r \in \overline{K}[t]$ such that $\mathcal{L}_{(u_i),\,(s_i)}(t)|_{t=\theta}\neq 0$ and
    \begin{equation}        ||u_i||_\infty<|\theta|_\infty^{\frac{s_i q}{q-1}}
    \end{equation}
    for all $1 \leq i \leq r$. 
    If $\mathbf{s}^{[i]}:=(s_i^\prime)$ with positive integer $s_i^\prime$ not divisible by $p$, then the equality $\Gamma_{M^\prime(i,\,m)}=G^\prime_{i,\,m}$ deduce $\Gamma_{M(i,\,m)}=G_{i,\,m}$ for $m \geq 1$ and $1 \leq i \leq r$.
\end{Proposition}
\begin{proof}
We note that our arguments are similar to those in Example \ref{simplestexample2}.
Also in this case, we have an exact sequence \eqref{exactseqinduction}.
Similarly to the proof of Proposition \ref{Propcasehigherdepth}, our primary goal is to obtain $ \dim \mathcal{V}_{i,\,m}=1$.
Our strategy here to achieve this is also similar to that in the proof of Proposition \ref{Propcasehigherdepth}.
That is, we choose appropriate matrices $R$ and $Q_b$ from $\Gamma_{M(i,\,m)}$ and calculating their commutator, likely to the discussions in Example \ref{simplestexample2}.

We choose arbitrary $b \in \overline{\mathbb{F}_q(t)}$. The assumption $\Gamma_{M^\prime(i,\,m)}=G^\prime_{i,\,m}$, enables us to pick $Q_b^\prime \in \Gamma_{M^\prime(i,\,m)}(\overline{\mathbb{F}_q(t)})$  given by $a_0=1$, $a_m=b$, and  $a_1=\cdots=a_{m-1}=a_{m+1}=\cdots=a_n=0$, and putting $x_{m^\prime,\,\mathbf{s}^{[j]}}=0$ for each $m^\prime$ and $j$ in the notation in Definition \ref{explicitform}. We then take any preimage $Q_b \in \varphi_{i,\,m}^{-1}(Q_b^\prime) \subset \Gamma_{M(i,\,m)}(\overline{\mathbb{F}_q(t)})$ under $\varphi_{i,\,m}$, which we recall that it is surjective.

Because of the chosen $a_{0},\ldots,a_{n}$ above, we obtain
\begin{align}
\underline{D}_{s}^{(m^\prime)}
\begin{pmatrix}
   a_0 &a_1& \cdots& a_{m^\prime}\\
   \vdots&\vdots &\vdots &\vdots\\
   a_0 &a_1& \cdots& a_{m^\prime}
   \end{pmatrix} &:=\sum_{\substack{l_1,\,\dots,\,l_s \geq0\\l_1+\cdots+l_s=m^\prime}}a_{l_1}\cdots a_{l_s}\\
   &=\begin{cases}
       1& \text{ if } m^\prime=0;\\
       0& \text{ if } 1\leq m^\prime \leq n \text{ and }m^\prime \neq m\\
       sb& \text{ if } m^\prime = m
   \end{cases}
\end{align}
see Subsection \ref{subsectionexplicit} for the notation. We note further that
\begin{equation}
    \underline{D}_{s+1}^{(m^\prime)}
\begin{pmatrix}
   a_0 &a_1& \cdots& a_{m^\prime}\\
   \vdots&\vdots &\vdots \\
   a_0 &a_1& \cdots& a_{m^\prime}\\
   x_{0,\,\mathbf{s}^{[j]}} & x_{1,\,\mathbf{s}^{[j]}} & \cdots& x_{m^\prime,\,\mathbf{s}^{[j]}}
   \end{pmatrix}
   :=\sum_{\substack{l_1,\,\dots,\,l_s,\,l^\prime \geq0\\l_1+\dots+l_s+l^\prime=m^\prime}}a_{l_1}\cdots a_{l_s}x_{l^\prime,\,\mathbf{s}^{[j]}}=0
\end{equation}
for $1 \leq j \leq 2^r-1$, $s \geq 0$, and $0\leq m^\prime \leq m$.
Hence, recalling Definitions \ref{Smj} and \ref{explicitform}, we can see that the matrix $Q_b$ is of the form
\begin{align}

& \left(
I_{n+1}+
b\begin{pmatrix}
    0& \cdots & & &\\
    1& 0 & \cdots & &  \\
    0 &1 & \ddots & &\\
    \vdots & & \ddots &&\\
    0& \cdots &\cdots & 1&0
\end{pmatrix}^{m}

\right)

\oplus
\bigoplus_{1\leq j<i}\begin{pmatrix}
   1 & 0 &\cdots&\cdots&0\\
   0& 1&0&\cdots&0\\
   \vdots& \vdots&\ddots&&\vdots \\
   s_j^\prime b& 0&\cdots&1&0\\
   0&0 &\cdots&0&1
\end{pmatrix}\\
& \quad \quad \oplus
\begin{pmatrix}
   1 & 0 &\cdots&\cdots&0\\
   0& 1&0&\cdots&0\\
   \vdots& \vdots&\ddots&&\vdots \\
   s_i^\prime b& 0&\cdots&1&0\\
   y&0 &\cdots&0&1
\end{pmatrix}
\oplus \bigoplus_{i<j \leq 2^r-1} I_{(\dep \mathbf{s}^{[j]}+1)m}
\end{align}
for some $y \in \overline{\mathbb{F}_q(t)}$,
here we put $\mathbf{s}^{[j]}=(s_j^\prime)$ for $1 \leq j \leq i$.
We note that the $s_j^\prime $ occurring in some entry of the matrix above is naturally viewed as $s_j^\prime $ modulo $p$ in $\mathbb{F}_p$.

We recall that we have an assumption $\Gamma_{M^\prime(i,\,m)}=G^\prime_{i,\,m}$, which enables us to take the matrix $R^\prime \in \Gamma_{M^\prime(i,\,m)}(\overline{\mathbb{F}_q(t)})$ given by putting $a_0=1, a_1=\cdots= a_n=0$, $x_{0,\,\mathbf{s}^{[i]}}=1$ and putting $x_{m^\prime,\,\mathbf{s}^{[j]}}=0$ for all other $m^\prime,\, j$, and we further take any preimage $R\in \varphi_{i,\,m}^{-1} (R^\prime)\subset \Gamma_{M(i,\,m)}(\overline{\mathbb{F}_q(t)})$.
Since $a_0=1$ and $a_{1}=\cdots=a_{n}=0$, we have  Equation \eqref{polynomi1or0} for each $0 \leq m^\prime \leq n$ and $s \geq 1$ also in this case.

Recall the following notation
\begin{equation}
    \underline{D}_{s+1}^{(m^\prime)}
\begin{pmatrix}
   a_0 &a_1& \cdots& a_{m^\prime}\\
   \vdots&\vdots &\vdots \\
   a_0 &a_1& \cdots& a_{m^\prime}\\
   x_{0,\,\mathbf{s}^{[j]}} & x_{1,\,\mathbf{s}^{[j]}} & \cdots& x_{m^\prime,\,\mathbf{s}^{[j]}}
   \end{pmatrix}
   :=\sum_{\substack{l_1,\,\dots,\,l_s,\,l^\prime \geq0\\l_1+\dots+l_s+l^\prime=m^\prime}}a_{l_1}\cdots a_{l_s}x_{l^\prime,\,\mathbf{s}^{[j]}}
\end{equation}
for $1 \leq j \leq 2^r-1$, $s \geq 0$, and $0\leq m^\prime \leq m$.
This value equals $1$ if $m^\prime=0$ and $j=i$, or $0$ otherwise.
Hence, recalling Definitions \ref{Smj} and \ref{explicitform}, we notice that the matrix $R$ is of the form:

\begin{align}
I_{n+1}\oplus \bigoplus_{j<i}I_{(\dep \mathbf{s}^{[j]}+1)(m+1)} \oplus
\begin{pmatrix}
   1 & 0 &\cdots&&&\cdots&0\\
   1& 1&0&\cdots&&\cdots&0\\
   0& 0& 1&0&\cdots&\cdots&0\\
   0& 0& 1&1&0&\cdots &0\\
   \vdots& \vdots&\ddots&&\ddots&& \vdots \\
   0& 0&\cdots&&\cdots&1&0\\
   z&0 &\cdots&&\cdots&1&1
\end{pmatrix}
\oplus \bigoplus_{j>i}R_j^\prime
\end{align}
where $z \in \overline{\mathbb{F}_q(t)}$ and $R_j^\prime \in \operatorname{GL}_{(\dep \mathbf{s}^{[j]}+1)(m)}$ are some lower triangle matrices.

Considering similarly as in the caluculation in Equation \eqref{commutator2}, we notice that the commutator $RQ_bR^{-1}Q_b^{-1} \in \Gamma_{M(i,\,m)}(\overline{\mathbb{F}_q(t)})$ is written as follows:
\begin{equation}
I_{n+1}\oplus \bigoplus_{j<i}I_{(\dep \mathbf{s}^{[j]}+1)(m+1)} \oplus
\begin{pmatrix}
   1 & 0 &\cdots&0\\
   0& 1&\ddots&0\\
   \vdots& \vdots&\ddots&\vdots \\
   s_i^\prime b& 0&\cdots&1\\
\end{pmatrix}
\oplus \bigoplus_{j>i} I_{(\dep \mathbf{s}^{[j]}+1)m},
\end{equation}
hence $RQ_bR^{-1}Q_b^{-1} \in \mathcal{V}_{i,\,m}(\overline{\mathbb{F}_q(t)}$.

As $b$ is an arbitrarily chosen element of $\overline{\mathbb{F}_q(t)}$ and we assumed that $p \nmid s_i^\prime$, we notice that the inclusion $\mathcal{V}_{i,\,m} \subset V_{i,\,m}\simeq \mathbb{G}_a$ given in Lemma \ref{LemmaVinclusion} is an isomorphism and hence we have $\dim \mathcal{V}_{i,\,m}=1$.

By Equation \eqref{dimGimsucsessor} and Exact sequence \eqref{exactseqinduction}, we have
\begin{equation}
    \dim \Gamma_{M(i,\,m)}=\dim \Gamma_{M^\prime(i,\,m)} + \dim \mathcal{V}_{i,\,m}=\dim G_{i,\,m}.
\end{equation}
Hence we have $\Gamma_{M(i,\,m)} =G_{i,\,m} $ as $G_{i,\,m}$ is smooth and irreducible.
\end{proof}

Now we are ready to obtain the following conclusion on $t$-motivic Galois groups. Via the theory of Papanikolas (Theorem \ref{ThmPapanikolas}), the following result on $t$-motivic Galois group gives us the algebraic independence result on Taylor coefficients of the power series $\Omega$ and those of Anderson-Thakur series (Theorem \ref{MainB}).

\begin{Theorem}\label{MainGroupeq}
    Take an index $\mathbf{s}=(s_1,\,\dots,\,s_r)\in \mathbb{Z}_{>0}^r$ such that $p \nmid s_i,\,(q-1) \nmid s_i$ for $1 \leq i \leq r$ and $s_1,\,\dots,\,s_r$ are distinct positive integers. 
    For any $u_1,\,\dots,\,u_r \in \overline{K}[t]$ chosen so that $\mathcal{L}_{(u_i),\,(s_i)}(t)|_{t=\theta}\neq 0$ and
    \begin{equation}        ||u_i||_\infty<|\theta|_\infty^{\frac{s_i q}{q-1}}
    \end{equation}
    for each $1 \leq i \leq r$, we have 
    \begin{equation}\label{maineqpartial}
        \Gamma_{M(i,\,m)} = G_{i,\,m} 
    \end{equation}
    for $1 \leq i \leq 2^r-1$ and $0 \leq m \leq n$ where $M(i,\,m)$ is the pre-$t$-motive given in Definition \ref{mainpretmotive} and $G_{i,\,m}$ is the algebraic variety introduced in Definition \ref{explicitform}.
\end{Theorem}
\begin{proof}
    We prove $\Gamma_{M(i,\,m)}=G_{i,\,m}$ by induction on tuple $(i,\,m)$ where the order is given by
\begin{equation}
    (i_1,\,m_1) \geq (i_2 ,\,m_2) \text{ if ``$m_1> m_2$'' or ``$m_1=m_2$ and $i_1 \geq i_2$''}.
\end{equation}
In the cases where $m=0$ and $1 \leq i \leq r$, the equality $\Gamma_{M(i,\,m)}=G_{i,\,m}$ is proven in Proposition \ref{PropMainBfirststep}. 
Further, Propositions \ref{Propcasehigherdepth} and \ref{Propcasedep1} enable us to accomplish the proof of the assertion by the induction.   
\end{proof}

\subsection{Transcendence result and some remarks}

In this section, we deduce our main result (Theorems \ref{MainBpartial} and \ref{MainB}) from the calculation in the previous section (Theorems \ref{ThmBGrouppartial} and \ref{MainGroupeq}) of $t$-motivic Galois groups.
We write $\mathcal{L}_{\mathbf{u}^\prime,\,\mathbf{s}^\prime}$ for the $t$-motivic CMPL as in Section \ref{subsectionATseries} for each tuple $\mathbf{u}^\prime=(u_1,\,\dots,\,u_d) \in \overline{K}[t]$ and $\mathbf{s}^\prime=(s_1,\,\dots,\,s_d) \in \mathbb{Z}_{\geq 1}^r$ ($r \geq 1)$.

\begin{Theorem}\label{MainB}

    Take an index $\mathbf{s}=(s_1,\,\dots,\,s_r) \in \mathbb{Z}_{\geq 1}^r$ and assume that $p \nmid s_i,\,(q-1) \nmid s_i$ for $1\leq i\leq r$ and $s_1,\,\dots,\,s_r$ are distinct. 
    
    \begin{enumerate}
    \item 
        Let $u_1,\,\dots,\,u_r$ be elements in $\overline{K}[t]$ such that $\mathcal{L}_{(u_i),\,(s_i)}(t)|_{t=\theta}\neq 0$ and $||u_i||_\infty<|\theta|_\infty^{\frac{s_i q}{q-1}}$ for all $1\leq i\leq r$.
        If we consider the Taylor expansions
            \begin{equation}  
            \Omega(t)=\sum_{n=0}^{\infty} \alpha_n (t-\theta)^n
    \end{equation}
    of the power series $\Omega$ (see Example \ref{Omega} for the definition) and 
    \begin{equation}
        \mathcal{L}_{(u_{i_1},\,\dots,\,u_{i_d}),(s_{i_1},\,\dots,\,s_{i_d})}(t)=\sum_{n=0}^{\infty} \alpha_{(s_{i_1},\,\dots,\,s_{i_d}),\,n}(t-\theta)^n
    \end{equation}
    of $t$-motivic CMPL for each $(s_{i_1},\,\dots,\,s_{i_d}) \in \operatorname{Sub}^\prime(\mathbf{s})$,
    then the field generated by the set
    \begin{equation}
        \{ \alpha_{n^\prime},\, \alpha_{\mathbf{s}^\prime,\,n^\prime} \mid 0\leq n^\prime  \leq n,\,\mathbf{s}^\prime \in \operatorname{Sub}^\prime(\mathbf{s})  \}
    \end{equation}
    over $\overline{K}$ has transcendental degree $(n+1)2^r$ for each $n\geq0$.
    \item 
        Considering the Taylor expansions
\begin{equation}
     \zeta_A^{\mathrm{AT}}(\mathbf{s})=\sum \beta_{\mathbf{s},\,n}(t-\theta)^n
\end{equation}
of Anderson-Thakur series (Definition \ref{DefATseries}), we can see that the field extension over $\overline{K}$ generated by
    \begin{equation}
        \{ \alpha_{n^\prime} ,\, \beta_{\mathbf{s}^\prime,\,n^\prime} \mid 0\leq n^\prime  \leq n,\,\mathbf{s}^\prime \in \operatorname{Sub}^\prime(K) \}
    \end{equation}
    has transcendental degree $(n+1)2^r$ over $\overline{K}$.
    \end{enumerate}
\end{Theorem}
\begin{proof}

    Proposition \ref{PropDimTrans} and Thoerem \ref{MainGroupeq} show that the transcendence degree is equal to the dimension of the algebraic variety $G_{2^r-1,\,n}$. The dimension is equal to $(n+1)\cdot \left(\# \operatorname{Sub}^\prime(\mathbf{s})+1 \right)$ by the construction of the variety. As positive integers $s_1,\,\dots,\,s_r$ are assumed to be pairwise distinct, the set $\operatorname{Sub}^\prime(\mathbf{s})$ has exact $2^r-1$ elements, hence we have (1).
    We can deduce the assertion (2) from (1) by considering the case where $u_i$ is chosen to be the Anderson-Thakur polynomial $H_{s_i-1}$ for each $1 \leq i \leq r$.   
\end{proof}

We can also consider the same problem for $v$-adic multiple zeta values defined by Chang and Mishiba~(\cite{Chang2019a}) and Taylor coefficients of their deformations given by Chen in his doctoral thesis \cite{ChenThesis} in the case where the degree of the finite place $v$ is $1$. We hope to work on the project in the near future.

If we focus on MZV's and Taylor coefficients of the power series $\Omega$, we can relax the assumption on the index $\mathbf{s}=(s_1,\,\dots,\,s_r)$. 
In order to consider only the $0$-the Taylor coefficients of Anderson-Thakur series, it is enough to assume $s_i/s_j \not \in p^{\mathbb{Z}}$ for all $1 \leq i <j \leq r$ and positive integers $s_1,\,\dots,\,s_r$ can be a multiple of $p$:
\begin{Theorem}\label{MainBpartial}
    Take an index $\mathbf{s}=(s_1,\,\dots,\,s_r) \in \mathbb{Z}_{\geq 1}^r$  and assume that $(q-1) \nmid s_i$ for $1\leq i\leq r$. The ratio $s_i/s_j$ are assumed not to be integer power of $p$ for each $1 \leq i <j \leq r$. Let us consider the following Taylor expansion of the power series $\Omega$ (see Example \ref{Omega}):
\begin{equation}
    \Omega=\sum \alpha_n (t-\theta)^n.
\end{equation}

    \begin{enumerate}
    \item 
        If we choose polynomials $u_1,\,\dots,\,u_r$ in $\overline{K}[t]$ such that $\mathcal{L}_{(u_i),\,(s_i)}(t)|_{t=\theta}\neq 0$ and $||u_i||_\infty<|\theta|_\infty^{\frac{s_i q}{q-1}}$ for all $1\leq i\leq r$,
    then the field generated by the set
    \begin{equation}
        \{ \alpha_{n^\prime},\, \mathcal{L}_{\mathbf{u}^{[i]},\,\mathbf{s}^{[i]}}|_{t=\theta} \mid 0\leq n^\prime  \leq n,  1 \leq i \leq 2^r-1\}
    \end{equation}
    over $\overline{K}$ has transcendental degree $n+2^r$ for each $n\geq0$. Here we write $\mathbf{u}^{[i]}:=(u_{i_1},\,\dots,\,u_{i_r})$ when $\mathbf{s}^{[i]}:=(s_{i_1},\,\dots,\,s_{i_r})$.
    \item 
        The field extension over $\overline{K}$ generated by
    \begin{equation}
        \{ \alpha_{n^\prime} ,\, \zeta_A(\mathbf{s}^{[i]} )\mid 0\leq n^\prime  \leq n,\,1 \leq i \leq 2^r-1 \}
    \end{equation}
    has transcendental degree $n+2^r$.
    \end{enumerate}
\end{Theorem}

We can obtain this theorem by the similar arguments as the proof of Theorem \ref{MainB} by using Proposition \ref{PropDimTrans} and Theorem \ref{ThmBGrouppartial}.

\begin{Remark}

\begin{enumerate}
    \item 
 In order to obtain the algebraic independence of higher Taylor coefficients of the Anderson-Thakur series, the condition $p \nmid s_i$ for all $i$ is indispensable. For example, if we consider the case $r=1$ and $s_1=ps^\prime$ for some $s^\prime \in \mathbb{Z}_{\geq 1}$, then the relation $\zeta_A(s^\prime p)=\zeta_A(s^\prime)^p$, which is known as $p$-th power relation and can be confirmed easily, lifts to the equation
\begin{equation}
    \zeta_A^{\mathrm{AT}}(p s^\prime)=\gamma \zeta_A^{\mathrm{AT}}(s^\prime)^p
\end{equation}
of Anderson Thakur series with some explicit $\gamma \in \mathbb{F}_q(t)^\times$ (see \cite[Lemma 3.2]{Mishiba2015}). Taking hyperderivatives of both sides, we have the following relation:
\begin{align}
    \partial\zeta_A^{\mathrm{AT}}(ps^\prime)&=\zeta_A^{\mathrm{AT}}(s^\prime)^p \partial \gamma+p\gamma \zeta_A^{\mathrm{AT}}(s^\prime)^{p-1}\partial \zeta_A^{\mathrm{AT}}(s^\prime)\\
    &=\zeta_A^{\mathrm{AT}}(s^\prime)^p \partial \gamma=\frac{\zeta_A^{\mathrm{AT}}(ps^\prime)\partial \gamma}{\gamma}.
\end{align}
Substituting $t$ by $\theta$, we obtain a non-trivial $K$-linear relation among the $0$-the and the first Taylor coefficients of $\zeta_A^{\mathrm{AT}}(ps^\prime)$.
\item
Mishiba (\cite{Mishiba2015} and \cite{Mishiba2017}) obtained the equality
\begin{equation}
    \trdeg_{\overline{K}}\overline{K}\left( \tilde{\pi},\,\zeta_A(\mathbf{s}^\prime) \mid \mathbf{s}^\prime \in \operatorname{Sub}^{\prime}(\mathbf{s})\right)=2^r
\end{equation}
under the same assumptions as those in Theorem \ref{MainBpartial}.
In a private discussion with Mishiba, he told the author that he had obtained algebraic independence of the set
    \begin{equation}
        \{\alpha_0,\,\dots,\,\alpha_n\} \cup \{\alpha_{\mathbf{s}^\prime,\,0} \mid \mathbf{s}^\prime\in \operatorname{Sub}^\prime(\mathbf{s}) \}\cup \{\alpha_{(s_j),\,1} \mid 1\leq j \leq r\}
    \end{equation}
    under the same notation and assumption as those of Theorem \ref{MainB}. Our results are regarded as a generalizations of his results.
\item
Based on the work of Anderson and Thakur in \cite{Anderson2009}, Chang and Mishiba (\cite{Chang2021}) constructed a $t$-module and its special point for each multiple zeta value and showed that the multiple zeta value in question occurs in an entry of the logarithm of the $t$-module at the special point. Chang, Green, and Mishiba \cite{Chang2021a} showed that the other entries of the logarithm at the point mentioned above can be written in terms of Taylor coefficients of the Anderson-Thakur series and $t$-motivic CMPLs.
\end{enumerate}
\end{Remark}

\noindent \textbf{Acknowledgement}.
The author would like to thank his advisor H.~Furusho for his profound instruction and continuous encouragements. 
This work was partially supported by JSPS KAKENHI Grant Number JP23KJ1079.


\end{document}